\documentclass[reqno]{amsart}
\usepackage{amsfonts}
\usepackage{amsmath}
\usepackage[left=3.10cm,right=3.10cm]{geometry}

\newtheorem{theorem}{Theorem}
\theoremstyle{plain}

\newtheorem{corollary}{Corollary}

\newtheorem{definition}{Definition}

\newtheorem{lemma}{Lemma}

\newtheorem{proposition}{Proposition}
\newtheorem{remark}{Remark}

\DeclareMathOperator{\Div}{div}
 \numberwithin{equation}{section}
\input{tcilatex}

\begin{document}
\title[Almost periodic homogenization of \ SPDEs]{Homogenization of a
stochastic nonlinear reaction-diffusion equation with a large reaction term:
the almost periodic framework}
\author{Paul Andr\'{e} Razafimandimby}
\address{Department of Mathematics and Applied Mathematics, University of
Pretoria, Pretoria 0002, South Africa (P.A. Razafimandimby)}
\email{paulrazafi@gmail.com}
\author{Mamadou Sango}
\address{Department of Mathematics and Applied Mathematics, University of
Pretoria, Pretoria 0002, South Africa (M. Sango)}
\email{mamadou.sango@up.ac.za}
\author{Jean Louis Woukeng}
\address{Department of Mathematics and Computer Science, University of
Dschang, P.O. Box 67, Dschang, Cameroon (J.L. Woukeng)}
\email{jwoukeng@yahoo.fr}
\date{July, 2011}

\begin{abstract}
Homogenization of a stochastic nonlinear reaction-diffusion equation with a
large nonlinear term is considered. Under a general Besicovitch almost
periodicity assumption on the coefficients of the equation we prove that the
sequence of solutions of the said problem converges in probability towards
the solution of a rather different type of equation, namely, the stochastic
nonlinear convection-diffusion equation which we explicitly derive in terms
of appropriated functionals. We study some particular cases such as the
periodic framework, and many others. This is achieved under a suitable
generalized concept of $\Sigma $-convergence for stochastic processes.
\end{abstract}

\maketitle

\section{Introduction}

The homogenization theory is an important branch of the asymptotic analysis.
Since the pioneering work of Bensoussan et al. \cite{BLP} it has grown very
significantly, giving rise to several sub-branches such as the \textit{%
deterministic} homogenization theory and the \textit{random} homogenization
theory. Each of these sub-branches has been developed and deepened.
Regarding the deterministic homogenization theory, from the classical
periodic theory \cite{BLP} to the recent general deterministic ergodic
theory \cite{Casado, Hom1, CMP, NA}, many results have been reported and
continue to be published. We refer to some of these results \cite{Allaire,
Casado, Hom1, CMP, NA} relating to the deterministic homogenization of
deterministic partial differential equations in the periodic framework and
in the deterministic ergodic framework in general.

The random homogenization theory is divided into two major subgroups: the
homogenization of differential operators with random coefficients, and the
homogenization of stochastic partial differential equations. As far as the
first subgroup is concerned, so many results are also available so far; we
refer e.g. to \cite{Bensoussan3, Bourgeat1, Bourgeat2, Bourgeat3, Jikov,
Kozlov1, Kozlov2, PardouxPiat, Telega, VARADHAN1, WoukAA}.

In contrast with either the deterministic homogenization theory or the
homogenization of partial differential operators with random coefficients,
very few results are available in the setting of the homogenization of
stochastic partial differential equations (SPDEs). We cite for example \cite%
{Bensoussan1, Ichihara1, Ichihara2, Wang1, Wang2} in which are considered
the homogenization problems related to SPDEs with periodic coefficients
(only!). See also \cite{Sango} in which homogenization of a SPDE with
constant coefficients is considered. It should be noted that unfortunately
so far, no result in this area is available beyond the periodic setting.

Given the interest of SPDEs in modeling of physical phenomena, which are
besides not only simple random periodically perturbed phenomena, it is
important to think of a theory generalizing that of the homogenization of
SPDEs with periodic coefficients. This is one of the objectives of this work.

More precisely, we discuss the homogenization problem for the following
nonlinear SPDE 
\begin{equation}
\left\{ 
\begin{array}{l}
du_{\varepsilon }=\left( \Div\left( a\left( \frac{x}{\varepsilon },\frac{t}{%
\varepsilon ^{2}}\right) Du_{\varepsilon }\right) +\frac{1}{\varepsilon }%
g\left( \frac{x}{\varepsilon },\frac{t}{\varepsilon ^{2}},u_{\varepsilon
}\right) \right) dt+M\left( \frac{x}{\varepsilon },\frac{t}{\varepsilon ^{2}}%
,u_{\varepsilon }\right) dW\text{\ \ in }Q_{T} \\ 
u_{\varepsilon }=0\text{\ \ on }\partial Q\times (0,T) \\ 
u_{\varepsilon }(x,0)=u^{0}(x)\text{\ \ in }Q%
\end{array}%
\right.  \label{1.1}
\end{equation}%
in the almost periodic environment, where $Q_{T}=Q\times (0,T)$, $Q$ being a
Lipschitz domain in $\mathbb{R}^{N}$ with smooth boundary $\partial Q$, $T$
is a positive real number and $W$ is a $m$-dimensional standard Wiener
process defined on a given probability space $(\Omega ,\mathcal{F},\mathbb{P}%
)$. The choice of the above problem lies in its application in engineering
(see for example \cite{AllPiat1, AllPiat2, Mik} in the deterministic
setting, and \cite{pardoux} in the stochastic framework, for more details).
In fact, as in \cite{AllPiat1}, the unknown $u_{\varepsilon }$ may be viewed
as the concentration of some chemical species diffusing in a porous medium
of constant porosity, with diffusivity $a(y,\tau )$ and reacting with
background medium through the nonlinear term $g(y,\tau ,u)$ under the
influence of a random external source $M(y,\tau ,u)$. The motivation of this
choice is several fold. Firstly, we start from a SPDE of reaction-diffusion
type, and we end up, after the passage to the limit, with a SPDE of a
convection-diffusion type; this is because of the large reaction's term $%
\frac{1}{\varepsilon }g(x/\varepsilon ,t/\varepsilon ^{2},u_{\varepsilon })$
which satisfies some kind of centering condition; see Section 4 for details.
Secondly, the order of the microscopic time scale here is twice that of the
microscopic spatial scale. This leads after the passage to the limit, to a
rather complicated so-called cell problem, which is besides, a deterministic
parabolic type equation, the random variable behaving in the latter equation
just like a parameter. Such a problem is difficult to deal with as, in our
situation, it involves a microscopic time derivative derived from the
semigroup theory, which is not easy to handle. Thirdly, in order to solve
the homogenization problem under consideration, we introduce a suitable type
of convergence which takes into account both deterministic and random
behavior of the data of the original problem. This method is formally
justified by the theory of Wiener chaos polynomials \cite{Cameron, Wiener}.
In fact, following \cite{Cameron} (see also \cite{Wiener}), any sequence of
stochastic processes $u^{\varepsilon }(x,t,\omega )\in L^{2}(Q\times
(0,T)\times \Omega )$ expresses as follows: 
\begin{equation*}
u^{\varepsilon }(x,t,\omega )=\sum_{j=1}^{\infty }u_{j}^{\varepsilon
}(x,t)\Phi _{j}(\omega )
\end{equation*}%
where the functions $\Phi _{j}$ are the generalized Hermite polynomials,
known as the Wiener-chaos polynomials. The above decomposition clearly
motivates the definition of the concept of convergence used in this work;
see Section 3 for further details. Finally, the periodicity assumption on
the coefficients is here replaced by the almost periodicity assumption.
Accordingly, it is the first time that an SPDE is homogenized beyond the
classical period framework, and our result is thus, new. It is also
important to note that in the deterministic, i.e. when $M=0$ in (\ref{1.1}),
the equivalent problem obtained has just been solved by Allaire and
Piatnitski \cite{AllPiat1} under the periodicity assumption on the
coefficients, but with a weight function on the derivative with respect to
time. Our result may therefore generalize to the almost periodic setting,
the one obtained by Allaire and Piatnitski in \cite{AllPiat1}.

The layout of the paper is as follows. In Section 2 we recall some useful
fact about almost periodicity that will be used in the next sections.
Section 3 deals with the concept of $\Sigma $-convergence for stochastic
processes. In Section 4, we state the problem to be studied. We proved there
a tightness result that will be used in the next section. We state and prove
homogenization results in Section 5. In particular we give in that section
the explicit form of the homogenization equation. Finally, in Section 6, we
give some applications of the result obtained in the previous section.

Unless otherwise specified, vector spaces throughout are assumed to be
complex vector spaces, and scalar functions are assumed to take complex
values. We shall always assume that the numerical space $\mathbb{R}^{m}$
(integer $m\geq 1$) and its open sets are each equipped with the Lebesgue
measure $dx=dx_{1}...dx_{m}$.

\section{Spaces of almost periodic functions}

The concept of almost periodic functions is well known in the literature. We
present in this section some basic facts about it, which will be used
throughout the paper. For a general presentation and an efficient treatment
of this concept, we refer to \cite{Bohr}, \cite{Besicovitch} and \cite%
{Levitan}.

Let $\mathcal{B}(\mathbb{R}^{N})$ denote the Banach algebra of bounded
continuous complex-valued functions on $\mathbb{R}^{N}$ endowed with the $%
\sup $ norm topology.

A function $u\in \mathcal{B}(\mathbb{R}^{N})$ is called a almost periodic
function if the set of all its translates $\{u(\cdot +a)\}_{a\in \mathbb{R}%
^{N}}$ is precompact in $\mathcal{B}(\mathbb{R}^{N})$. The set of all such
functions forms a closed subalgebra of $\mathcal{B}(\mathbb{R}^{N})$, which
we denote by $AP(\mathbb{R}^{N})$. From the above definition, it is an easy
matter to see that every element of $AP(\mathbb{R}^{N})$ is uniformly
continuous. It is classically known that the algebra $AP(\mathbb{R}^{N})$
enjoys the following properties:

\begin{itemize}
\item[(i)] $\overline{u}\in AP(\mathbb{R}^{N})$ whenever $u\in AP(\mathbb{R}%
^{N})$, where $\overline{u}$ stands for the complex conjugate of $u$;

\item[(ii)] $u(\cdot +a)\in AP(\mathbb{R}^{N})$ for any $u\in AP(\mathbb{R}%
^{N})$ and each $a\in \mathbb{R}^{N}$;

\item[(iii)] For each $u\in AP(\mathbb{R}^{N})$ the closed convex hull of $%
\{u(\cdot +a)\}_{a\in \mathbb{R}^{N}}$ in $\mathcal{B}(\mathbb{R}^{N})$
contains a unique complex constant $M(u)$ called the mean value of $u$, and
which satisfies the property that the sequence $(u^{\varepsilon
})_{\varepsilon >0}$ (where $u^{\varepsilon }(x)=u(x/\varepsilon )$, $x\in 
\mathbb{R}^{N}$) weakly $\ast $-converges in $L^{\infty }(\mathbb{R}^{N})$
to $M(u)$ as $\varepsilon \rightarrow 0$. $M(u)$ also satisfies the property 
\begin{equation*}
M(u)=\lim_{R\rightarrow +\infty }\frac{1}{(2R)^{N}}\int_{[-R,R]^{N}}u(y)dy.
\end{equation*}
\end{itemize}

As a result of (i)-(iii) above we get that $AP(\mathbb{R}^{N})$ is an
algebra with mean value on $\mathbb{R}^{N}$ \cite{Jikov}. The spectrum of $%
AP(\mathbb{R}^{N})$ (viewed as $\mathcal{C}^{\ast }$-algebra) is the Bohr
compactification of $\mathbb{R}^{N}$, denoted usually in the literature by $b%
\mathbb{R}^{N}$, and, in order to simplify the notation, we denote it here
by $\mathcal{K}$. Then, as it is classically known, $\mathcal{K}$ is a
compact topological\ Abelian group. We denote its Haar measure by $\beta $
(as in \cite{Hom1}). The following result is due to the Gelfand
representation theory of $\mathcal{C}^{\ast }$-algebras.

\begin{theorem}
\label{t2.1}There exists an isometric $\ast $-isomorphism $\mathcal{G}$ of $%
AP(\mathbb{R}^{N})$ onto $\mathcal{C}(\mathcal{K})$ such that every element
of $AP(\mathbb{R}^{N})$ is viewed as a restriction to $\mathbb{R}^{N}$ of a
unique element in $\mathcal{C}(\mathcal{K})$. Moreover the mean value $M$
defined on $AP(\mathbb{R}^{N})$ has an integral representation in terms of
the Haar measure $\beta $ as follows: 
\begin{equation*}
M(u)=\int_{\mathcal{K}}\mathcal{G}(u)d\beta \text{\ \ for all }u\in AP(%
\mathbb{R}^{N})\text{.}
\end{equation*}
\end{theorem}

The isometric $\ast $-isomorphism $\mathcal{G}$ of the above theorem is
referred to as the Gelfand transformation. The image $\mathcal{G}(u)$ of $u$
will very often be denoted by $\widehat{u}$.

For $m\in \mathbb{N}$ (the positive integers) we introduce the space $AP^{m}(%
\mathbb{R}^{N})=\{u\in AP(\mathbb{R}^{N}):D_{y}^{\alpha }u\in AP(\mathbb{R}%
^{N})$ for every $\alpha =(\alpha _{1},\ldots ,\alpha _{N})\in \mathbb{N}%
^{N} $ with $\left\vert \alpha \right\vert \leq m\}$, a Banach space with
the norm $\left\Vert \left\vert u\right\vert \right\Vert
_{m}=\sup_{\left\vert \alpha \right\vert \leq m}\sup_{y\in \mathbb{R}%
^{N}}\left\vert D_{y}^{\alpha }u\right\vert $, where $D_{y}^{\alpha }=\frac{%
\partial ^{\left\vert \alpha \right\vert }}{\partial y_{1}^{\alpha
_{1}}\ldots \partial y_{N}^{\alpha _{N}}}$. We also define the space $%
AP^{\infty }(\mathbb{R}^{N})=\cap _{m}AP^{m}(\mathbb{R}^{N})$, a Fr\'{e}chet
space with respect to the natural topology of projective limit, defined by
the increasing family of norms $\left\Vert \left\vert \cdot \right\vert
\right\Vert _{m}$ ($m\in \mathbb{N}$).

Next, let $B_{AP}^{p}(\mathbb{R}^{N})$ ($1\leq p<\infty $) denote the space
of Besicovitch almost periodic functions on $\mathbb{R}^{N}$, that is the
closure of $AP(\mathbb{R}^{N})$ with respect to the Besicovitch seminorm 
\begin{equation*}
\left\Vert u\right\Vert _{p}=\left( \underset{r\rightarrow +\infty }{\lim
\sup }\frac{1}{\left\vert B_{r}\right\vert }\int_{B_{r}}\left\vert
u(y)\right\vert ^{p}dy\right) ^{1/p}
\end{equation*}%
where $B_{r}$ is the open ball of $\mathbb{R}^{N}$ of radius $r$ centered at
the origin. It is known that $B_{AP}^{p}(\mathbb{R}^{N})$ is a complete
seminormed vector space verifying $B_{AP}^{q}(\mathbb{R}^{N})\subset
B_{AP}^{p}(\mathbb{R}^{N})$ for $1\leq p\leq q<\infty $. From this last
property one may naturally define the space $B_{AP}^{\infty }(\mathbb{R}%
^{N}) $ as follows: 
\begin{equation*}
B_{AP}^{\infty }(\mathbb{R}^{N})=\{f\in \cap _{1\leq p<\infty }B_{AP}^{p}(%
\mathbb{R}^{N}):\sup_{1\leq p<\infty }\left\Vert f\right\Vert _{p}<\infty \}%
\text{.}\;\;\;\;\;\;\;\;\;
\end{equation*}%
We endow $B_{AP}^{\infty }(\mathbb{R}^{N})$ with the seminorm $\left[ f%
\right] _{\infty }=\sup_{1\leq p<\infty }\left\Vert f\right\Vert _{p}$,
which makes it a complete seminormed space. We recall that the spaces $%
B_{AP}^{p}(\mathbb{R}^{N})$ ($1\leq p\leq \infty $) are not general Fr\'{e}%
chet spaces since they are not separated. The following properties are worth
noticing \cite{CMP, NA}:

\begin{itemize}
\item[(\textbf{1)}] The Gelfand transformation $\mathcal{G}:AP(\mathbb{R}%
^{N})\rightarrow \mathcal{C}(\mathcal{K})$ extends by continuity to a unique
continuous linear mapping, still denoted by $\mathcal{G}$, of $B_{AP}^{p}(%
\mathbb{R}^{N})$ into $L^{p}(\mathcal{K})$, which in turn induces an
isometric isomorphism $\mathcal{G}_{1}$, of $B_{AP}^{p}(\mathbb{R}^{N})/%
\mathcal{N}=\mathcal{B}_{AP}^{p}(\mathbb{R}^{N})$ onto $L^{p}(\mathcal{K})$
(where $\mathcal{N}=\{u\in B_{AP}^{p}(\mathbb{R}^{N}):\mathcal{G}(u)=0\}$).
Moreover if $u\in B_{AP}^{p}(\mathbb{R}^{N})\cap L^{\infty }(\mathbb{R}^{N})$
then $\mathcal{G}(u)\in L^{\infty }(\mathcal{K})$ and $\left\Vert \mathcal{G}%
(u)\right\Vert _{L^{\infty }(\mathcal{K})}\leq \left\Vert u\right\Vert
_{L^{\infty }(\mathbb{R}^{N})}$.

\item[(\textbf{2)}] The mean value $M$ viewed as defined on $AP(\mathbb{R}%
^{N})$, extends by continuity to a positive continuous linear form (still
denoted by $M$) on $B_{AP}^{p}(\mathbb{R}^{N})$ satisfying $M(u)=\int_{%
\mathcal{K}}\mathcal{G}(u)d\beta $ ($u\in B_{AP}^{p}(\mathbb{R}^{N})$).
Furthermore, $M(u(\cdot +a))=M(u)$ for each $u\in B_{AP}^{p}(\mathbb{R}^{N})$
and all $a\in \mathbb{R}^{N}$, where $u(\cdot +a)(z)=u(z+a)$ for almost all $%
z\in \mathbb{R}^{N}$. Moreover for $u\in B_{AP}^{p}(\mathbb{R}^{N})$ we have 
$\left\Vert u\right\Vert _{p}=\left[ M(\left\vert u\right\vert ^{p})\right]
^{1/p}$.
\end{itemize}

We refer to \cite{Frid, Blot} for the definitions and properties of the
vector-valued spaces of almost periodic functions, namely, $AP(\mathbb{R}%
^{N};X)$ and $B_{AP}^{p}(\mathbb{R}^{N};X)$ and the connected spaces $%
\mathcal{C}(\mathcal{K};X)$ and $L^{p}(\mathcal{K};X)$, where $X$ is a given
Banach space. In particular when $X=\mathbb{C}$ we get $AP(\mathbb{R}^{N})$
and $B_{AP}^{p}(\mathbb{R}^{N})$ respectively.

Now let $\mathbb{R}_{y,\tau }^{N+1}=\mathbb{R}_{y}^{N}\times \mathbb{R}%
_{\tau }$ denotes the space $\mathbb{R}^{N}\times \mathbb{R}$ with generic
variables $(y,\tau )$. It is known that $AP(\mathbb{R}_{y,\tau }^{N+1})=AP(%
\mathbb{R}_{\tau };AP(\mathbb{R}_{y}^{N}))$ is the closure in $\mathcal{B}(%
\mathbb{R}_{y,\tau }^{N+1})$ of the tensor product $AP(\mathbb{R}%
_{y}^{N})\otimes AP(\mathbb{R}_{\tau })$ \cite{Chou}. In what follows, we
set $A_{y}=AP(\mathbb{R}_{y}^{N})$, $A_{\tau }=AP(\mathbb{R}_{\tau })$ and $%
A=AP(\mathbb{R}_{y,\tau }^{N+1})$. We denote the mean value on $A_{\zeta }$ (%
$\zeta =y,\tau $) by $M_{\zeta }$.

In the above notations, let $g\in A$ with $M_{y}(g)=0$. Then arguing as in 
\cite[p. 246]{Jikov} we see that there exists a unique $R\in A$ with $%
M_{y}(R)=0$ such that 
\begin{equation}
g=\Delta _{y}R  \label{2.1}
\end{equation}%
where $\Delta _{y}$ stands for the Laplacian operator defined on $\mathbb{R}%
_{y}^{N}$: $\Delta _{y}=\sum_{i=1}^{N}\partial ^{2}/\partial y_{i}^{2}$.
Owing to the hypoellipticity of the Laplacian on $\mathbb{R}^{N}$ we deduce
that the function $R$ is at least of class $C^{2}$ with respect to the
variable $y$. The above fact will be very useful in the last two sections of
the work.

\bigskip

Next following the theory presented in \cite[Chap. B1]{Vo-Khac} (see also 
\cite{Blot}), let $1\leq p<\infty $ and consider the $N$-parameter group of
isometries $\{T(y):y\in \mathbb{R}^{N}\}$ defined by 
\begin{equation*}
T(y):\mathcal{B}_{AP}^{p}(\mathbb{R}^{N})\rightarrow \mathcal{B}_{AP}^{p}(%
\mathbb{R}^{N})\text{,\ }T(y)(u+\mathcal{N})=\tau _{y}u+\mathcal{N}\text{
for }u\in B_{AP}^{p}(\mathbb{R}^{N})
\end{equation*}%
where $\tau _{y}u=u(\cdot +y)$. Since the elements of $AP(\mathbb{R}^{N})$
are uniformly continuous, $\{T(y):y\in \mathbb{R}^{N}\}$ is a strongly
continuous group in the sense of semigroups: $T(y)(u+\mathcal{N})\rightarrow
u+\mathcal{N}$ in $\mathcal{B}_{AP}^{p}(\mathbb{R}^{N})$ as $\left\vert
y\right\vert \rightarrow 0$. In view of the isometric isomorphism $\mathcal{G%
}_{1}$ we associated to $\{T(y):y\in \mathbb{R}^{N}\}$ the following $N$%
-parameter group $\{\overline{T}(y):y\in \mathbb{R}^{N}\}$ defined by 
\begin{equation*}
\begin{array}{l}
\overline{T}(y):L^{p}(\mathcal{K})\rightarrow L^{p}(\mathcal{K}) \\ 
\overline{T}(y)\mathcal{G}_{1}(u+\mathcal{N})=\mathcal{G}_{1}(T(y)(u+%
\mathcal{N}))=\mathcal{G}_{1}(\tau _{y}u+\mathcal{N})\text{\ for }u\in
B_{AP}^{p}(\mathbb{R}^{N})\text{.}%
\end{array}%
\end{equation*}%
The group $\{\overline{T}(y):y\in \mathbb{R}^{N}\}$ is also strongly
continuous. The infinitesimal generator of $T(y)$ (resp. $\overline{T}(y)$)
along the $i$th coordinate direction is denoted by $D_{i,p}$ (resp. $%
\partial _{i,p}$) and is defined by 
\begin{equation*}
\begin{array}{l}
D_{i,p}u=\lim_{t\rightarrow 0}t^{-1}\left( T(te_{i})u-u\right) \text{\ in }%
\mathcal{B}_{AP}^{p}(\mathbb{R}^{N})\text{ } \\ 
\text{(resp. }\partial _{i,p}v=\lim_{t\rightarrow 0}t^{-1}\left( \overline{T}%
(te_{i})v-v\right) \text{\ in }L^{p}(\mathcal{K})\text{)}%
\end{array}%
\end{equation*}%
where: here and henceforth, we have used the same letter $u$ to denote the
equivalence class of an element $u\in B_{AP}^{p}(\mathbb{R}^{N})$ in $%
\mathcal{B}_{AP}^{p}(\mathbb{R}^{N})$, $e_{i}=(\delta _{ij})_{1\leq j\leq N}$
($\delta _{ij}$ being the Kronecker $\delta $). The domain of $D_{i,p}$
(resp. $\partial _{i,p}$) in $\mathcal{B}_{AP}^{p}(\mathbb{R}^{N})$ (resp. $%
L^{p}(\mathcal{K})$) is denoted by $\mathcal{D}_{i,p}$ (resp. $\mathcal{W}%
_{i,p}$). By using the general theory of semigroups \cite[Chap. VIII,
Section 1]{DS}, the following result holds.

\begin{proposition}
\label{p2.1}$\mathcal{D}_{i,p}$ (resp. $\mathcal{W}_{i,p}$) is a vector
subspace of $\mathcal{B}_{AP}^{p}(\mathbb{R}^{N})$ (resp. $L^{p}(\mathcal{K}%
) $), $D_{i,p}:\mathcal{D}_{i,p}\rightarrow \mathcal{B}_{AP}^{p}(\mathbb{R}%
^{N})$ (resp. $\partial _{i,p}:\mathcal{W}_{i,p}\rightarrow L^{p}(\mathcal{K}%
)$) is a linear operator, $\mathcal{D}_{i,p}$ (resp. $\mathcal{W}_{i,p}$) is
dense in $\mathcal{B}_{AP}^{p}(\mathbb{R}^{N})$ (resp. $L^{p}(\mathcal{K})$%
), and the graph of $D_{i,p}$ (resp. $\partial _{i,p}$) is closed in $%
\mathcal{B}_{AP}^{p}(\mathbb{R}^{N})\times \mathcal{B}_{AP}^{p}(\mathbb{R}%
^{N})$ (resp. $L^{p}(\mathcal{K})\times L^{p}(\mathcal{K})$).
\end{proposition}

In the sequel we denote by $\varrho $ the canonical mapping of $B_{AP}^{p}(%
\mathbb{R}^{N})$ onto $\mathcal{B}_{AP}^{p}(\mathbb{R}^{N})$, that is, $%
\varrho (u)=u+\mathcal{N}$ for $u\in B_{AP}^{p}(\mathbb{R}^{N})$. The
following properties are immediate. The verification can be found either in 
\cite[Chap. B1]{Vo-Khac} or in \cite{Blot}.

\begin{lemma}
\label{l2.1}Let $1\leq i\leq N$. \emph{(1)} If $u\in AP^{1}(\mathbb{R}^{N})$
then $\varrho (u)\in \mathcal{D}_{i,p}$ and 
\begin{equation}
D_{i,p}\varrho (u)=\varrho \left( \frac{\partial u}{\partial y_{i}}\right)
.\ \ \ \ \ \ \ \ \ \ \ \ \ \ \ \ \ \ \ \ \ \ \ \ \ \ \ \ \ \ \ \ \ 
\label{2.2}
\end{equation}%
\emph{(2)} If $u\in \mathcal{D}_{i,p}$ then $\mathcal{G}_{1}(u)\in \mathcal{W%
}_{i,p}$ and $\mathcal{G}_{1}(D_{i,p}u)=\partial _{i,p}\mathcal{G}_{1}(u)$.
\end{lemma}

One can naturally define higher order derivatives by setting $D_{p}^{\alpha
}=D_{1,p}^{\alpha _{1}}\circ \cdot \cdot \cdot \circ D_{N,p}^{\alpha _{N}}$
(resp. $\partial _{p}^{\alpha }=\partial _{1,p}^{\alpha _{1}}\circ \cdot
\cdot \cdot \circ \partial _{N,p}^{\alpha _{N}}$) for $\alpha =(\alpha
_{1},...,\alpha _{N})\in \mathbb{N}^{N}$ with $D_{i,p}^{\alpha
_{i}}=D_{i,p}\circ \cdot \cdot \cdot \circ D_{i,p}$, $\alpha _{i}$-times.
Now, let 
\begin{equation*}
\mathcal{B}_{AP}^{1,p}(\mathbb{R}^{N})=\cap _{i=1}^{N}\mathcal{D}%
_{i,p}=\{u\in \mathcal{B}_{AP}^{p}(\mathbb{R}^{N}):D_{i,p}u\in \mathcal{B}%
_{AP}^{p}(\mathbb{R}^{N})\ \forall 1\leq i\leq N\}
\end{equation*}%
and 
\begin{equation*}
\mathcal{D}_{AP}(\mathbb{R}^{N})=\{u\in \mathcal{B}_{AP}^{\infty }(\mathbb{R}%
^{N}):D_{\infty }^{\alpha }u\in \mathcal{B}_{AP}^{\infty }(\mathbb{R}^{N})\
\forall \alpha \in \mathbb{N}^{N}\}.
\end{equation*}%
It can be shown that $\mathcal{D}_{AP}(\mathbb{R}^{N})$ is dense in $%
\mathcal{B}_{AP}^{p}(\mathbb{R}^{N})$, $1\leq p<\infty $. We also have that $%
\mathcal{B}_{AP}^{1,p}(\mathbb{R}^{N})$ is a Banach space under the norm 
\begin{equation*}
\left\| u\right\| _{\mathcal{B}_{AP}^{1,p}(\mathbb{R}^{N})}=\left( \left\|
u\right\| _{p}^{p}+\sum_{i=1}^{N}\left\| D_{i,p}u\right\| _{p}^{p}\right)
^{1/p}\ \ (u\in \mathcal{B}_{AP}^{1,p}(\mathbb{R}^{N}));
\end{equation*}%
this comes from the fact that the graph of $D_{i,p}$ is closed.

The counter-part of the above properties also holds with 
\begin{equation*}
W^{1,p}(\mathcal{K})=\cap _{i=1}^{N}\mathcal{W}_{i,p}\text{\ in place of }%
\mathcal{B}_{AP}^{1,p}(\mathbb{R}^{N})
\end{equation*}%
and 
\begin{equation*}
\mathcal{D}(\mathcal{K})=\{u\in L^{\infty }(\mathcal{K}):\partial _{\infty
}^{\alpha }u\in L^{\infty }(\mathcal{K})\ \forall \alpha \in \mathbb{N}^{N}\}%
\text{\ in that of }\mathcal{D}_{AP}(\mathbb{R}^{N})\text{.}
\end{equation*}%
Moreover the restriction of $\mathcal{G}_{1}$ to $\mathcal{B}_{AP}^{1,p}(%
\mathbb{R}^{N})$ is an isometric isomorphism of $\mathcal{B}_{AP}^{1,p}(%
\mathbb{R}^{N})$ onto $W^{1,p}(\mathcal{K})$; this comes from [Part (2) of]
Lemma \ref{l2.1}.

Let $u\in \mathcal{D}_{i,p}$ ($p\geq 1$, $1\leq i\leq N$). Then the
inequality 
\begin{equation*}
\left\Vert t^{-1}(T(te_{i})u-u)-D_{i,p}u\right\Vert _{1}\leq c\left\Vert
t^{-1}(T(te_{i})u-u)-D_{i,p}u\right\Vert _{p}
\end{equation*}%
for a positive constant $c$ independent of $u$ and $t$, yields $%
D_{i,1}u=D_{i,p}u$, so that $D_{i,p}$ is the restriction to $\mathcal{B}%
_{AP}^{p}(\mathbb{R}^{N})$ of $D_{i,1}$. Therefore, for all $u\in \mathcal{D}%
_{i,\infty }$ we have $u\in \mathcal{D}_{i,p}$ ($p\geq 1$) and $D_{i,\infty
}u=D_{i,p}u$\ $\forall 1\leq i\leq N$. We will need the following result in
the sequel.

\begin{lemma}
\label{l2.3}We have $\mathcal{D}_{AP}(\mathbb{R}^{N})=\varrho (AP^{\infty }(%
\mathbb{R}^{N}))$.
\end{lemma}

\begin{proof}
From (\ref{2.2}) we have that, for $u\in \varrho (AP^{\infty }(\mathbb{R}%
^{N}))$ and $\alpha \in \mathbb{N}^{N}$, $D_{\infty }^{\alpha }u=\varrho
(D_{y}^{\alpha }v)$ where $v\in AP^{\infty }(\mathbb{R}^{N})$ is such that $%
u=\varrho (v)$. This leads at once to $\varrho (AP^{\infty }(\mathbb{R}%
^{N}))\subset \mathcal{D}_{AP}(\mathbb{R}^{N})$. Conversely if $u\in 
\mathcal{D}_{AP}(\mathbb{R}^{N})$, then $u\in \mathcal{B}_{AP}^{\infty }(%
\mathbb{R}^{N})$ with $D_{\infty }^{\alpha }u\in \mathcal{B}_{AP}^{\infty }(%
\mathbb{R}^{N})$ for all $\alpha \in \mathbb{N}^{N}$, that is, $u=v+\mathcal{%
N}$ with $v\in B_{AP}^{\infty }(\mathbb{R}^{N})$ being such that $%
D_{y}^{\alpha }v\in B_{AP}^{\infty }(\mathbb{R}^{N})$ for all $\alpha \in 
\mathbb{N}^{N}$, i.e., $v\in AP^{\infty }(\mathbb{R}^{N})$ since, as $v$ is
in $L_{\text{loc}}^{p}(\mathbb{R}^{N})$ with all its distributional
derivatives, $v$ is of class $\mathcal{C}^{\infty }$. Hence $u=v+\mathcal{N}$
with $v\in AP^{\infty }(\mathbb{R}^{N})$, so that $u\in \varrho (AP^{\infty
}(\mathbb{R}^{N}))$.
\end{proof}

From now on, we write $\widehat{u}$ either for $\mathcal{G}(u)$ if $u\in
B_{AP}^{p}(\mathbb{R}^{N})$ or for $\mathcal{G}_{1}(u)$ if $u\in \mathcal{B}%
_{AP}^{p}(\mathbb{R}^{N})$. The following properties are easily verified
(see once again either \cite[Chap. B1]{Vo-Khac} or \cite{Blot}).

\begin{proposition}
\label{p2.2}The following assertions hold.

\begin{itemize}
\item[(i)] $\int_{\mathcal{K}}\partial _{\infty }^{\alpha }\widehat{u}d\beta
=0$ for all $u\in \mathcal{D}_{AP}(\mathbb{R}^{N})$ and $\alpha \in \mathbb{N%
}^{N}$;

\item[(ii)] $\int_{\mathcal{K}}\partial _{i,p}\widehat{u}d\beta =0$ for all $%
u\in \mathcal{D}_{i,p}$ and $1\leq i\leq N$;

\item[(iii)] $D_{i,p}(u\phi )=uD_{i,\infty }\phi +\phi D_{i,p}u$ for all $%
(\phi ,u)\in \mathcal{D}_{AP}(\mathbb{R}^{N})\times \mathcal{D}_{i,p}$ and $%
1\leq i\leq N$.
\end{itemize}
\end{proposition}

The formula (iii) in the above proposition leads to the equality 
\begin{equation*}
\int_{\mathcal{K}}\widehat{\phi }\partial _{i,p}\widehat{u}d\beta =-\int_{%
\mathcal{K}}\widehat{u}\partial _{i,\infty }\widehat{\phi }d\beta \ \
\forall (u,\phi )\in \mathcal{D}_{i,p}\times \mathcal{D}_{AP}(\mathbb{R}%
^{N}).
\end{equation*}%
This suggests us to define the concept of distributions on $\mathcal{D}_{AP}(%
\mathbb{R}^{N})$ and of a weak derivative. Before we can do that, let us
endow $\mathcal{D}_{AP}(\mathbb{R}^{N})=\varrho (AP^{\infty }(\mathbb{R}%
^{N}))$ with its natural topology defined by the family of norms $%
N_{n}(u)=\sup_{\left\vert \alpha \right\vert \leq n}\sup_{y\in \mathbb{R}%
^{N}}\left\vert D_{\infty }^{\alpha }u(y)\right\vert $, integers $n\geq 0$.
In this topology, $\mathcal{D}_{AP}(\mathbb{R}^{N})$ is a Fr\'{e}chet space.
We denote by $\mathcal{D}_{AP}^{\prime }(\mathbb{R}^{N})$ the topological
dual of $\mathcal{D}_{AP}(\mathbb{R}^{N})$. We endow it with the strong dual
topology. The elements of $\mathcal{D}_{AP}^{\prime }(\mathbb{R}^{N})$ are
called \textit{the distributions on }$\mathcal{D}_{AP}(\mathbb{R}^{N})$. One
can also define the weak derivative of $f\in \mathcal{D}_{AP}^{\prime }(%
\mathbb{R}^{N})$ as follows: for any $\alpha \in \mathbb{N}^{N}$, $D^{\alpha
}f$ stands for the distribution defined by the formula 
\begin{equation*}
\left\langle D^{\alpha }f,\phi \right\rangle =(-1)^{\left\vert \alpha
\right\vert }\left\langle f,D_{\infty }^{\alpha }\phi \right\rangle \text{\
for all }\phi \in \mathcal{D}_{AP}(\mathbb{R}^{N}).
\end{equation*}%
Since $\mathcal{D}_{AP}(\mathbb{R}^{N})$ is dense in $\mathcal{B}_{AP}^{p}(%
\mathbb{R}^{N})$ ($1\leq p<\infty $), it is immediate that $\mathcal{B}%
_{AP}^{p}(\mathbb{R}^{N})\subset \mathcal{D}_{AP}^{\prime }(\mathbb{R}^{N})$
with continuous embedding, so that one may define the weak derivative of any 
$f\in \mathcal{B}_{AP}^{p}(\mathbb{R}^{N})$, and it verifies the following
functional equation: 
\begin{equation*}
\left\langle D^{\alpha }f,\phi \right\rangle =(-1)^{\left\vert \alpha
\right\vert }\int_{\mathcal{K}}\widehat{f}\partial _{\infty }^{\alpha }%
\widehat{\phi }d\beta \text{\ for all }\phi \in \mathcal{D}_{AP}(\mathbb{R}%
^{N}).
\end{equation*}%
In particular, for $f\in \mathcal{D}_{i,p}$ we have 
\begin{equation*}
-\int_{\mathcal{K}}\widehat{f}\partial _{i,p}\widehat{\phi }d\beta =\int_{%
\mathcal{K}}\widehat{\phi }\partial _{i,p}\widehat{f}d\beta \ \ \forall \phi
\in \mathcal{D}_{AP}(\mathbb{R}^{N}),
\end{equation*}%
so that we may identify $D_{i,p}f$ with $D^{\alpha _{i}}f$, $\alpha
_{i}=(\delta _{ij})_{1\leq j\leq N}$. Conversely, if $f\in \mathcal{B}%
_{AP}^{p}(\mathbb{R}^{N})$ is such that there exists $f_{i}\in \mathcal{B}%
_{AP}^{p}(\mathbb{R}^{N})$ with $\left\langle D^{\alpha _{i}}f,\phi
\right\rangle =-\int_{\mathcal{K}}\widehat{f}_{i}\widehat{\phi }d\beta $ for
all $\phi \in \mathcal{D}_{AP}(\mathbb{R}^{N})$, then $f\in \mathcal{D}%
_{i,p} $ and $D_{i,p}f=f_{i}$. We are therefore justified in saying that $%
\mathcal{B}_{AP}^{1,p}(\mathbb{R}^{N})$ is a Banach space under the norm $%
\left\Vert \cdot \right\Vert _{\mathcal{B}_{AP}^{1,p}(\mathbb{R}^{N})}$. The
same result holds for $W^{1,p}(\mathcal{K})$. Moreover it is a fact that $%
\mathcal{D}_{AP}(\mathbb{R}^{N})$ (resp. $\mathcal{D}(\mathcal{K})$) is a
dense subspace of $\mathcal{B}_{AP}^{1,p}(\mathbb{R}^{N})$ (resp. $W^{1,p}(%
\mathcal{K})$).

We end this section with the definition of the space of correctors. For that
we need the following space: 
\begin{equation*}
\mathcal{B}_{AP}^{1,p}(\mathbb{R}^{N})/\mathbb{C}=\{u\in \mathcal{B}%
_{AP}^{1,p}(\mathbb{R}^{N}):M(u)=0\}.
\end{equation*}%
We endow it with the seminorm 
\begin{equation*}
\left\Vert u\right\Vert _{\#,p}=\left( \sum_{i=1}^{N}\left\Vert
D_{i,p}u\right\Vert _{p}^{p}\right) ^{1/p}\ \ (u\in \mathcal{B}_{AP}^{1,p}(%
\mathbb{R}^{N})/\mathbb{C}).
\end{equation*}%
One can check that this is actually a norm on $\mathcal{B}_{AP}^{1,p}(%
\mathbb{R}^{N})/\mathbb{C}$. With this norm $\mathcal{B}_{AP}^{1,p}(\mathbb{R%
}^{N})/\mathbb{C}$ is a normed vector space which is unfortunately not
complete. We denote by $\mathcal{B}_{\#AP}^{1,p}(\mathbb{R}^{N})$ its
completion with respect to the above norm and by $J_{p}$ the canonical
embedding of $\mathcal{B}_{AP}^{1,p}(\mathbb{R}^{N})/\mathbb{C}$ into $%
\mathcal{B}_{\#AP}^{1,p}(\mathbb{R}^{N})$. The following properties are due
to the theory of completion of uniform spaces (see \cite{Bourbaki}):

\begin{itemize}
\item[(P$_{1}$)] The gradient operator $D_{p}=(D_{1,p},...,D_{N,p}):\mathcal{%
B}_{AP}^{1,p}(\mathbb{R}^{N})/\mathbb{C}\rightarrow (\mathcal{B}_{AP}^{p}(%
\mathbb{R}^{N}))^{N}$ extends by continuity to a unique mapping $\overline{D}%
_{p}:\mathcal{B}_{\#AP}^{1,p}(\mathbb{R}^{N})\rightarrow (\mathcal{B}%
_{AP}^{p}(\mathbb{R}^{N}))^{N}$ with the properties 
\begin{equation*}
D_{i,p}=\overline{D}_{i,p}\circ J_{p}
\end{equation*}%
and 
\begin{equation*}
\left\Vert u\right\Vert _{\#,p}=\left( \sum_{i=1}^{N}\left\Vert \overline{D}%
_{i,p}u\right\Vert _{p}^{p}\right) ^{1/p}\ \ \text{for }u\in \mathcal{B}%
_{\#AP}^{1,p}(\mathbb{R}^{N}).
\end{equation*}

\item[(P$_{2}$)] The space $J_{p}(\mathcal{B}_{AP}^{1,p}(\mathbb{R}^{N})/%
\mathbb{C})$ (and hence $J_{p}(\mathcal{D}_{AP}(\mathbb{R}^{N})/\mathbb{C})$%
) is dense in $\mathcal{B}_{\#AP}^{1,p}(\mathbb{R}^{N})$.
\end{itemize}

\noindent Moreover the mapping $\overline{D}_{p}$ is an isometric embedding
of $\mathcal{B}_{\#AP}^{1,p}(\mathbb{R}^{N})$ onto a closed subspace of $(%
\mathcal{B}_{AP}^{p}(\mathbb{R}^{N}))^{N}$, so that $\mathcal{B}%
_{\#AP}^{1,p}(\mathbb{R}^{N})$ is a reflexive Banach space. By duality we
define the divergence operator $\Div_{p^{\prime }}:(\mathcal{B}_{AP}^{p}(%
\mathbb{R}^{N}))^{N}\rightarrow (\mathcal{B}_{\#AP}^{1,p}(\mathbb{R}%
^{N}))^{\prime }$ ($p^{\prime }=p/(p-1)$) by 
\begin{equation}
\left\langle \text{div}_{p^{\prime }}u,v\right\rangle =-\left\langle u,%
\overline{D}_{p}v\right\rangle \text{\ for }v\in \mathcal{B}_{\#AP}^{1,p}(%
\mathbb{R}^{N})\text{ and }u=(u_{i})\in (\mathcal{B}_{AP}^{p^{\prime }}(%
\mathbb{R}^{N}))^{N}\text{,}  \label{2.6}
\end{equation}%
where $\left\langle u,\overline{D}_{p}v\right\rangle =\sum_{i=1}^{N}\int_{%
\mathcal{K}}\widehat{u}_{i}\partial _{i,p}\widehat{v}d\beta $. The operator
div$_{p^{\prime }}$ just defined extends the natural divergence operator
defined in $\mathcal{D}_{AP}(\mathbb{R}^{N})$ since $D_{i,p}f=\overline{D}%
_{i,p}(J_{p}f)$ for all $f\in \mathcal{D}_{AP}(\mathbb{R}^{N})$.

Now if in (\ref{2.6}) we take $u=D_{p^{\prime }}w$ with $w\in \mathcal{B}%
_{AP}^{p^{\prime }}(\mathbb{R}^{N})$ being such that $D_{p^{\prime }}w\in (%
\mathcal{B}_{AP}^{p^{\prime }}(\mathbb{R}^{N}))^{N}$ then this allows us to
define the Laplacian operator on $\mathcal{B}_{AP}^{p^{\prime }}(\mathbb{R}%
^{N})$, denoted here by $\Delta _{p^{\prime }}$, as follows: 
\begin{equation}
\left\langle \Delta _{p^{\prime }}w,v\right\rangle =\left\langle \text{div}%
_{p^{\prime }}(D_{p^{\prime }}w),v\right\rangle =-\left\langle D_{p^{\prime
}}w,\overline{D}_{p}v\right\rangle \text{\ for all }v\in \mathcal{B}%
_{\#AP}^{1,p}(\mathbb{R}^{N}).  \label{2.7}
\end{equation}%
If in addition $v=J_{p}(\phi )$ with $\phi \in \mathcal{D}_{AP}(\mathbb{R}%
^{N})/\mathbb{C}$ then $\left\langle \Delta _{p^{\prime }}w,J_{p}(\phi
)\right\rangle =-\left\langle D_{p^{\prime }}w,D_{p}\phi \right\rangle $, so
that, for $p=2$, we get 
\begin{equation*}
\left\langle \Delta _{2}w,J_{2}(\phi )\right\rangle =\left\langle w,\Delta
_{2}\phi \right\rangle \text{\ for all }w\in \mathcal{B}_{AP}^{2}(\mathbb{R}%
^{N})\text{ and }\phi \in \mathcal{D}_{AP}(\mathbb{R}^{N})/\mathbb{C}\text{.}
\end{equation*}

The following result is also immediate.

\begin{proposition}
\label{p2.5}For $u\in AP^{\infty }(\mathbb{R}^{N})$ we have 
\begin{equation*}
\Delta _{p}\varrho (u)=\varrho (\Delta _{y}u)
\end{equation*}%
where $\Delta _{y}$ stands for the usual Laplacian operator on $\mathbb{R}%
_{y}^{N}$.
\end{proposition}

We end this subsection with some notations. Let $f\in \mathcal{B}_{AP}^{p}(%
\mathbb{R}^{N})$. We know that $D^{\alpha _{i}}f$ exists (in the sense of
distributions) and that $D^{\alpha _{i}}f=D_{i,p}f$ if $f\in \mathcal{D}%
_{i,p}$. So we can drop the subscript $p$ and therefore denote $D_{i,p}$
(resp. $\partial _{i,p}$) by $\overline{\partial }/\partial y_{i}$ (resp. $%
\partial _{i}$). Thus, $\overline{D}_{y}$ will stand for the gradient
operator $(\overline{\partial }/\partial y_{i})_{1\leq i\leq N}$ and $%
\overline{\Div}_{y}$ for the divergence operator $\Div_{p}$. We will also
denote the operator $\overline{D}_{i,p}$ by $\overline{\partial }/\partial
y_{i}$. Since $J_{p}$ is an embedding, this allows us to view $\mathcal{B}%
_{AP}^{1,p}(\mathbb{R}^{N})/\mathbb{C}$ (and hence $\mathcal{D}_{AP}(\mathbb{%
R}^{N})/\mathbb{C}$) as a dense subspace of $\mathcal{B}_{\#AP}^{1,p}(%
\mathbb{R}^{N})$. $D_{i,p}$ will therefore be seen as the restriction of $%
\overline{D}_{i,p}$ to $\mathcal{B}_{AP}^{1,p}(\mathbb{R}^{N})/\mathbb{C}$.
Thus we will henceforth omit $J_{p}$ in the notation if it is understood
from the context and there is no risk of confusion. This will lead to the
notation $\overline{D}_{p}=\overline{D}_{y}=(\overline{\partial }/\partial
y_{i})_{1\leq i\leq N}$ and $\partial _{p}=\partial =(\partial _{i})_{1\leq
i\leq N}$. Finally, we will denote the Laplacian operator on $\mathcal{B}%
_{AP}^{p}(\mathbb{R}^{N})$ by $\overline{\Delta }_{y} $.

\section{The $\Sigma $-convergence method for stochastic processes}

In this section we define an appropriate notion of the concept of $\Sigma $%
-convergence adapted to our situation. It is to be noted that it is built
according to the original notion introduced by Nguetseng \cite{Hom1}. Here
we adapt it to systems involving random behavior. In all that follows $Q$ is
an open subset of $\mathbb{\mathbb{R}}^{N}$ (integer $N\geq 1$), $T$ is a
positive real number and $Q_{T}=Q\times (0,T)$. Let $(\Omega ,\mathcal{F},%
\mathbb{P})$ be a probability space. The expectation on $(\Omega ,\mathcal{F}%
,\mathbb{P})$ will throughout be denoted by $\mathbb{E}$. Let us first
recall the definition of the Banach space of all bounded $\mathcal{F}$%
-measurable functions. Denoting by $F(\Omega )$ the Banach space of all
bounded functions $f:\Omega \rightarrow \mathbb{R}$ (with the sup norm), we
define $B(\Omega )$ as the closure in $F(\Omega )$ of the vector space $%
H(\Omega )$ consisting of all finite linear combinations of the
characteristic functions $1_{X}$ of sets $X\in \mathcal{F}$. Since $\mathcal{%
F}$ is an $\sigma $-algebra, $B(\Omega )$ is the Banach space of all bounded 
$\mathcal{F}$-measurable functions. Likewise we define the space $B(\Omega
;Z)$ of all bounded $(\mathcal{F},B_{Z})$-measurable functions $f:\Omega
\rightarrow Z$, where $Z$ is a Banach space endowed with the $\sigma $%
-algebra of Borelians $B_{Z}$. It is a fact that the tensor product $%
B(\Omega )\otimes Z$ is a dense subspace of $B(\Omega ;Z)$.

This being so, let $A_{y}=AP(\mathbb{R}_{y}^{N})$ and $A_{\tau }=AP(\mathbb{R%
}_{\tau })$. We know that $A=AP(\mathbb{R}_{y,\tau }^{N+1})$ is the closure
in $\mathcal{B}(\mathbb{R}_{y,\tau }^{N+1})$ of the tensor product $%
A_{y}\otimes A_{\tau }$. We denote by $\mathcal{K}_{y}$ (resp. $\mathcal{K}%
_{\tau }$, $\mathcal{K}$) the spectrum of $A_{y}$ (resp. $A_{\tau }$, $A$).
The same letter $\mathcal{G}$ will denote the Gelfand transformation on $%
A_{y}$, $A_{\tau }$ and $A$, as well. Points in $\mathcal{K}_{y}$ (resp. $%
\mathcal{K}_{\tau }$) are denoted by $s$ (resp. $s_{0}$). The Haar measure
on the compact group $\mathcal{K}_{y}$ (resp. $\mathcal{K}_{\tau }$) is
denoted by $\beta _{y}$ (resp. $\beta _{\tau }$). We have $\mathcal{K}=%
\mathcal{K}_{y}\times \mathcal{K}_{\tau }$ (Cartesian product) and the Haar
measure on $\mathcal{K}$ is precisely the product measure $\beta =\beta
_{y}\otimes \beta _{\tau }$; the last equality follows in an obvious way by
the density of $A_{y}\otimes A_{\tau }$ in $A$ and by the Fubini's theorem.
Points in $\Omega $ are as usual denoted by $\omega $.

Unless otherwise stated, random variables will always be considered on the
probability space $(\Omega ,\mathcal{F},\mathbb{P})$. Finally, the letter $E$
will throughout denote exclusively an ordinary sequence $(\varepsilon
_{n})_{n\in \mathbb{N}}$ with $0<\varepsilon _{n}\leq 1$ and $\varepsilon
_{n}\rightarrow 0$ as $n\rightarrow \infty $. In what follow, we use the
same notation as in the preceding section.

\begin{definition}
\label{d3.1}\emph{A sequence }$(u_{\varepsilon })_{\varepsilon >0}$\emph{\
of }$L^{p}(Q_{T})$\emph{-valued random variables (}$1\leq p<\infty $\emph{)
is said to }weakly $\Sigma $-converge\emph{\ in }$L^{p}(Q_{T}\times \Omega )$%
\emph{\ to some }$L^{p}(Q_{T};\mathcal{B}_{AP}^{p}(\mathbb{R}_{y,\tau
}^{N+1}))$\emph{-valued random variable }$u_{0}$\emph{\ if as }$\varepsilon
\rightarrow 0$\emph{, we have } 
\begin{equation}
\begin{array}{l}
\int_{Q_{T}\times \Omega }u_{\varepsilon }(x,t,\omega )f\left( x,t,\frac{x}{%
\varepsilon },\frac{t}{\varepsilon ^{2}},\omega \right) dxdtd\mathbb{P} \\ 
\ \ \ \ \ \ \ \ \ \ \ \ \ \ \rightarrow \iint_{Q_{T}\times \Omega \times 
\mathcal{K}}\widehat{u}_{0}(x,t,s,s_{0},\omega )\widehat{f}%
(x,t,s,s_{0},\omega )dxdtd\mathbb{P}d\beta%
\end{array}
\label{3.1}
\end{equation}%
\emph{for every }$f\in B(\Omega ;L^{p^{\prime }}(Q_{T};A))$\emph{\ (}$%
1/p^{\prime }=1-1/p$\emph{), where }$\widehat{u}_{0}=\mathcal{G}_{1}\circ
u_{0}$\emph{\ and }$\widehat{f}=\mathcal{G}_{1}\circ (\varrho \circ f)=%
\mathcal{G}\circ f$\emph{. We express this by writing} $u_{\varepsilon
}\rightarrow u_{0}$ in $L^{p}(Q_{T}\times \Omega )$-weak $\Sigma $.
\end{definition}

\begin{remark}
\label{r3.0}\emph{The above weak }$\Sigma $\emph{-convergence in }$%
L^{p}(Q_{T}\times \Omega )$\emph{\ implies the weak convergence in }$%
L^{p}(Q_{T}\times \Omega )$\emph{. One can also see from the density of }$%
B(\Omega )$\emph{\ in }$L^{p^{\prime }}(\Omega )$\emph{\ (in the case }$%
1<p<\infty $\emph{) that (\ref{3.1}) obviously holds for }$f\in L^{p^{\prime
}}(\Omega ;L^{p^{\prime }}(Q_{T};A))$\emph{. One can show as in the usual
framework of }$\Sigma $\emph{-convergence method \cite{Hom1} that each }$%
f\in L^{p}(\Omega ;L^{p}(Q_{T};A))$\emph{\ weakly }$\Sigma $\emph{-converges
to }$\varrho \circ f$\emph{\ (that we can identified here with its
representative }$f$\emph{).}
\end{remark}

As said in the introduction, in the case $p=2$, our convergence method is
formally motivated by the following fact: using the chaos decomposition of $%
u_{\varepsilon }$ and $f$ we get $u_{\varepsilon }(x,t,\omega
)=\sum_{j=1}^{\infty }u_{\varepsilon ,j}(x,t)\Phi _{j}(\omega )$ and $%
f(x,t,y,\tau ,\omega )=\sum_{k=1}^{\infty }f_{k}(x,t,y,\tau )\Phi
_{k}(\omega )$ where $u_{\varepsilon ,j}\in L^{2}(Q_{T})$ and $f_{k}\in
L^{2}(Q_{T};A)$, so that 
\begin{equation*}
\int_{Q_{T}\times \Omega }u_{\varepsilon }(x,t,\omega )f\left( x,t,\frac{x}{%
\varepsilon },\frac{t}{\varepsilon ^{2}},\omega \right) dxdtd\mathbb{P}
\end{equation*}%
can be formally written as 
\begin{equation*}
\sum_{j,k}\int_{\Omega }\Phi _{j}(\omega )\Phi _{k}(\omega )d\mathbb{P}%
\int_{Q_{T}}u_{\varepsilon ,j}(x,t)f_{k}\left( x,t,\frac{x}{\varepsilon },%
\frac{t}{\varepsilon ^{2}}\right) dxdt,
\end{equation*}%
and by the usual $\Sigma $-convergence method (see \cite{NA, Hom1}), as $%
\varepsilon \rightarrow 0$, 
\begin{equation*}
\int_{Q_{T}}u_{\varepsilon ,j}(x,t)f_{k}\left( x,t,\frac{x}{\varepsilon },%
\frac{t}{\varepsilon ^{2}}\right) dxdt\rightarrow \iint_{Q_{T}\times 
\mathcal{K}}\widehat{u}_{0,j}(x,t,s,s_{0})\widehat{f}_{k}\left(
x,t,s,s_{0}\right) dxdtd\beta .
\end{equation*}%
Hence, by setting 
\begin{equation*}
\widehat{u}_{0}(x,t,s,s_{0},\omega )=\sum_{j=1}^{\infty }\widehat{u}%
_{0,j}(x,t,s,s_{0})\Phi _{j}(\omega );\;\widehat{f}\left( x,t,s,s_{0},\omega
\right) =\sum_{k=1}^{\infty }\widehat{f}_{k}\left( x,t,s,s_{0}\right) \Phi
_{k}(\omega )
\end{equation*}%
we get (\ref{3.1}). This is of course what formally motivated our definition.

The following result holds.

\begin{theorem}
\label{t3.1}Let $1<p<\infty $. Let $(u_{\varepsilon })_{\varepsilon \in E}$
be a sequence of $L^{p}(Q_{T})$-valued random variables verifying the
following boundedness condition: 
\begin{equation*}
\sup_{\varepsilon \in E}\mathbb{E}\left\| u_{\varepsilon }\right\|
_{L^{p}(Q_{T})}^{p}<\infty .
\end{equation*}%
Then there exists a subsequence $E^{\prime }$ from $E$ such that the
sequence $(u_{\varepsilon })_{\varepsilon \in E^{\prime }}$ is weakly $%
\Sigma $-convergent in $L^{p}(Q_{T}\times \Omega )$.
\end{theorem}

\begin{proof}
Applying \cite[Theorem 3.1]{CMP} with $Y=L^{p^{\prime }}(Q_{T}\times \Omega
\times \mathcal{K})$ and $X=B(\Omega ;L^{p^{\prime }}(Q_{T};\mathcal{C}(%
\mathcal{K})))=\mathcal{G}(B(\Omega ;L^{p^{\prime }}(Q_{T};A)))$ we are led
at once to the result.
\end{proof}

The next result is of capital interest in the homogenization process.

\begin{theorem}
\label{t3.2}Let $1<p<\infty $. Let $(u_{\varepsilon })_{\varepsilon \in E}$
be a sequence of $L^{p}(0,T;W_{0}^{1,p}(Q))$-valued random variables which
satisfies the following estimate: 
\begin{equation*}
\sup_{\varepsilon \in E}\mathbb{E}\left\| u_{\varepsilon }\right\|
_{L^{p}(0,T;W_{0}^{1,p}(Q))}^{p}<\infty .
\end{equation*}%
Then there exist a subsequence $E^{\prime }$ of $E$, an $%
L^{p}(0,T;W_{0}^{1,p}(Q))$-valued random variable $u_{0}$ and an $%
L^{p}(Q_{T};\mathcal{B}_{AP}^{p}(\mathbb{R}_{\tau };\mathcal{B}_{\#AP}^{1,p}(%
\mathbb{R}_{y}^{N})))$-valued random variable $u_{1}$ such that, as $%
E^{\prime }\ni \varepsilon \rightarrow 0$, 
\begin{equation*}
u_{\varepsilon }\rightarrow u_{0}\ \text{in }L^{p}(Q_{T}\times \Omega )\text{%
-weak;}
\end{equation*}%
\begin{equation}
\frac{\partial u_{\varepsilon }}{\partial x_{i}}\rightarrow \frac{\partial
u_{0}}{\partial x_{i}}+\frac{\overline{\partial }u_{1}}{\partial y_{i}}\text{%
\ in }L^{p}(Q_{T}\times \Omega )\text{-weak }\Sigma \text{, }1\leq i\leq N.
\label{3.2}
\end{equation}
\end{theorem}

\begin{proof}
The proof of the above theorem follows exactly the same lines of reasoning
as the one of \cite[Theorem 3.6]{NA}.
\end{proof}

The above theorem will not be used in its present form. In practice, the
following modified version will be used.

\begin{theorem}
\label{t3.3}Assume the hypotheses of Theorem \emph{\ref{t3.2}} are
satisfied. Assume further that $p\geq 2$ and that there exist a subsequence $%
E^{\prime }$ from $E$ and a random variable $u_{0}$ with values in $%
L^{p}(0,T;W_{0}^{1,p}(Q))$ such that, as $E^{\prime }\ni \varepsilon
\rightarrow 0$, 
\begin{equation}
u_{\varepsilon }\rightarrow u_{0}\text{\ in }L^{2}(Q_{T})\text{\ }\mathbb{P}%
\text{-almost surely.}  \label{3.3}
\end{equation}%
Then there exist a subsequence of $E^{\prime }$ not relabeled and a random
variable $u_{1}$ with values in $L^{p}(Q_{T};$\linebreak $\mathcal{B}%
_{AP}^{p}(\mathbb{R}_{\tau };\mathcal{B}_{\#AP}^{1,p}(\mathbb{R}_{y}^{N})))$
such that \emph{(\ref{3.2})} holds as $E^{\prime }\ni \varepsilon
\rightarrow 0$.
\end{theorem}

\begin{proof}
Since $\sup_{\varepsilon \in E^{\prime }}\mathbb{E}\left\| Du_{\varepsilon
}\right\| _{L^{p}(Q_{T})^{N}}^{p}<\infty $, there exist a subsequence of $%
E^{\prime }$ not relabeled and $v=(v_{i})_{i}\in L^{p}(Q_{T}\times \Omega ;%
\mathcal{B}_{AP}^{p}(\mathbb{R}_{y,\tau }^{N+1}))^{N}$ such that $\frac{%
\partial u_{\varepsilon }}{\partial x_{j}}\rightarrow v_{j}$ in $%
L^{p}(Q_{T}\times \Omega )$-weak $\Sigma $. Let $\Phi _{\varepsilon
}(x,t,\omega )=\varphi (x,t)\Psi (x/\varepsilon )\chi (t/\varepsilon
^{2})\phi (\omega )$ ($(x,t,\omega )\in Q_{T}\times \Omega $) with $\varphi
\in \mathcal{C}_{0}^{\infty }(Q_{T})$, $\chi \in AP^{\infty }(\mathbb{R}%
_{\tau })$, $\phi \in B(\Omega )$ and $\Psi =(\psi _{j})_{1\leq j\leq N}\in
(AP^{\infty }(\mathbb{R}_{y}^{N}))^{N}$ with $\overline{{\Div}}_{y}[\varrho
_{y}^{N}(\Psi )]=0$ where $\varrho _{y}^{N}(\Psi ):=(\varrho _{y}(\psi
_{j}))_{1\leq j\leq N}$, $\varrho _{y}$ being denoting the canonical mapping
of $B_{AP}^{p}(\mathbb{R}_{y}^{N})$ into $\mathcal{B}_{AP}^{p}(\mathbb{R}%
_{y}^{N})$. Clearly 
\begin{equation*}
\sum_{j=1}^{N}\int_{Q_{T}\times \Omega }\frac{\partial u_{\varepsilon }}{%
\partial x_{j}}\varphi \psi _{j}^{\varepsilon }\chi ^{\varepsilon }\phi dxdtd%
\mathbb{P}=-\sum_{j=1}^{N}\int_{Q_{T}\times \Omega }u_{\varepsilon }\psi
_{j}^{\varepsilon }\frac{\partial \varphi }{\partial x_{j}}\chi
^{\varepsilon }\phi dxdtd\mathbb{P}
\end{equation*}%
where $\psi _{j}^{\varepsilon }(x)=\psi _{j}(x/\varepsilon )$ and $\chi
^{\varepsilon }(t)=\chi (t/\varepsilon ^{2})$. One can easily see that
assumption (\ref{3.3}) implies the weak $\Sigma $-convergence of $%
(u_{\varepsilon })_{\varepsilon \in E^{\prime }}$ towards $u_{0}$, so that,
passing to the limit in the above equation when $E^{\prime }\ni \varepsilon
\rightarrow 0$ yields 
\begin{equation*}
\sum_{j=1}^{N}\iint_{Q_{T}\times \Omega \times \mathcal{K}}\widehat{v}%
_{j}\varphi \widehat{\psi }_{j}\widehat{\chi }\phi dxdtd\mathbb{P}d\beta
=-\sum_{j=1}^{N}\iint_{Q_{T}\times \Omega \times \mathcal{K}}u_{0}\widehat{%
\psi }_{j}\frac{\partial \varphi }{\partial x_{j}}\widehat{\chi }\phi dxdtd%
\mathbb{P}d\beta
\end{equation*}%
or equivalently, 
\begin{equation*}
\iint_{Q_{T}\times \Omega \times \mathcal{K}}\left( \widehat{\mathbf{v}}%
-Du_{0}\right) \cdot \widehat{\Psi }\varphi \widehat{\chi }\phi dxdtd\mathbb{%
P}d\beta =0\text{,}
\end{equation*}%
and so, as $\varphi $, $\phi $ and $\chi $ are arbitrarily fixed, 
\begin{equation*}
\int_{\mathcal{K}_{y}}\left( \widehat{\mathbf{v}}(x,t,s,s_{0},\omega
)-Du_{0}(x,t,\omega )\right) \cdot \widehat{\Psi }(s)d\beta _{y}=0
\end{equation*}%
for all $\Psi $ as above and for a.e. $x,t,s_{0},\omega $. Therefore, the
existence of a function $u_{1}(x,t,\cdot ,\tau ,\omega )\in \mathcal{B}%
_{\#AP}^{1,p}(\mathbb{R}_{y}^{N})$ such that 
\begin{equation*}
\mathbf{v}(x,t,\cdot ,\tau ,\omega )-Du_{0}(x,t,\omega )=\overline{D}%
_{y}u_{1}(x,t,\cdot ,\tau ,\omega )
\end{equation*}%
for a.e. $x,t,\tau ,\omega $ is ensured by a well-known classical result.
This yields the existence of a random variable $u_{1}:(x,t,\tau ,\omega
)\mapsto u_{1}(x,t,\cdot ,\tau ,\omega )$ with values in $\mathcal{B}%
_{\#AP}^{1,p}(\mathbb{R}_{y}^{N})$ such that $\mathbf{v}=Du_{0}+\overline{D}%
_{y}u_{1}$.
\end{proof}

We will also deal with the product of sequences. For that reason, we give
one further

\begin{definition}
\label{d3.2}\emph{A sequence }$(u_{\varepsilon })_{\varepsilon >0}$\emph{\
of }$L^{p}(Q_{T})$\emph{-valued random variables (}$1\leq p<\infty $\emph{)
is said to }strongly $\Sigma $-converge\emph{\ in }$L^{p}(Q_{T}\times \Omega
)$\emph{\ to some }$L^{p}(Q_{T};\mathcal{B}_{AP}^{p}(\mathbb{R}_{y,\tau
}^{N+1}))$\emph{-valued random variable }$u_{0}$\emph{\ if it is weakly }$%
\Sigma $\emph{-convergent towards }$u_{0}$\emph{\ and further satisfies the
following condition: }%
\begin{equation}
\left\| u_{\varepsilon }\right\| _{L^{p}(Q_{T}\times \Omega )}\rightarrow
\left\| \widehat{u}_{0}\right\| _{L^{p}(Q_{T}\times \Omega \times \mathcal{K}%
)}.  \label{3.12}
\end{equation}%
\emph{We denote this by }$u_{\varepsilon }\rightarrow u_{0}$\emph{\ in }$%
L^{p}(Q_{T}\times \Omega )$\emph{-strong }$\Sigma $\emph{.}
\end{definition}

\begin{remark}
\label{r3.1}\emph{(1) By the above definition, the uniqueness of the limit
of such a sequence is ensured. (2) By \cite{Hom1} it is immediate that for
any }$u\in L^{p}(Q_{T}\times \Omega ;AP(\mathbb{R}_{y,\tau }^{N+1}))$\emph{,
the sequence }$(u^{\varepsilon })_{\varepsilon >0}$\emph{\ is strongly }$%
\Sigma $\emph{-convergent to }$\varrho (u)$\emph{.}
\end{remark}

The next result will be very useful in the last section of this paper. Its
proof is copied on the one of \cite[Theorem 6]{WoukengArxiv}; see also \cite%
{Zhikov1}.

\begin{theorem}
\label{t3.4}Let $1<p,q<\infty $ and $r\geq 1$ be such that $1/r=1/p+1/q\leq
1 $. Assume $(u_{\varepsilon })_{\varepsilon \in E}\subset L^{q}(Q_{T}\times
\Omega )$ is weakly $\Sigma $-convergent in $L^{q}(Q_{T}\times \Omega )$ to
some $u_{0}\in L^{q}(Q_{T}\times \Omega ;\mathcal{B}_{AP}^{q}(\mathbb{R}%
_{y,\tau }^{N+1}))$, and $(v_{\varepsilon })_{\varepsilon \in E}\subset
L^{p}(Q_{T}\times \Omega )$ is strongly $\Sigma $-convergent in $%
L^{p}(Q_{T}\times \Omega )$ to some $v_{0}\in L^{p}(Q_{T}\times \Omega ;%
\mathcal{B}_{AP}^{p}(\mathbb{R}_{y,\tau }^{N+1}))$. Then the sequence $%
(u_{\varepsilon }v_{\varepsilon })_{\varepsilon \in E}$ is weakly $\Sigma $%
-convergent in $L^{r}(Q_{T}\times \Omega )$ to $u_{0}v_{0}$.
\end{theorem}

The following result will be of great interest in practice. It is a mere
consequence of the preceding theorem.

\begin{corollary}
\label{c3.4}Let $(u_{\varepsilon })_{\varepsilon \in E}\subset
L^{p}(Q_{T}\times \Omega )$ and $(v_{\varepsilon })_{\varepsilon \in
E}\subset L^{p^{\prime }}(Q_{T}\times \Omega )\cap L^{\infty }(Q_{T}\times
\Omega )$ ($1<p<\infty $ and $p^{\prime }=p/(p-1)$) be two sequences such
that:

\begin{itemize}
\item[(i)] $u_{\varepsilon }\rightarrow u_{0}$ in $L^{p}(Q_{T}\times \Omega
) $-weak $\Sigma $;

\item[(ii)] $v_{\varepsilon }\rightarrow v_{0}$ in $L^{p^{\prime
}}(Q_{T}\times \Omega )$-strong $\Sigma $;

\item[(iii)] $(v_{\varepsilon })_{\varepsilon \in E}$ is bounded in $%
L^{\infty }(Q_{T}\times \Omega )$.
\end{itemize}

\noindent Then $u_{\varepsilon }v_{\varepsilon }\rightarrow u_{0}v_{0}$ in $%
L^{p}(Q_{T}\times \Omega )$-weak $\Sigma $.
\end{corollary}

\begin{proof}
By Theorem \ref{t3.4}, the sequence $(u_{\varepsilon }v_{\varepsilon
})_{\varepsilon \in E}$ $\Sigma $-converges towards $u_{0}v_{0}$ in $%
L^{1}(Q_{T}\times \Omega )$. Besides the same sequence is bounded in $%
L^{p}(Q_{T}\times \Omega )$ so that by the Theorem \ref{t3.1}, it weakly $%
\Sigma $-converges in $L^{p}(Q_{T}\times \Omega )$ towards some $w_{0}\in
L^{p}(Q_{T}\times \Omega ;\mathcal{B}_{AP}^{p}(\mathbb{R}^{N+1}))$. This
gives as a result $w_{0}=u_{0}v_{0}$.
\end{proof}

\section{Statement of the problem: A priori estimates and tightness property}

\subsection{Statement of the problem}

Let $Q$ be a Lipschitz domain of $\mathbb{R}^{N}$ and $T$ a positive real
number. By $Q_{T}$ we denote the cylinder $Q\times (0,T)$. On a given
probability space $(\Omega ,\mathcal{F},\mathbb{P})$ is defined a prescribed 
$m$-dimensional standard Wiener process $W$. We equip $(\Omega ,\mathcal{F},%
\mathbb{P})$ with the natural filtration of $W$. We consider the following
stochastic partial differential equations 
\begin{equation}
\left\{ 
\begin{array}{l}
du_{\varepsilon }=\left( \Div\left( a\left( \frac{x}{\varepsilon },\frac{t}{%
\varepsilon ^{2}}\right) Du_{\varepsilon }\right) +\frac{1}{\varepsilon }%
g\left( \frac{x}{\varepsilon },\frac{t}{\varepsilon ^{2}},u_{\varepsilon
}\right) \right) dt+M\left( \frac{x}{\varepsilon },\frac{t}{\varepsilon ^{2}}%
,u_{\varepsilon }\right) dW\text{ in }Q_{T} \\ 
u_{\varepsilon }=0\text{ on }\partial Q\times (0,T) \\ 
u_{\varepsilon }(x,0)=u^{0}(x)\in L^{2}(Q).%
\end{array}%
\right.  \label{4.1}
\end{equation}%
We assume that the coefficients of (\ref{4.1}) are constrained as follows:

\begin{itemize}
\item[\textbf{A1}] \textbf{Uniform ellipticity}. The matrix $a(y,\tau
)=(a_{ij}(y,\tau ))_{1\leq i,j\leq N}\in (L^{\infty }(\mathbb{R}%
^{N+1}))^{N\times N}$ is real, not necessarily symmetric, positive definite,
i.e, there exists $\Lambda >0$ such that 
\begin{equation*}
\begin{array}{l}
\left\Vert a_{ij}\right\Vert _{L^{\infty }(\mathbb{R}^{N+1})}<\Lambda
^{-1},\,\,1\leq i,j\leq N, \\ 
\sum_{i,j=1}^{N}a_{ij}(y,\tau )\zeta _{i}\zeta _{j}\geq \Lambda |\zeta |^{2}%
\text{ for all }(y,\tau )\in \mathbb{R}^{N+1},\zeta \in \mathbb{R}^{N}.%
\end{array}%
\end{equation*}

\item[\textbf{A2}] \textbf{Lipschitz continuity}. There exists $C>0$ such
that for any $(y,\tau )\in \mathbb{R}^{N+1}$ and $u\in \mathbb{R}$ 
\begin{equation*}
\begin{array}{l}
\left\vert \partial _{u}g(y,\tau ,u)\right\vert \leq C \\ 
\left\vert \partial _{u}g(y,\tau ,u_{1})-\partial _{u}g(y,\tau
,u_{2})\right\vert \leq C\left\vert u_{1}-u_{2}\right\vert (1+\left\vert
u_{1}\right\vert +\left\vert u_{2}\right\vert )^{-1}.%
\end{array}%
\end{equation*}

\item[\textbf{A3}] $g(y,\tau ,0)=0$ for any $(y,\tau )\in \mathbb{R}^{N+1}$.

\item[\textbf{A4}] \textbf{Almost periodicity}. We assume that $g(\cdot
,\cdot ,u)\in AP(\mathbb{R}_{y,\tau }^{N+1})$ for any $u\in \mathbb{R}$ with 
$M_{y}(g(\cdot ,\tau ,u))=0$ for all $(\tau ,u)\in \mathbb{R}^{2}$. We see
by (\ref{2.1}) (see Section 2) that there exists a unique $R(\cdot ,\cdot
,u)\in AP(\mathbb{R}_{y,\tau }^{N+1})$ such that $\Delta _{y}R(\cdot ,\cdot
,u)=g(\cdot ,\cdot ,u)$ and $M_{y}(R(\cdot ,\tau ,u))=0$ for all $\tau $, $%
u\in \mathbb{R}$. Moreover $R(\cdot ,\cdot ,u)$ is at least twice
differentiable with respect to $y$. Let $G(y,\tau ,u)=D_{y}R(y,\tau ,u)$.
Thanks to \textbf{A2} and \textbf{A3} we see that 
\begin{equation}
\left\vert G(y,\tau ,u)\right\vert \leq C\left\vert u\right\vert \text{, }%
\left\vert \partial _{u}G(y,\tau ,u)\right\vert \leq C\text{,}  \label{4.0}
\end{equation}%
\begin{equation}
\left\vert \partial _{u}G(y,\tau ,u_{1})-\partial _{u}G(y,\tau
,u_{2})\right\vert \leq C\left\vert u_{1}-u_{2}\right\vert (1+\left\vert
u_{1}\right\vert +\left\vert u_{2}\right\vert )^{-1}.  \label{4.0'}
\end{equation}%
We also assume that the functions $a_{ij}$ lie in $B_{AP}^{2}(\mathbb{R}%
_{y,\tau }^{N+1})$ for all $1\leq i,j\leq N$.

\item[\textbf{A5}] Let $M(u)=(M_{i}(y,\tau ,u))_{1\leq i\leq m}$ and we
assume that there exists $K>0$ such that 
\begin{equation*}
\begin{array}{l}
\sum_{i=1}^{m}\left\vert M_{i}(y,\tau ,0)\right\vert ^{2} \leq K, \\ 
\left\vert M_{i}(y,\tau ,u_{1})-M_{i}(y,\tau ,u_{2})\right\vert \leq
K\left\vert u_{1}-u_{2}\right\vert \,\,\,\,i=1,...,m%
\end{array}%
\end{equation*}%
\text{for any } $(y,\tau )\in \mathbb{R}^{N+1}\text{ and }u_{1},u_{2}\in 
\mathbb{R}$. We easily see from these equations that 
\begin{equation*}
\sum_{i=1}^{m}\left\vert M_{i}(y,\tau ,u)\right\vert ^{2}\leq K(1+\left\vert
u\right\vert ^{2})\text{ for any }u\in \mathbb{R},(y,\tau )\in \mathbb{R}%
^{N+1}.
\end{equation*}%
Moreover we assume that the function $(y,\tau )\mapsto M_{i}(y,\tau ,u)$
lies in $B_{AP}^{2}(\mathbb{R}_{y,\tau }^{N+1})\cap L^{\infty }(\mathbb{R}%
_{y,\tau }^{N+1})$.
\end{itemize}

In order to simplify our presentation, we need to make some notations that
will be used in the sequel. We denote by $L^{2}(Q)$ and $H^{1}(Q)$ the usual
Lebesgue space and Sobolev space, respectively. By $(u,v)$ we denote the
inner product in $L^{2}(Q)$. Its associated norm is denoted by $\left\vert
\cdot \right\vert $. The space of elements of $H^{1}(Q)$ whose trace
vanishes on $\partial Q$ is denoted by $H_{0}^{1}(Q)$. Thanks to Poincar\'{e}%
's inequality we can endow $H_{0}^{1}(Q)$ with the inner product $%
((u,v))=\int_{Q}Du\cdot Dvdx$ whose associated norm $\left\Vert u\right\Vert 
$ is equivalent to the usual $H^{1}$-norm for any $u\in H_{0}^{1}(Q)$. The
duality pairing between $H_{0}^{1}(Q)$ and $H^{-1}(Q)$ is denoted by $%
\left\langle \cdot ,\cdot \right\rangle $.

Let $X$ be a Banach space, by $L^{p}(0,T;X)$ we mean the space of measurable
functions $\phi :[0,T]\rightarrow X$ such that 
\begin{equation*}
\begin{cases}
\left( \int_{0}^{T}\left\| \phi (t)\right\| _{X}^{p}\right) ^{1/p}<\infty 
\text{ if }1\leq 1<\infty , \\ 
\text{ess}\sup_{t\in \lbrack 0,T]}\left\| \phi (t)\right\| _{X}<\infty \text{
if }p=\infty .%
\end{cases}%
\end{equation*}%
Similarly we can define the space $L^{p}(\Omega ;X)$ where $(\Omega ,%
\mathcal{F},\mathbb{P})$ is a probability space.

From the work of \cite{KRYLOV} for example (see also \cite{pardoux}), the
existence and uniqueness of solution $u_{\varepsilon }$ of (\ref{4.1}) which
is subjected to conditions \textbf{A1}-\textbf{A5} are very well-known.

\begin{theorem}[\protect\cite{KRYLOV}]
\label{t4.1}For any fixed $\varepsilon >0$, there exists an $\mathcal{F}^{t}$%
-progressively measurable process $u_{\varepsilon }\in L^{2}(\Omega \times
\lbrack 0,T];H_{0}^{1}(Q))$ such that 
\begin{eqnarray}
\left( u_{\varepsilon }(t),v\right) &=&\left( u^{0},v\right)
-\int_{0}^{t}\left( a\left( \frac{x}{\varepsilon },\frac{\tau }{\varepsilon
^{2}}\right) Du_{\varepsilon }(\tau ),Dv\right) d\tau +\frac{1}{\varepsilon }%
\int_{0}^{t}g\left( \frac{x}{\varepsilon },\frac{\tau }{\varepsilon ^{2}}%
,u_{\varepsilon }(\tau )\right) vd\tau  \notag \\
&&+\int_{0}^{t}\left( M\left( \frac{x}{\varepsilon },\frac{\tau }{%
\varepsilon ^{2}},u_{\varepsilon }(\tau )\right) ,v\right) dW  \label{4.2}
\end{eqnarray}%
for any $v\in H_{0}^{1}(Q)$ and for almost all $(\omega ,t)\in \Omega \times
\lbrack 0,T]$. Such a process is unique in the following sense: 
\begin{equation*}
\mathbb{P}\left( \omega :u_{\varepsilon }(t)=\overline{u}_{\varepsilon }(t)%
\text{\ in }H^{-1}(Q)\ \forall t\in \lbrack 0,T]\right) =1
\end{equation*}%
for any $u_{\varepsilon }$ and $\overline{u}_{\varepsilon }$ satisfying 
\emph{(\ref{4.2})}.
\end{theorem}

\subsection{A priori estimates and tightness property of $u_{\protect%
\varepsilon }$}

We begin this section by obtaining crucial uniform a priori energy estimates
for the process $u_{\varepsilon }$.

\begin{lemma}
\label{l4.1}Under assumptions \textbf{A1}-\textbf{A5} the following
estimates hold true for $1\leq p<\infty $: 
\begin{equation}
\mathbb{E}\sup_{0\leq t\leq T}\left| u_{\varepsilon }(t)\right| ^{p}\leq C,
\label{4.3}
\end{equation}%
\begin{equation}
\mathbb{E}\left( \int_{0}^{T}\left\| u_{\varepsilon }(t)\right\|
^{2}dt\right) ^{p/2}\leq C  \label{4.4}
\end{equation}%
where $C$ is a positive constant which does not depend on $\varepsilon $.
\end{lemma}

\begin{proof}
Thanks to \cite{KRYLOV} or \cite{pardoux} $u_{\varepsilon }\in \mathcal{C}%
(0,T;L^{2}(Q))$ almost surely and $u_{\varepsilon }\in L^{2}(\Omega \times
\lbrack 0,T];H_{0}^{1}(Q))$, then we may apply It\^{o}'s formula to $%
|u_{\varepsilon }(t)|^{2}$ and we get%
\begin{eqnarray}
d|u_{\varepsilon }(t)|^{2} &=&-2\left( a\left( \frac{x}{\varepsilon },\frac{t%
}{\varepsilon ^{2}}\right) Du_{\varepsilon }(t),Du_{\varepsilon }(t)\right)
dt+\frac{2}{\varepsilon }\left( g\left( \frac{x}{\varepsilon },\frac{t}{%
\varepsilon ^{2}},u_{\varepsilon }(t)\right) ,u_{\varepsilon }(t)\right) dt 
\notag \\
&&+\sum_{k=1}^{m}|M_{k}^{\varepsilon }(u_{\varepsilon
}(t))|^{2}dt+2(M^{\varepsilon }(u_{\varepsilon }(t)),u_{\varepsilon }(t))dW
\label{4.5}
\end{eqnarray}%
where we set $M^{\varepsilon }(u_{\varepsilon })(x,t,\omega
)=M(x/\varepsilon ,t/\varepsilon ^{2},u_{\varepsilon }(x,t,\omega ))$.
Thanks to condition \textbf{A1} we have%
\begin{eqnarray}
d\left| u_{\varepsilon }(t)\right| ^{2}+2\Lambda \left\| u_{\varepsilon
}(t)\right\| ^{2}dt &\leq &\frac{2}{\varepsilon }(g(\frac{x}{\varepsilon },%
\frac{t}{\varepsilon ^{2}},u_{\varepsilon }(t)),u_{\varepsilon
}(t))dt+\sum_{k=1}^{m}\left| M_{k}^{\varepsilon }\left( u_{\varepsilon
}(t)\right) \right| ^{2}dt  \notag \\
&&+2(M^{\varepsilon }\left( u_{\varepsilon }(t)\right) ,u_{\varepsilon
}(t))dW.  \label{4}
\end{eqnarray}%
To deal with the first term of the right hand side of \eqref{4}, we use the
following representation 
\begin{equation}
\frac{1}{\varepsilon }g\left( \frac{x}{\varepsilon },\frac{t}{\varepsilon
^{2}},u_{\varepsilon }\right) =\Div G\left( \frac{x}{\varepsilon },\frac{t}{%
\varepsilon ^{2}},u_{\varepsilon }\right) -\partial _{u}G\left( \frac{x}{%
\varepsilon },\frac{t}{\varepsilon ^{2}},u_{\varepsilon }\right) \cdot
Du_{\varepsilon }  \label{4d}
\end{equation}%
which can be checked by straightforward computation. From this we see that 
\begin{equation*}
\frac{1}{\varepsilon }\left( g\left( \frac{x}{\varepsilon },\frac{t}{%
\varepsilon ^{2}},u_{\varepsilon }(t)\right) ,u_{\varepsilon }(t)\right)
=\left( G\left( \frac{x}{\varepsilon },\frac{t}{\varepsilon ^{2}}%
,u_{\varepsilon }\right) ,Du_{\varepsilon }(t)\right) -\left( \partial
_{u}G\left( \frac{x}{\varepsilon },\frac{t}{\varepsilon ^{2}},u_{\varepsilon
}(t)\right) \cdot Du_{\varepsilon }(t),u_{\varepsilon }(t)\right) ,
\end{equation*}%
from which we infer that 
\begin{equation}
\frac{2}{\varepsilon }\left( g\left( \frac{x}{\varepsilon },\frac{t}{%
\varepsilon ^{2}},u_{\varepsilon }(t)\right) ,u_{\varepsilon }(t)\right)
\leq C\left| u_{\varepsilon }(t)\right| \left\| u_{\varepsilon }(t)\right\|
+C\left| u_{\varepsilon }(t)\right| \left\| u_{\varepsilon }(t)\right\| .
\label{5}
\end{equation}%
Here we have used the assumptions \textbf{A2}-\textbf{A4}. Thanks to \textbf{%
A5} the second term of the right hand side of \eqref{4} can be estimated as 
\begin{equation}
\sum_{k=1}^{m}\left| M_{k}^{\varepsilon }(u_{\varepsilon }(t))\right|
^{2}\leq C(1+\left| u_{\varepsilon }(t)\right| ^{2}).  \label{6}
\end{equation}%
Using \eqref{5} and \eqref{6} in \eqref{4} and integrating over $0\leq \tau
\leq t$ both sides of the resulting inequality yields%
\begin{eqnarray}
\left| u_{\varepsilon }(t)\right| ^{2}+2\Lambda \int_{0}^{t}\left\|
u_{\varepsilon }(\tau )\right\| ^{2}d\tau &\leq &\left| u^{0}\right|
^{2}+C\int_{0}^{t}\left| u_{\varepsilon }(\tau )\right| \left\|
u_{\varepsilon }(\tau )\right\| d\tau +C(T)  \label{7} \\
&&+C\int_{0}^{t}\left| u_{\varepsilon }(\tau )\right| ^{2}d\tau
+2\int_{0}^{t}(M^{\varepsilon }(u_{\varepsilon }(\tau )),u_{\varepsilon
}(\tau ))dW.  \notag
\end{eqnarray}%
By Cauchy's inequality we have 
\begin{eqnarray}
\left| u_{\varepsilon }(t)\right| ^{2}+2\Lambda \int_{0}^{t}\left\|
u_{\varepsilon }(\tau )\right\| ^{2}d\tau &\leq &\left| u^{0}\right|
^{2}+C(\delta )\int_{0}^{t}\left| u_{\varepsilon }(\tau )\right| ^{2}d\tau
+\delta \int_{0}^{t}||u_{\varepsilon }(\tau )||^{2}d\tau +C(T)  \notag
\label{8} \\
&&+C\int_{0}^{t}\left| u_{\varepsilon }(\tau )\right| ^{2}d\tau
+2\int_{0}^{t}(M^{\varepsilon }\left( u_{\varepsilon }(\tau )\right)
,u_{\varepsilon }(\tau ))dW,  \notag
\end{eqnarray}%
where $\delta $ is an arbitrary positive constant. We choose $\delta
=\Lambda $ so that we see from \eqref{8} that 
\begin{eqnarray}
\left| u_{\varepsilon }(t)\right| ^{2}+\Lambda \int_{0}^{t}\left\|
u_{\varepsilon }(\tau )\right\| ^{2}d\tau &\leq &\left| u^{0}\right|
^{2}+C(T)+C\int_{0}^{t}\left| u_{\varepsilon }(\tau )\right| ^{2}d\tau 
\notag \\
&&+2\int_{0}^{t}(M^{\varepsilon }\left( u_{\varepsilon }(\tau )\right)
,u_{\varepsilon }(\tau ))dW.  \label{9}
\end{eqnarray}%
In \eqref{9} we take the $\sup $ over $0\leq \tau \leq t$ and the
mathematical expectation. This procedure implies that 
\begin{eqnarray}
\mathbb{E}\sup_{0\leq \tau \leq t}\left| u_{\varepsilon }(\tau )\right|
^{2}+\Lambda \mathbb{E}\int_{0}^{t}\left\| u_{\varepsilon }(\tau )\right\|
^{2}d\tau &\leq &\left| u^{0}\right| ^{2}+C(T)+C\mathbb{E}\int_{0}^{t}\left|
u_{\varepsilon }(\tau )\right| ^{2}d\tau  \notag  \label{10} \\
&&+2\mathbb{E}\sup_{0\leq s\leq t}\left| \int_{0}^{s}(M^{\varepsilon }\left(
u_{\varepsilon }(\tau )\right) ,u_{\varepsilon }(\tau ))dW\right| .  \notag
\end{eqnarray}%
By Burkh\"{o}lder-Davis-Gundy's inequality we have that 
\begin{eqnarray*}
2\mathbb{E}\sup_{0\leq s\leq t}\left| \int_{0}^{s}(M^{\varepsilon
}(u_{\varepsilon }(\tau )),u_{\varepsilon }(\tau ))dW\right| &\leq &6\mathbb{%
E}\left( \int_{0}^{t}(M^{\varepsilon }(u_{\varepsilon }(\tau
)),u_{\varepsilon }(\tau ))^{2}d\tau \right) ^{1/2} \\
&\leq &6\mathbb{E}\left( \sup_{0\leq s\leq t}\left| u_{\varepsilon
}(s)\right| \left( \int_{0}^{t}\left| M^{\varepsilon }\left( u_{\varepsilon
}(\tau )\right) \right| ^{2}d\tau \right) ^{1/2}\right) .
\end{eqnarray*}%
By Cauchy's inequality, 
\begin{equation*}
2\mathbb{E}\sup_{0\leq s\leq t}\left| \int_{0}^{s}(M^{\varepsilon }\left(
u_{\varepsilon }(\tau )\right) ,u_{\varepsilon }(\tau ))dW\right| \leq \frac{%
1}{2}\mathbb{E}\sup_{0\leq s\leq t}\left| u_{\varepsilon }(s)\right| ^{2}+18%
\mathbb{E}\int_{0}^{t}\left| M^{\varepsilon }(u_{\varepsilon }(\tau
))\right| ^{2}d\tau .
\end{equation*}%
By using condition \textbf{A5} we see from this last inequality that 
\begin{equation*}
2\mathbb{E}\sup_{0\leq s\leq t}\left| \int_{0}^{s}(M^{\varepsilon }\left(
u_{\varepsilon }(\tau )\right) ,u_{\varepsilon }(\tau ))dW\right| \leq \frac{%
1}{2}\mathbb{E}\sup_{0\leq s\leq t}\left| u_{\varepsilon }(s)\right|
^{2}+C(T)+C\mathbb{E}\int_{0}^{t}\left| u_{\varepsilon }(\tau )\right|
^{2}d\tau .
\end{equation*}%
From this and \eqref{10} we derive that 
\begin{equation}
\mathbb{E}\sup_{0\leq \tau \leq t}\left| u_{\varepsilon }(\tau )\right|
^{2}+\Lambda \mathbb{E}\int_{0}^{t}\left\| u_{\varepsilon }(\tau )\right\|
^{2}d\tau \leq C(\left| u^{0}\right| ^{2},T)+C\mathbb{E}\int_{0}^{t}\left|
u_{\varepsilon }(\tau )\right| ^{2}d\tau .  \label{11}
\end{equation}%
Now it follows from Gronwall's inequality that 
\begin{equation}
\mathbb{E}\sup_{0\leq t\leq T}\left| u_{\varepsilon }(t)\right| ^{2}\leq C,
\label{11d}
\end{equation}%
where $C>0$ is independent of $\varepsilon $. Thanks to this last estimate
we derive from \eqref{11} that 
\begin{equation}
\mathbb{E}\int_{0}^{T}||u_{\varepsilon }(\tau )||^{2}d\tau \leq C.
\label{11dd}
\end{equation}%
As above $C>0$ does not depend on $\varepsilon $. Now let $p>2$. Thanks to It%
\^{o}'s formula we derive from (\ref{4.5}) that 
\begin{eqnarray*}
d|u_{\varepsilon }(t)|^{p} &=&-p\left( a\left( \frac{x}{\varepsilon },\frac{t%
}{\varepsilon ^{2}}\right) Du_{\varepsilon }(t),Du_{\varepsilon }(t)\right)
\left| u_{\varepsilon }(t)\right| ^{p-2}dt \\
&&+\frac{p}{\varepsilon }(g(\frac{x}{\varepsilon },\frac{t}{\varepsilon ^{2}}%
,u_{\varepsilon }(t)),u_{\varepsilon }(t))\left| u_{\varepsilon }(t)\right|
^{p-2}dt \\
&&+\frac{p}{2}\left| u_{\varepsilon }(t)\right| ^{p-2}\sum_{k=1}^{m}\left|
M_{k}^{\varepsilon }(u_{\varepsilon }(t))\right| ^{2}dt+\frac{p(p-2)}{2}%
\left| u_{\varepsilon }(t)\right| ^{p-4}(M^{\varepsilon }(u_{\varepsilon
}(t)),u_{\varepsilon }(t))^{2}dt \\
&&+p\left| u_{\varepsilon }(t)\right| ^{p-2}(M^{\varepsilon }(u_{\varepsilon
}(t)),u_{\varepsilon }(t))dW.
\end{eqnarray*}%
Thanks to \textbf{A1}, \eqref{4d}, \eqref{5} and \eqref{6} we have that 
\begin{eqnarray}
d\left| u_{\varepsilon }(t)\right| ^{p}+p\Lambda \left| u_{\varepsilon
}(t)\right| ^{p-2}\left\| u_{\varepsilon }(t)\right\| ^{2}dt &\leq &pC\left|
u_{\varepsilon }(t)\right| ^{p-1}\left\| u_{\varepsilon }(t)\right\| dt 
\notag \\
&&+\frac{p}{2}C\left| u_{\varepsilon }(t)\right| ^{p-2}(1+\left|
u_{\varepsilon }(t)\right| ^{2})dt  \notag \\
&&+\frac{p(p-2)}{4}\left| u_{\varepsilon }(t)\right| ^{p-4}(M^{\varepsilon
}(u_{\varepsilon }(t)),u_{\varepsilon }(t))^{2}dt  \notag \\
&&+p\left| u_{\varepsilon }(t)\right| ^{p-2}(M^{\varepsilon }(u_{\varepsilon
}(t)),u_{\varepsilon }(t))dW.  \label{12}
\end{eqnarray}%
Thanks to \textbf{A5} we get form easy calculations that 
\begin{equation}
\left| u_{\varepsilon }(t)\right| ^{p-4}(M^{\varepsilon }(u_{\varepsilon
}(t)),u_{\varepsilon }(t))^{2}\leq C(p)\left| u_{\varepsilon }(t)\right|
^{p}.  \label{13}
\end{equation}%
Using \eqref{13} in \eqref{12} yields 
\begin{eqnarray}
d|u_{\varepsilon }(t)|^{p}+p\Lambda \left| u_{\varepsilon }(t)\right|
^{p-2}\left\| u_{\varepsilon }(t)\right\| ^{2}dt &\leq &pC\left|
u_{\varepsilon }(t)\right| ^{p-1}\left\| u_{\varepsilon }(t)\right\|
dt+C(p)\left| u_{\varepsilon }(t)\right| ^{p}dt  \notag \\
&&+p\left| u_{\varepsilon }(t)\right| ^{p-2}(M^{\varepsilon }(u_{\varepsilon
}(t)),u_{\varepsilon }(t))dW,  \label{16}
\end{eqnarray}%
which is equivalent to 
\begin{eqnarray}
\left| u_{\varepsilon }(t)\right| ^{p}+p\Lambda \int_{0}^{t}\left|
u_{\varepsilon }(\tau )\right| ^{p-2}\left\| u_{\varepsilon }(\tau )\right\|
^{2}d\tau &\leq &p\int_{0}^{t}\left| u_{\varepsilon }(\tau )\right|
^{p-2}(M^{\varepsilon }\left( u_{\varepsilon }(\tau )\right) ,u_{\varepsilon
}(\tau ))dW  \notag \\
&&+\left| u^{0}\right| ^{p}+C(p)\int_{0}^{t}\left| u_{\varepsilon }(\tau
)\right| ^{p}d\tau  \notag \\
&&+C(p)\int_{0}^{t}\left| u_{\varepsilon }(\tau )\right| ^{p-1}\left\|
u_{\varepsilon }(\tau )\right\| d\tau .  \label{17}
\end{eqnarray}%
Due to Cauchy's inequality the second term of the right hand side of %
\eqref{17} can be estimated as follows 
\begin{equation*}
C(p)\int_{0}^{t}\left| u_{\varepsilon }(\tau )\right| ^{p-1}\left\|
u_{\varepsilon }(\tau )\right\| d\tau \leq C(p,\delta )\int_{0}^{t}\left|
u_{\varepsilon }(\tau )\right| ^{p}d\tau +\delta \int_{0}^{t}\left|
u_{\varepsilon }(\tau )\right| ^{p-2}\left\| u_{\varepsilon }(\tau )\right\|
^{2}d\tau ,
\end{equation*}%
where $\delta >0$ is arbitrary. Choosing $\delta =p\Lambda /2$ in the last
inequality and using the resulting estimate in \eqref{17} implies that 
\begin{eqnarray*}
\left| u_{\varepsilon }(t)\right| ^{p}+(p\Lambda /2)\int_{0}^{t}\left|
u_{\varepsilon }(\tau )\right| ^{p-2}\left\| u_{\varepsilon }(\tau )\right\|
^{2}d\tau &\leq &\left| u^{0}\right| ^{p}+C(p,\Lambda )\int_{0}^{t}\left|
u_{\varepsilon }(\tau )\right| ^{p}d\tau \\
&&+p\int_{0}^{t}\left| u_{\varepsilon }(\tau )\right| ^{p-2}(M^{\varepsilon
}\left( u_{\varepsilon }(\tau )\right) ,u_{\varepsilon }(\tau ))dW.
\end{eqnarray*}%
Taking the supremum over $0\leq \tau \leq t$ and the mathematical
expectation to both sides of this last inequality yields 
\begin{equation}
\begin{array}{l}
\mathbb{E}\sup_{0\leq \tau \leq t}\left| u_{\varepsilon }(\tau )\right|
^{p}+(p\Lambda /2)\mathbb{E}\int_{0}^{t}\left| u_{\varepsilon }(\tau
)\right| ^{p-2}\left\| u_{\varepsilon }(\tau )\right\| ^{2}d\tau \\ 
\leq \left| u^{0}\right| ^{p}+C(p,\Lambda )\mathbb{E}\int_{0}^{t}\left|
u_{\varepsilon }(\tau )\right| ^{p}d\tau +p\mathbb{E}\sup_{0\leq s\leq
t}\left| \int_{0}^{s}\left| u_{\varepsilon }(\tau )\right|
^{p-2}(M^{\varepsilon }(u_{\varepsilon }(\tau )),u_{\varepsilon }(\tau
))dW\right| .%
\end{array}
\label{18}
\end{equation}%
Thanks to Burkh\"{o}lder-Davis-Gundy's inequality we have that 
\begin{equation*}
\begin{array}{l}
p\mathbb{E}\sup_{0\leq s\leq t}\left| \int_{0}^{s}\left| u_{\varepsilon
}(\tau )\right| ^{p-2}(M^{\varepsilon }(u_{\varepsilon }(\tau
)),u_{\varepsilon }(\tau ))dW\right| \\ 
\leq 3p\mathbb{E}\left( \int_{0}^{t}\left| u_{\varepsilon }(\tau )\right|
^{2p-4}(M^{\varepsilon }\left( u_{\varepsilon }(\tau )\right)
,u_{\varepsilon }(\tau ))^{2}d\tau \right) ^{1/2} \\ 
\leq 3p\mathbb{E}\left( \sup_{0\leq \tau \leq t}\left| u_{\varepsilon }(\tau
)\right| ^{p/2}\int_{0}^{t}\left| u_{\varepsilon }(\tau )\right|
^{p-2}\left| M^{\varepsilon }\left( u_{\varepsilon }(\tau )\right) \right|
^{2}d\tau \right) ^{1/2}%
\end{array}%
\end{equation*}%
Thanks to Cauchy's inequality and the assumption \textbf{A5} we get that 
\begin{equation}
\begin{array}{l}
p\mathbb{E}\sup_{0\leq s\leq t}\left| \int_{0}^{s}\left| u_{\varepsilon
}(\tau )\right| ^{p-2}(M^{\varepsilon }(u_{\varepsilon }(\tau
),u_{\varepsilon }(\tau ))dW\right| \\ 
\leq 3p\delta \mathbb{E}\sup_{0\leq \tau \leq t}\left| u_{\varepsilon }(\tau
)\right| ^{p}+C(p,\delta ,T)+C(p,\delta ,T)\mathbb{E}\int_{0}^{t}\left|
u_{\varepsilon }(\tau )\right| ^{p}d\tau ,%
\end{array}
\label{19}
\end{equation}%
where $\delta >0$ is arbitrary. Using \eqref{19} in \eqref{18} yields 
\begin{equation*}
\begin{array}{l}
\mathbb{E}\sup_{0\leq \tau \leq t}\left| u_{\varepsilon }(\tau )\right|
^{p}+(p\Lambda /2)\mathbb{E}\int_{0}^{t}\left| u_{\varepsilon }(\tau
)\right| ^{p-2}\left\| u_{\varepsilon }(\tau )\right\| ^{2}d\tau \\ 
\leq C(p,\Lambda ,\delta ,T)\mathbb{E}\int_{0}^{t}\left| u_{\varepsilon
}(\tau )\right| ^{p}d\tau +\left| u^{0}\right| ^{p}+C(\delta ,T,p)+3p\delta 
\mathbb{E}\sup_{0\leq \tau \leq t}\left| u_{\varepsilon }(\tau )\right| ^{p}.%
\end{array}%
\end{equation*}%
It follows from this and by taking $\delta =1/6p$ that 
\begin{eqnarray*}
\mathbb{E}\sup_{0\leq \tau \leq t}\left| u_{\varepsilon }(\tau )\right|
^{p}+p\Lambda \mathbb{E}\int_{0}^{t}\left| u_{\varepsilon }(\tau )\right|
^{p-2}\left\| u_{\varepsilon }(\tau )\right\| ^{2}d\tau &\leq &\left|
u^{0}\right| ^{p}+C(\delta ,T,p) \\
&&+C(p,\Lambda ,\delta ,T)\mathbb{E}\int_{0}^{t}\left| u_{\varepsilon }(\tau
)\right| ^{p}d\tau .
\end{eqnarray*}%
Gronwall's Lemma implies that 
\begin{equation}
\mathbb{E}\sup_{0\leq \tau \leq T}\left| u_{\varepsilon }(\tau )\right|
^{p}\leq C,  \label{19d}
\end{equation}%
where $C>0$ is independent of $\varepsilon $. From \eqref{9} we see that 
\begin{eqnarray*}
\int_{0}^{t}\left\| u_{\varepsilon }(\tau )\right\| ^{2}d\tau &\leq
&C(\left| u^{0}\right| ^{2},T,\Lambda )+C(\Lambda )\int_{0}^{T}\left|
u_{\varepsilon }(\tau )\right| ^{2}d\tau \\
&&+C(\Lambda )\int_{0}^{t}(M^{\varepsilon }\left( u_{\varepsilon }(\tau
)\right) ,u_{\varepsilon }(\tau ))dW.
\end{eqnarray*}%
Raising both sides of this inequality to the power $p/2$ and taking the
mathematical expectation imply that 
\begin{eqnarray*}
\mathbb{E}\left( \int_{0}^{t}\left\| u_{\varepsilon }(\tau )\right\|
^{2}\right) ^{p/2}d\tau &\leq &C(\Lambda ,p)\mathbb{E}\left(
\int_{0}^{t}(M^{\varepsilon }\left( u_{\varepsilon }(\tau )\right)
,u_{\varepsilon }(\tau ))dW\right) ^{p/2} \\
&&+C(\left| u^{0}\right| ^{2},T,\Lambda ,p).
\end{eqnarray*}%
Here we have used \eqref{19d} to deal with the term $C(\Lambda ,p,T)\mathbb{E%
}\sup_{0\leq t\leq T}\left| u_{\varepsilon }(t)\right| ^{p}$. It follows
from martingale inequality and some straightforward computations that 
\begin{equation}
\mathbb{E}\left( \int_{0}^{T}\left\| u_{\varepsilon }(t)\right\|
^{2}dt\right) ^{p/2}\leq C.  \label{20}
\end{equation}%
The estimates \eqref{11d}, \eqref{11dd}, \eqref{19d} and \eqref{20} complete
the proof of the lemma.
\end{proof}

\begin{lemma}
\label{l4.2}There exists a constant $C>0$ such that 
\begin{equation*}
\mathbb{E}\sup_{|\theta |\leq \delta }\int_{0}^{T}\left| u_{\varepsilon
}(t+\theta )-u_{\varepsilon }(t)\right| _{H^{-1}(Q)}^{2}dt\leq C\delta ,
\end{equation*}%
for any $\varepsilon $, and $\delta \in (0,1)$. Here $u_{\varepsilon }(t)$
is extended to zero outside the interval $[0,T]$.
\end{lemma}

\begin{proof}
Let $\theta >0$. We have that 
\begin{eqnarray*}
u_{\varepsilon }(t+\theta )-u_{\varepsilon }(t) &=&\int_{t}^{t+\theta }{\Div}%
\left( a\left( \frac{x}{\varepsilon },\frac{\tau }{\varepsilon ^{2}}\right)
Du_{\varepsilon }(\tau )\right) d\tau +\frac{1}{\varepsilon }%
\int_{t}^{t+\theta }g\left( \frac{x}{\varepsilon },\frac{\tau }{\varepsilon
^{2}},u_{\varepsilon }\right) d\tau \\
&&+\int_{t}^{t+\theta }M^{\varepsilon }(u_{\varepsilon }(\tau ))dW,
\end{eqnarray*}%
as an equality of random variables taking values in $H^{-1}(Q)$. It follows
from this that 
\begin{eqnarray}
\left| u_{\varepsilon }(t+\theta )-u_{\varepsilon }(t)\right| _{H^{-1}(Q)}
&\leq &C\theta \int_{t}^{t+\theta }\left| \Div\left( a\left( \frac{x}{%
\varepsilon },\frac{\tau }{\varepsilon ^{2}}\right) Du_{\varepsilon }(\tau
)\right) \right| _{H^{-1}(Q)}^{2}d\tau  \notag \\
&&+C\theta \int_{t}^{t+\theta }\left| \frac{1}{\varepsilon }g\left( \frac{x}{%
\varepsilon },\frac{\tau }{\varepsilon ^{2}},u_{\varepsilon }(\tau )\right)
\right| _{H^{-1}(Q)}^{2}d\tau  \notag \\
&&+\left| \int_{t}^{t+\theta }M^{\varepsilon }\left( u_{\varepsilon }(\tau
)\right) dW\right| ^{2}.  \label{21}
\end{eqnarray}%
Firstly, 
\begin{align*}
\left| \Div\left( a\left( \frac{x}{\varepsilon },\frac{\tau }{\varepsilon
^{2}}\right) Du_{\varepsilon }(\tau )\right) \right| _{H^{-1}(Q)}& =\sup 
_{\substack{ \phi \in H_{0}^{1}(Q)  \\ \left\| \phi \right\| =1}}\left|
\left\langle \Div\left( a\left( \frac{x}{\varepsilon },\frac{t}{\varepsilon
^{2}}\right) Du_{\varepsilon }\right) ,\phi \right\rangle \right| \\
& =\sup_{\substack{ \phi \in H_{0}^{1}(Q)  \\ \left\| \phi \right\| =1}}%
\left| \int_{Q}a\left( \frac{x}{\varepsilon },\frac{t}{\varepsilon ^{2}}%
\right) Du_{\varepsilon }D\phi dx\right|
\end{align*}%
from which we derive that 
\begin{equation}
\left| \Div\left( a\left( \frac{x}{\varepsilon },\frac{t}{\varepsilon ^{2}}%
\right) Du_{\varepsilon }\right) \right| _{H^{-1}(Q)}^{2}\leq C(\Lambda
)\left\| u_{\varepsilon }\right\| ^{2},  \label{22}
\end{equation}%
where the assumption \textbf{A1} was used. Secondly, 
\begin{equation*}
\left| \frac{1}{\varepsilon }g\left( \frac{x}{\varepsilon },\frac{\tau }{%
\varepsilon ^{2}},u_{\varepsilon }\right) \right| _{H^{-1}(Q)}=\sup 
_{\substack{ \phi \in H_{0}^{1}(Q)  \\ \left\| \phi \right\| =1}}\left|
\int_{Q}G\left( \frac{x}{\varepsilon },\frac{t}{\varepsilon ^{2}}%
,u_{\varepsilon }\right) \cdot D\phi dx+\int_{Q}\left( \partial _{u}G\left( 
\frac{x}{\varepsilon },\frac{t}{\varepsilon ^{2}},u_{\varepsilon }\right)
\cdot Du_{\varepsilon }\right) \phi dx\right| .
\end{equation*}%
By using the conditions in \textbf{A4} and Poincar\'{e}'s inequality we get
that 
\begin{equation}
\left| \frac{1}{\varepsilon }g\left( \frac{x}{\varepsilon },\frac{\tau }{%
\varepsilon ^{2}},u_{\varepsilon }\right) \right| _{H^{-1}(Q)}\leq
\sup_{\phi \in H_{0}^{1}(Q),\left\| \phi \right\| =1}(C\left| u_{\varepsilon
}\right| +C\left\| u_{\varepsilon }\right\| \left| \phi \right| )\leq
C\left| u_{\varepsilon }\right| +C\left\| u_{\varepsilon }\right\|
\label{23}
\end{equation}%
Using \eqref{22} and \eqref{23} in \eqref{21} yields 
\begin{equation*}
\left| u_{\varepsilon }(t+\theta )-u_{\varepsilon }(t)\right|
_{H^{-1}(Q)}^{2}\leq C\theta \int_{t}^{t+\theta }\left\| u_{\varepsilon
}(\tau )\right\| ^{2}d\tau +C\theta \int_{t}^{t+\theta }\left|
u_{\varepsilon }(\tau )\right| ^{2}d\tau +\left| \int_{t}^{t+\theta
}M^{\varepsilon }\left( u_{\varepsilon }(\tau )\right) dW\right| ^{2},
\end{equation*}%
which implies that%
\begin{eqnarray*}
\mathbb{E}\int_{0}^{T}\sup_{0\leq \theta \leq \delta }\left| u_{\varepsilon
}(t+\theta )-u_{\varepsilon }(t)\right| _{H^{-1}(Q)}^{2}dt &\leq &C\delta 
\mathbb{E}\int_{0}^{T}\int_{t}^{t+\delta }\left\| u_{\varepsilon }(\tau
)\right\| ^{2}d\tau dt \\
&&+\mathbb{E}\int_{0}^{T}\sup_{0\leq \theta \leq \delta }\left|
\int_{t}^{t+\theta }M^{\varepsilon }\left( u_{\varepsilon }(\tau )\right)
dW\right| ^{2}dt \\
&&+C\delta \mathbb{E}\int_{0}^{T}\int_{t}^{t+\theta }\left| u_{\varepsilon
}(\tau )\right| ^{2}d\tau dt.
\end{eqnarray*}%
Thanks to Lemma \ref{l4.1} we have that 
\begin{equation*}
\mathbb{E}\int_{0}^{T}\sup_{0\leq \theta \leq \delta }\left| u_{\varepsilon
}(t+\theta )-u_{\varepsilon }(t)\right| _{H^{-1}(Q)}^{2}dt\leq C\delta +%
\mathbb{E}\int_{0}^{T}\sup_{0\leq \theta \leq \delta }\left|
\int_{t}^{t+\theta }M^{\varepsilon }\left( u_{\varepsilon }(\tau )\right)
dW\right| ^{2}dt.
\end{equation*}%
Due to Fubini's theorem and Burkh\"{o}lder-Davis-Gundy's inequality we see
from this last estimate that 
\begin{equation*}
\mathbb{E}\int_{0}^{T}\sup_{0\leq \theta \leq \delta }\left| u_{\varepsilon
}(t+\theta )-u_{\varepsilon }(t)\right| _{H^{-1}(Q)}^{2}dt\leq C\delta +%
\mathbb{E}\int_{0}^{T}\int_{t}^{t+\delta }\left| M^{\varepsilon }\left(
u_{\varepsilon }(\tau )\right) \right| ^{2}d\tau dt.
\end{equation*}%
Assumptions \textbf{A5} and Lemma \ref{l4.1} yields that 
\begin{equation*}
\mathbb{E}\int_{0}^{T}\sup_{0\leq \theta \leq \delta }\left| u_{\varepsilon
}(t+\theta )-u_{\varepsilon }(t)\right| _{H^{-1}(Q)}^{2}dt\leq C\delta ,
\end{equation*}%
where $C>0$ does not depend on $\varepsilon $ and $\delta $. By the same
argument, we can show that a similar inequality holds for negative values of 
$\theta $. This completes the proof of the lemma.
\end{proof}

The following compactness result plays a crucial role in the proof of the
tightness of the probability measures generated by the sequence $%
(u_{\varepsilon })_{\varepsilon }$.

\begin{lemma}
\label{l4.3}Let $\mu _{n}$, $\nu _{n}$ two sequences of positive real
numbers which tend to zero as $n\rightarrow \infty $, the injection of 
\begin{equation*}
\begin{array}{l}
D_{\nu _{n},\mu _{n}}:=\{q\in L^{\infty }(0,T;L^{2}(Q))\cap
L^{2}(0,T;H_{0}^{1}(Q)): \\ 
\sup_{n}\frac{1}{\nu _{n}}\sup_{\left| \theta \right| \leq \mu _{n}}\left(
\int_{0}^{T}\left| q(t+\theta )-q(t)\right| _{H^{-1}(Q)}^{2}\right)
^{1/2}<\infty \}%
\end{array}%
\end{equation*}%
in $L^{2}(Q_{T})$ is compact.
\end{lemma}

The proof, which is similar to the analogous result in \cite{bensoussan2},
follows from the application of Lemmas \ref{l4.1}, \ref{l4.2}. The space $%
D_{\nu _{n},\mu _{n}}$ is a Banach space with the norm 
\begin{eqnarray*}
\left\Vert q\right\Vert _{D_{\nu _{n},\mu _{n}}} &=&\underset{0\leq t\leq T}{%
~\text{ess}\sup }\left\vert q(t)\right\vert +\left( \int_{0}^{T}\left\Vert
q\right\Vert ^{2}dt\right) \\
&&+\sup_{n}\frac{1}{\nu _{n}}\sup_{\left\vert \theta \right\vert \leq \mu
_{n}}\left( \int_{0}^{T}\left\vert q(t+\theta )-q(t)\right\vert
_{H^{-1}(Q)}^{2}\right) ^{1/2}.
\end{eqnarray*}%
Alongside $D_{\nu _{n},\mu _{n}}$, we also consider the space $X_{p,\nu
_{n},\mu _{n}}$, $1\leq p<\infty $, of random variables $\zeta $ endowed
with the norm 
\begin{eqnarray*}
\mathbb{E}||\zeta ||_{X_{p,\nu _{n},\mu _{n}}} &=&\mathbb{E}\text{ess}%
\sup_{0\leq t\leq T}\left\vert \zeta (t)\right\vert ^{p}+\mathbb{E}\left(
\int_{0}^{T}\left\Vert \zeta (t)\right\Vert ^{2}\right) ^{p/2} \\
&&+\mathbb{E}\sup_{n}\frac{1}{\nu _{n}}\sup_{|\theta |\leq \mu _{n}}\left(
\int_{0}^{T}\left\vert \zeta (t+\theta )-\zeta (t)\right\vert
_{H^{-1}}^{2}\right) ^{1/2};
\end{eqnarray*}%
$X_{p,\nu _{n},\mu _{n}}$ is a Banach space.

Combining Lemma \ref{l4.1} and the estimates in Lemma \ref{l4.2} we have

\begin{proposition}
\label{p4.1}For any real number $p\in \lbrack 1,\infty )$ and for any
sequences $\nu _{n},\mu _{n}$ converging to $0$ such that the series $%
\sum_{n}\frac{\sqrt{\mu _{n}}}{\nu _{n}}$ converges, the sequence $%
(u_{\varepsilon })_{\varepsilon }$ is bounded uniformly in $\varepsilon $ in 
$X_{p,\nu _{n},\mu _{n}}$ for all $n$.
\end{proposition}

Next we consider the space $\mathfrak{S}=\mathcal{C}(0,T;\mathbb{R}%
^{m})\times L^{2}(Q_{T})$ equipped with the Borel $\sigma $-algebra $%
\mathcal{B}(\mathfrak{S})$. For $0<\varepsilon <1$, let $\Phi _{\varepsilon
} $ be the measurable $\mathfrak{S}$-valued mapping defined on $(\Omega ,%
\mathcal{F},\mathbb{P})$ by 
\begin{equation*}
\Phi _{\varepsilon }(\omega )=({W}(\omega ),u_{\varepsilon }(\omega )).
\end{equation*}%
For each $\varepsilon $ we introduce a probability measure $\Pi
^{\varepsilon }$ on $(\mathfrak{S};\mathcal{B}(\mathfrak{S}))$ defined by 
\begin{equation*}
\Pi ^{\varepsilon }(S)=\mathbb{P}(\Phi _{\varepsilon }^{-1}(S))\text{, for
any }S\in \mathcal{B}(\mathfrak{S}).
\end{equation*}

\begin{theorem}
\label{t4.2}The family of probability measures $\{\Pi ^{\varepsilon
}:0<\varepsilon <1\}$ is tight in $(\mathfrak{S};\mathcal{B}(\mathfrak{S}))$.
\end{theorem}

\begin{proof}
For $\delta >0$ we should find compact subsets 
\begin{equation*}
\Sigma _{\delta }\subset \mathcal{C}(0,T;\mathbb{R}^{m});Y_{\delta }\subset
L^{2}(Q_{T}),
\end{equation*}%
such that 
\begin{equation}
\mathbb{P}\left( \omega :W(\cdot ,\omega )\notin \Sigma _{\delta }\right)
\leq \frac{\delta }{2},  \label{47*}
\end{equation}%
\begin{equation}
\mathbb{P}\left( \omega :u_{\varepsilon }(\cdot ,\omega )\notin Y_{\delta
}\right) \leq \frac{\delta }{2},  \label{47**}
\end{equation}%
for all $\varepsilon $.

The quest for $\Sigma _{\delta }$ is made by taking into account some facts
about Wiener process such as the formula 
\begin{equation}
\mathbb{E}\left| W(t)-W(s)\right| ^{2j}=(2j-1)!(t-s)^{j},j=1,2,....
\label{48}
\end{equation}%
For a constant $L_{\delta }>0$ depending on $\delta $ to be fixed later and $%
n\in \mathbb{N}$, we consider the set 
\begin{equation*}
\Sigma _{\delta }=\{W(\cdot )\in \mathcal{C}(0,T;\mathbb{R}^{m}):\sup 
_{\substack{ t,s\in \lbrack 0,T]  \\ \left| t-s\right| <\frac{1}{n^{6}}}}%
n\left| W(s)-W(t)\right| \leq L_{\delta }\}.
\end{equation*}%
The set $\Sigma _{\delta }$ is relatively compact in $\mathcal{C}(0,T;%
\mathbb{R}^{m})$ by Ascoli-Arzela's theorem. Furthermore $\Sigma _{\delta }$
is closed in $\mathcal{C}(0,T;\mathbb{R}^{m})$, therefore it is compact in $%
\mathcal{C}(0,T;\mathbb{R}^{m})$. Making use of Markov's inequality 
\begin{equation*}
\mathbb{P}(\omega ;\zeta (\omega )\geq \beta )\leq \frac{1}{\beta ^{k}}%
\mathbb{E}[\left| \zeta (\omega )\right| ^{k}],
\end{equation*}%
for any random variable $\zeta $ and real numbers $k$ we get%
\begin{align*}
\mathbb{P}\left( \omega :W(\omega )\notin \Sigma _{\delta }\right) & \leq 
\mathbb{P}\left[ \cup _{n}\left\{ \omega :\sup_{\substack{ t,s\in \lbrack
0,T]  \\ |t-s|<\frac{1}{n^{6}}}}\left| W(s)-W(t)\right| \geq \frac{L_{\delta
}}{n}\right\} \right] , \\
& \leq \sum_{n=1}^{\infty }\sum_{i=0}^{n^{6}-1}\left( \frac{n}{L_{\delta }}%
\right) ^{4}\mathbb{E}\sup_{\frac{iT}{n^{6}}\leq t\leq \frac{(i+1)T}{n^{6}}%
}\left| W(t)-W(iTn^{-6}\right| ^{4}, \\
& \leq C\sum_{n=1}^{\infty }\sum_{i=0}^{n^{6}-1}\left( \frac{n}{L_{\delta }}%
\right) ^{4}(Tn^{-6})^{2}n^{6}=\frac{C}{L_{\delta }^{4}}\sum_{n=1}^{\infty }%
\frac{1}{n^{2}},
\end{align*}%
where we have used \eqref{48}. Since the right hand side of \eqref{48} is
independent of $\varepsilon $, then so is the constant $C$ in the above
estimate. We take $L_{\delta }^{4}=\frac{1}{2C\varepsilon }\left(
\sum_{n=1}^{\infty }\frac{1}{n^{2}}\right) ^{-1}$ and get \eqref{47*}.

Next we choose $Y_{\delta }$ as a ball of radius $M_{\delta }$ in $D_{\nu
_{n},\mu _{m}}$ centered at 0 and with $\nu _{n},\mu _{n}$ independent of $%
\delta $, converging to $0$ and such that the series $\sum_{n}\frac{\sqrt{%
\mu _{n}}}{\nu _{n}}$ converges, from Lemma \ref{l4.3}, $Y_{\delta }$ is a
compact subset of $L^{2}(Q_{T})$. Furthermore, we have 
\begin{eqnarray*}
\mathbb{P}\left( \omega :u_{\varepsilon }(\omega )\notin Y_{\delta }\right)
&\leq &\mathbb{P}\left( \omega :\left\| u_{\varepsilon }\right\| _{D_{\nu
_{n},\mu _{m}}}>M_{\delta }\right) \\
&\leq &\frac{1}{M_{\delta }}\left( \mathbb{E}\left\| u_{\varepsilon
}\right\| _{D_{\nu _{n},\mu _{m}}}\right) \\
&\leq &\frac{1}{M_{\delta }}\left( \mathbb{E}\left\| u_{\varepsilon
}\right\| _{X_{1,\nu _{n},\mu _{n}}}\right) \\
&\leq &\frac{C}{M_{\delta }}
\end{eqnarray*}%
where $C>0$ is independent of $\varepsilon $ (see Proposition \ref{p4.1} for
the justification.)

Choosing $M_{\delta }=2C\delta ^{-1}$, we get \eqref{47**}. From the
inequalities \eqref{47*}-\eqref{47**} we deduce that 
\begin{equation*}
\mathbb{P}\left( \omega :W(\omega )\in \Sigma _{\delta };u_{\varepsilon
}(\omega )\in Y_{\delta }\right) \geq 1-\delta ,
\end{equation*}%
for all $0<\varepsilon \leq 1$. This proves that for all $0<\varepsilon \leq
1$ 
\begin{equation*}
\Pi ^{\varepsilon }(\Sigma _{\delta }\times Y_{\delta })\geq 1-\delta ,
\end{equation*}%
from which we deduce the tightness of $\{\Pi ^{\varepsilon }:0<\varepsilon
\leq 1\}$ in $(\mathfrak{S},\mathcal{B}(\mathfrak{S}))$.
\end{proof}

Prokhorov's compactness result enables us to extract from $\left( \Pi
^{\varepsilon }\right) $ a subsequence $\left( \Pi ^{\varepsilon
_{j}}\right) $ such that 
\begin{equation*}
\Pi ^{\varepsilon _{j}}\text{ weakly converges to a probability measure }\Pi 
\text{ on }\mathfrak{S}.
\end{equation*}%
Skorokhod's theorem ensures the existence of a complete probability space $(%
\bar{\Omega},\bar{\mathcal{F}},\bar{\mathbb{P}})$ and random variables $({W}%
^{\varepsilon _{j}},u_{\varepsilon _{j}})$ and $(\bar{W},u_{0})$ defined on $%
(\bar{\Omega},\bar{\mathcal{F}},\bar{\mathbb{P}})$ with values in $\mathfrak{%
S}$ such that 
\begin{equation}
\text{The probability law of }({W}^{\varepsilon _{j}},u_{\varepsilon _{j}})%
\text{ is }\Pi ^{\varepsilon _{j}},  \label{4.6}
\end{equation}%
\begin{equation}
\text{The probability law of }(\bar{W},u_{0})\text{ is }\Pi ,\ \ \ \ \ \ \ 
\label{4.7}
\end{equation}%
\begin{equation}
{W}^{\varepsilon _{j}}\rightarrow \bar{W}\text{ in }\mathcal{C}(0,T;\mathbb{R%
}^{m})\,\,\bar{\mathbb{P}}\text{-a.s.,\ \ \ \ \ \ \ \ \ \ }  \label{4.8}
\end{equation}%
\begin{equation}
u_{\varepsilon _{j}}\rightarrow u_{0}\text{ in }L^{2}(Q_{T})\,\,\bar{\mathbb{%
P}}\text{-a.s..\ \ \ \ \ \ \ \ \ \ \ \ \ \ \ \ \ \ \ }  \label{4.9}
\end{equation}%
We can see that $\left\{ W^{\varepsilon _{j}}:\varepsilon _{j}\right\} $ is
a sequence of $m$-dimensional standard Brownian Motions. We let $\bar{%
\mathcal{F}}^{t}$ be the $\sigma $-algebra generated by $(\bar{W}%
(s),u_{0}(s)),0\leq s\leq t$ and the null sets of $\bar{\mathcal{F}}$. We
can show by arguing as in \cite{bensoussan2} that $\bar{W}$ is an $\bar{%
\mathcal{F}}^{t}$-adapted standard $\mathbb{R}^{m}$-valued Wiener process.
By the same argument as in \cite{bensoussan} we can show that 
\begin{eqnarray}
u_{\varepsilon _{j}}(t) &=&u^{0}+\int_{0}^{t}\Div\left( a\left( \frac{x}{%
\varepsilon _{j}},\frac{\tau }{\varepsilon _{j}^{2}}\right) Du_{\varepsilon
_{j}}(\tau )\right) d\tau +\frac{1}{\varepsilon _{j}}\int_{0}^{t}g\left( 
\frac{x}{\varepsilon _{j}},\frac{\tau }{\varepsilon _{j}^{2}},u_{\varepsilon
_{j}}\right) d\tau  \label{4.10} \\
&&+\int_{0}^{t}M^{\varepsilon _{j}}(u_{\varepsilon _{j}}(\tau
))dW^{\varepsilon _{j}},  \notag
\end{eqnarray}%
holds (as an equation in $H^{-1}(Q)$) for almost all $(\bar{\omega},t)\in 
\bar{\Omega}\times \lbrack 0,T]$.

\section{Homogenization results}

We assume in this section that all vector spaces are real vector spaces, and
all functions are real-valued. We keep using the same notation as in the
previous sections.

\subsection{Preliminary results}

Let $1<p<\infty $. It is a fact that the topological dual of $\mathcal{B}%
_{AP}^{p}(\mathbb{R}_{\tau };\mathcal{B}_{\#AP}^{1,p}(\mathbb{R}_{y}^{N}))$
is $\mathcal{B}_{AP}^{p^{\prime }}(\mathbb{R}_{\tau };[\mathcal{B}%
_{\#AP}^{1,p}(\mathbb{R}_{y}^{N})]^{\prime })$; this can be easily seen from
the fact that $\mathcal{B}_{\#AP}^{1,p}(\mathbb{R}_{y}^{N})$ is reflexive
(see Section 2) and $\mathcal{B}_{AP}^{p}(\mathbb{R}_{\tau };\mathcal{B}%
_{\#AP}^{1,p}(\mathbb{R}_{y}^{N}))$ is isometrically isomorphic to $L^{p}(%
\mathcal{K}_{\tau };\mathcal{B}_{\#AP}^{1,p}(\mathbb{R}_{y}^{N}))$. We
denote by $\left\langle ,\right\rangle $ (resp. $[,]$) the duality pairing
between $\mathcal{B}_{\#AP}^{1,p}(\mathbb{R}_{y}^{N})$ (resp. $\mathcal{B}%
_{AP}^{p}(\mathbb{R}_{\tau };\mathcal{B}_{\#AP}^{1,p}(\mathbb{R}_{y}^{N}))$)
and $[\mathcal{B}_{\#AP}^{1,p}(\mathbb{R}_{y}^{N})]^{\prime }$ (resp. $%
\mathcal{B}_{AP}^{p^{\prime }}(\mathbb{R}_{\tau };[\mathcal{B}_{\#AP}^{1,p}(%
\mathbb{R}_{y}^{N})]^{\prime })$). For the above reason, we have, for $u\in 
\mathcal{B}_{AP}^{p^{\prime }}(\mathbb{R}_{\tau };[\mathcal{B}_{\#AP}^{1,p}(%
\mathbb{R}_{y}^{N})]^{\prime })$ and $v\in \mathcal{B}_{AP}^{p}(\mathbb{R}%
_{\tau };\mathcal{B}_{\#AP}^{1,p}(\mathbb{R}_{y}^{N}))$, 
\begin{equation*}
\left[ u,v\right] =\int_{\mathcal{K}_{\tau }}\left\langle \widehat{u}(s_{0}),%
\widehat{v}(s_{0})\right\rangle d\beta _{\tau }(s_{0}).
\end{equation*}%
For a function $\psi \in \mathcal{D}_{AP}(\mathbb{R}_{y}^{N})/\mathbb{C}$ we
know that $\psi $ expresses as follows: $\psi =\varrho _{y}(\psi _{1})$ with 
$\psi _{1}\in AP^{\infty }(\mathbb{R}_{y}^{N})/\mathbb{C}$ where $\varrho
_{y}$ denotes the canonical mapping of $B_{AP}^{p}(\mathbb{R}_{y}^{N})$ onto 
$\mathcal{B}_{AP}^{p}(\mathbb{R}_{y}^{N})$; see Section 2. We will refer to $%
\psi _{1}$ as the representative of $\psi $ in $AP^{\infty }(\mathbb{R}%
_{y}^{N})/\mathbb{C}$. Likewise we define the representative of $\psi \in 
\mathcal{D}_{AP}(\mathbb{R}_{\tau })\otimes \lbrack \mathcal{D}_{AP}(\mathbb{%
R}_{y}^{N})/\mathbb{C}]$ as an element of $AP^{\infty }(\mathbb{R}_{\tau
})\otimes \lbrack AP^{\infty }(\mathbb{R}_{y}^{N})/\mathbb{C}]$ satisfying a
similar property.

With all this in mind, we have the following

\begin{lemma}
\label{l5.1}Let $\psi \in B(\bar{\Omega})\otimes \mathcal{C}_{0}^{\infty
}(Q_{T})\otimes (\mathcal{D}_{AP}(\mathbb{R}_{\tau })\otimes \lbrack 
\mathcal{D}_{AP}(\mathbb{R}_{y}^{N})/\mathbb{C}])$ and $\psi _{1}$ be its
representative in $B(\bar{\Omega})\otimes \mathcal{C}_{0}^{\infty
}(Q_{T})\otimes \lbrack AP^{\infty }(\mathbb{R}_{\tau })\otimes (AP^{\infty
}(\mathbb{R}_{y}^{N})/\mathbb{C})]$. Let $(u_{\varepsilon })_{\varepsilon
\in E}$, $E^{\prime }$ and $(u_{0},u_{1})$ be either as in Theorem \emph{\ref%
{t3.2}} or as in Theorem \emph{\ref{t3.3}}. Then, as $E^{\prime }\ni
\varepsilon \rightarrow 0$ 
\begin{equation*}
\int_{Q_{T}\times \bar{\Omega}}\frac{1}{\varepsilon }u_{\varepsilon }\psi
_{1}^{\varepsilon }dxdtd\bar{\mathbb{P}}\rightarrow \int_{Q_{T}\times \bar{%
\Omega}}\left[ u_{1}(x,t,\omega ),\psi (x,t,\omega )\right] dxdtd\bar{%
\mathbb{P}}.
\end{equation*}
\end{lemma}

\begin{proof}
We recall that for $\psi _{1}$ as above, we have 
\begin{equation*}
\psi _{1}^{\varepsilon }(x,t,\omega )=\psi _{1}\left( x,t,\frac{x}{%
\varepsilon },\frac{t}{\varepsilon ^{2}},\omega \right) \text{ for }%
(x,t,\omega )\in Q_{T}\times \bar{\Omega}\text{.}
\end{equation*}%
This being so, since $\psi _{1}(x,t,\cdot ,\tau ,\omega )\in AP^{\infty }(%
\mathbb{R}_{y}^{N})/\mathbb{C}=\{u\in AP^{\infty }(\mathbb{R}%
_{y}^{N}):M_{y}(u)=0\}$, there exists a unique $\phi \in B(\bar{\Omega}%
)\otimes \mathcal{C}_{0}^{\infty }(Q_{T})\otimes \lbrack AP^{\infty }(%
\mathbb{R}_{\tau })\otimes (AP^{\infty }(\mathbb{R}_{y}^{N})/\mathbb{C})]$
such that $\psi _{1}=\Delta _{y}\phi $. We therefore have 
\begin{eqnarray*}
\int_{Q_{T}\times \bar{\Omega}}\frac{1}{\varepsilon }u_{\varepsilon }\psi
_{1}^{\varepsilon }dxdtd\bar{\mathbb{P}} &=&\int_{Q_{T}\times \bar{\Omega}}%
\frac{1}{\varepsilon }u_{\varepsilon }(\Delta _{y}\phi )^{\varepsilon }dxdtd%
\bar{\mathbb{P}} \\
&=&-\int_{Q_{T}\times \bar{\Omega}}Du_{\varepsilon }\cdot (D_{y}\phi
)^{\varepsilon }dxdtd\bar{\mathbb{P}} \\
&&-\int_{Q_{T}\times \bar{\Omega}}u_{\varepsilon }(\text{div}_{x}(D_{y}\phi
))^{\varepsilon }dxdtd\bar{\mathbb{P}}.
\end{eqnarray*}%
Passing to the limit in the above equation as $E^{\prime }\ni \varepsilon
\rightarrow 0$ we are led to 
\begin{eqnarray*}
\int_{Q_{T}\times \bar{\Omega}}\frac{1}{\varepsilon }u_{\varepsilon }\psi
_{1}^{\varepsilon }dxdtd\bar{\mathbb{P}} &\rightarrow &-\iint_{Q_{T}\times 
\bar{\Omega}\times \mathcal{K}}(Du_{0}+\partial \widehat{u}_{1})\cdot
\partial \widehat{\phi }dxdtd\bar{\mathbb{P}}d\beta \\
&&-\iint_{Q_{T}\times \bar{\Omega}\times \mathcal{K}}u_{0}~\text{div}%
_{x}(\partial \widehat{\phi })dxdtd\bar{\mathbb{P}}d\beta \\
&=&-\iint_{Q_{T}\times \bar{\Omega}\times \mathcal{K}}\partial \widehat{u}%
_{1}\cdot \partial \widehat{\phi }dxdtd\bar{\mathbb{P}}d\beta
\end{eqnarray*}%
since $\iint_{Q_{T}\times \bar{\Omega}\times \mathcal{K}}u_{0}\Div%
_{x}(\partial \widehat{\phi })dxdtd\bar{\mathbb{P}}d\beta
=-\iint_{Q_{T}\times \bar{\Omega}\times \mathcal{K}}Du_{0}\cdot \partial 
\widehat{\phi }dxdtd\bar{\mathbb{P}}d\beta $. But 
\begin{eqnarray*}
&&-\iint_{Q_{T}\times \bar{\Omega}\times \mathcal{K}}\partial \widehat{u}%
_{1}\cdot \partial \widehat{\phi }dxdtd\bar{\mathbb{P}}d\beta \\
&=&\int_{Q_{T}\times \bar{\Omega}}\left[ \int_{\mathcal{K}_{\tau }}\left(
-\int_{\mathcal{K}_{y}}\partial \widehat{u}_{1}(x,t,s,s_{0},\omega )\cdot
\partial \widehat{\phi }(x,t,s,s_{0},\omega )d\beta _{y}\right) d\beta
_{\tau }\right] dxdtd\bar{\mathbb{P}}.
\end{eqnarray*}

Recalling the definition of the Laplacian $\overline{\Delta }_{y}$ in
Section 2, we deduce from (\ref{2.7}) and Proposition \ref{p2.5} that 
\begin{eqnarray*}
&&-\int_{\mathcal{K}_{y}}\partial \widehat{u}_{1}(x,t,s,s_{0},\omega )\cdot
\partial \widehat{\phi }(x,t,s,s_{0},\omega )d\beta _{y} \\
&=&\left\langle \overline{\Delta }_{y}\varrho _{y}(\widehat{\phi }(x,t,\cdot
,s_{0},\omega )),\widehat{u}_{1}(x,t,\cdot ,s_{0},\omega )\right\rangle \\
&=&\left\langle \varrho _{y}(\Delta _{y}\widehat{\phi }(x,t,\cdot
,s_{0},\omega )),\widehat{u}_{1}(x,t,\cdot ,s_{0},\omega )\right\rangle \\
&=&\left\langle \widehat{\varrho _{y}(\Delta _{y}\phi )}(x,t,\cdot
,s_{0},\omega )),\widehat{u}_{1}(x,t,\cdot ,s_{0},\omega )\right\rangle \\
&=&\left\langle \widehat{\psi }(x,t,\cdot ,s_{0},\omega )),\widehat{u}%
_{1}(x,t,\cdot ,s_{0},\omega )\right\rangle
\end{eqnarray*}%
where from the first of the above series of equalities, the hat $\widehat{.}$
stands for the Gelfand transform with respect to $AP(\mathbb{R}_{\tau })$
and so, does not act on $\Delta _{y}$ and $\varrho _{y}$. The lemma
therefore follows from the equalities 
\begin{eqnarray*}
&&\int_{\mathcal{K}_{\tau }}\left( -\int_{\mathcal{K}_{y}}\partial \widehat{u%
}_{1}(x,t,s,s_{0},\omega )\cdot \partial \widehat{\phi }(x,t,s,s_{0},\omega
)d\beta _{y}\right) d\beta _{\tau } \\
&=&\int_{\mathcal{K}_{\tau }}\left\langle \widehat{\psi }(x,t,\cdot
,s_{0},\omega )),\widehat{u}_{1}(x,t,\cdot ,s_{0},\omega )\right\rangle
d\beta _{\tau }(s_{0}) \\
&=&\left[ \psi (x,t,\cdot ,\cdot ,\omega )),u_{1}(x,t,\cdot ,\cdot ,\omega )%
\right] .
\end{eqnarray*}
\end{proof}

For $u\in \mathcal{B}_{AP}^{p}(\mathbb{R}_{\tau })$ we denote by $\overline{%
\partial }/\partial \tau $ the temporal derivative defined exactly as its
spatial counterpart $\overline{\partial }/\partial y_{i}$. We also put $%
\partial _{0}=\mathcal{G}_{1}(\overline{\partial }/\partial \tau )$. $%
\overline{\partial }/\partial \tau $ and $\partial _{0}$ enjoy the same
properties as $\overline{\partial }/\partial y_{i}$ (see Section 2). In
particular, they are skew adjoint. Now, let us view $\overline{\partial }%
/\partial \tau $ as an unbounded operator defined from $\mathcal{V}=\mathcal{%
B}_{AP}^{p}(\mathbb{R}_{\tau };\mathcal{B}_{\#AP}^{1,p}(\mathbb{R}_{y}^{N}))$
into $\mathcal{V}^{\prime }=\mathcal{B}_{AP}^{p^{\prime }}(\mathbb{R}_{\tau
};[\mathcal{B}_{\#AP}^{1,p}(\mathbb{R}_{y}^{N})]^{\prime })$. Proceeding as
in \cite[pp. 1243-1244]{EfendievPankov}, it gives rise to an unbounded
operator still denoted by $\overline{\partial }/\partial \tau $ with the
following properties:

\begin{itemize}
\item[(P)$_{1}$] The domain of $\overline{\partial }/\partial \tau $ is $%
\mathcal{W}=\left\{ v\in \mathcal{V}:\overline{\partial }v/\partial \tau \in 
\mathcal{V}^{\prime }\right\} $;

\item[(P)$_{2}$] $\overline{\partial }/\partial \tau $ is skew adjoint, that
is, for all $u,v\in \mathcal{W}$, 
\begin{equation*}
\left[ u,\frac{\overline{\partial }v}{\partial \tau }\right] =-\left[ \frac{%
\overline{\partial }u}{\partial \tau },v\right] .
\end{equation*}

\item[(P)$_{3}$] The space $\mathcal{E}=\mathcal{D}_{AP}(\mathbb{R}_{\tau
})\otimes \lbrack \mathcal{D}_{AP}(\mathbb{R}_{y}^{N})/\mathbb{C}]$ is dense
in $\mathcal{W}$.
\end{itemize}

The above operator will be useful in the homogenization process. This being
so, the preceding lemma has a crucial corollary.

\begin{corollary}
\label{c5.1}Let the hypotheses be those of Lemma \emph{\ref{l5.1}}. Assume
moreover that $u_{1}\in \mathcal{W}$. Then, as $E^{\prime }\ni \varepsilon
\rightarrow 0$,%
\begin{equation*}
\int_{Q_{T}\times \bar{\Omega}}\varepsilon u_{\varepsilon }\frac{\partial
\psi _{1}^{\varepsilon }}{\partial t}dxdtd\bar{\mathbb{P}}\rightarrow
-\int_{Q_{T}\times \bar{\Omega}}\left[ \frac{\overline{\partial }u_{1}}{%
\partial \tau }(x,t,\omega ),\psi (x,t,\omega )\right] dxdtd\bar{\mathbb{P}}.
\end{equation*}
\end{corollary}

\begin{proof}
We have 
\begin{eqnarray*}
\int_{Q_{T}\times \bar{\Omega}}\varepsilon u_{\varepsilon }\frac{\partial
\psi _{1}^{\varepsilon }}{\partial t}dxdtd\bar{\mathbb{P}} &=&\varepsilon
\int_{Q_{T}\times \bar{\Omega}}u_{\varepsilon }\left( \frac{\partial \psi
_{1}}{\partial t}\right) ^{\varepsilon }dxdtd\bar{\mathbb{P}} \\
&&+\frac{1}{\varepsilon }\int_{Q_{T}\times \bar{\Omega}}u_{\varepsilon
}\left( \frac{\partial \psi _{1}}{\partial \tau }\right) ^{\varepsilon }dxdtd%
\bar{\mathbb{P}}.
\end{eqnarray*}%
Since $\frac{\partial \psi _{1}}{\partial \tau }$ is a representative of
some function in $B(\bar{\Omega})\otimes \mathcal{C}_{0}^{\infty
}(Q_{T})\otimes (\mathcal{D}_{AP}(\mathbb{R}_{\tau })\otimes \lbrack 
\mathcal{D}_{AP}(\mathbb{R}_{y}^{N})/\mathbb{C}])$, we infer from Lemma \ref%
{l5.1} that, as $E^{\prime }\ni \varepsilon \rightarrow 0$, 
\begin{eqnarray*}
&&\int_{Q_{T}\times \bar{\Omega}}\varepsilon u_{\varepsilon }\frac{\partial
\psi _{1}^{\varepsilon }}{\partial t}dxdtd\bar{\mathbb{P}} \\
&\rightarrow &\int_{Q_{T}\times \bar{\Omega}}\left[ \int_{\mathcal{K}_{\tau
}}\left\langle \widehat{u}_{1}(x,t,\cdot ,s_{0},\omega ),\partial _{0}%
\widehat{\psi }(x,t,\cdot ,s_{0},\omega ))\right\rangle d\beta _{\tau
}(s_{0})\right] dxdtd\bar{\mathbb{P}}.
\end{eqnarray*}%
But 
\begin{eqnarray*}
&&\int_{\mathcal{K}_{\tau }}\left\langle \widehat{u}_{1}(x,t,\cdot
,s_{0},\omega ),\partial _{0}\widehat{\psi }(x,t,\cdot ,s_{0},\omega
))\right\rangle d\beta _{\tau }(s_{0}) \\
&=&\left[ u_{1}(x,t,\cdot ,\cdot ,\omega ),\frac{\overline{\partial }\psi }{%
\partial \tau }(x,t,\cdot ,\cdot ,\omega ))\right] \\
&=&-\left[ \frac{\overline{\partial }u_{1}}{\partial \tau }(x,t,\cdot ,\cdot
,\omega ),\psi (x,t,\cdot ,\cdot ,\omega ))\right] ,
\end{eqnarray*}%
the last equality coming from the fact that $\overline{\partial }/\partial
\tau $ is skew adjoint.
\end{proof}

We will also need the following

\begin{lemma}
\label{l5.2}Let $g:\mathbb{R}_{y}^{N}\times \mathbb{R}_{\tau }\times \mathbb{%
R}_{u}\rightarrow \mathbb{R}$ be a function verifying the following
conditions:

\begin{itemize}
\item[(i)] $\left\vert \partial _{u}g(y,\tau ,u)\right\vert \leq C$

\item[(ii)] $g(\cdot ,\cdot ,u)\in AP(\mathbb{R}_{y,\tau }^{N+1})$.
\end{itemize}

\noindent Let $(u_{\varepsilon })_{\varepsilon }$ be a sequence in $%
L^{2}(Q_{T}\times \bar{\Omega})$ such that $u_{\varepsilon }\rightarrow
u_{0} $ in $L^{2}(Q_{T}\times \bar{\Omega})$ as $\varepsilon \rightarrow 0$
where $u_{0}\in L^{2}(Q_{T}\times \bar{\Omega})$. Then, setting $%
g^{\varepsilon }(u_{\varepsilon })(x,t,\omega )=g(x/\varepsilon
,t/\varepsilon ^{2},u_{\varepsilon }(x,t,\omega ))$ we have, as $\varepsilon
\rightarrow 0$, 
\begin{equation*}
g^{\varepsilon }(u_{\varepsilon })\rightarrow g(\cdot ,\cdot ,u_{0})\text{
in }L^{2}(Q_{T}\times \bar{\Omega})\text{-weak }\Sigma .
\end{equation*}
\end{lemma}

\begin{proof}
Assumption (i) implies the Lipschitz condition 
\begin{equation}
\left| g(y,\tau ,u)-g(y,\tau ,v)\right| \leq C\left| u-v\right| \text{\ for
all }y,\tau ,u,v.  \label{5.0}
\end{equation}%
Next, observe that from (ii) and (\ref{5.0}), the function $(x,t,y,\tau
,\omega )\mapsto g(y,\tau ,u_{0}(x,t,\omega ))$ lies in $L^{2}(Q_{T}\times 
\bar{\Omega};AP(\mathbb{R}_{y,\tau }^{N+1}))$, so that by Remark \ref{r3.0},
we have $g^{\varepsilon }(u_{0})\rightarrow g(\cdot ,\cdot ,u_{0})$ in $%
L^{2}(Q_{T}\times \bar{\Omega})$-weak $\Sigma $ as $\varepsilon \rightarrow
0 $. Now, let $f\in B(\bar{\Omega};L^{2}(Q_{T};AP(\mathbb{R}_{y,\tau
}^{N+1}))) $; then 
\begin{eqnarray*}
&&\int_{Q_{T}\times \bar{\Omega}}g^{\varepsilon }(u_{\varepsilon
})f^{\varepsilon }dxdtd\bar{\mathbb{P}}-\iint_{Q_{T}\times \bar{\Omega}%
\times \mathcal{K}}\widehat{g}(\cdot ,\cdot ,u_{0})\widehat{f}dxdtd\bar{%
\mathbb{P}}d\beta \\
&=&\int_{Q_{T}\times \bar{\Omega}}(g^{\varepsilon }(u_{\varepsilon
})-g^{\varepsilon }(u_{0}))f^{\varepsilon }dxdtd\bar{\mathbb{P}}%
+\int_{Q_{T}\times \bar{\Omega}}g^{\varepsilon }(u_{0})f^{\varepsilon }dxdtd%
\bar{\mathbb{P}} \\
&&-\iint_{Q_{T}\times \bar{\Omega}\times \mathcal{K}}\widehat{g}(\cdot
,\cdot ,u_{0})\widehat{f}dxdtd\bar{\mathbb{P}}d\beta .
\end{eqnarray*}%
Using the inequality 
\begin{equation*}
\left| \int_{Q_{T}\times \bar{\Omega}}(g^{\varepsilon }(u_{\varepsilon
})-g^{\varepsilon }(u_{0}))f^{\varepsilon }dxdtd\bar{\mathbb{P}}\right| \leq
C\left\| u_{\varepsilon }-u_{0}\right\| _{L^{2}(Q_{T}\times \bar{\Omega}%
)}\left\| f^{\varepsilon }\right\| _{L^{2}(Q_{T}\times \bar{\Omega})}
\end{equation*}%
in conjunction with the above convergence results leads at once to the
result.
\end{proof}

\begin{remark}
\label{r5.1}\emph{From the Lipschitz property of the function }$g$\emph{\
above we may get more information on the limit of the sequence }$%
g^{\varepsilon }(u_{\varepsilon })$\emph{. Indeed, since }$\left|
g^{\varepsilon }(u_{\varepsilon })-g^{\varepsilon }(u_{0})\right| \leq
C\left| u_{\varepsilon }-u_{0}\right| $\emph{, we deduce the following
convergence result: }%
\begin{equation*}
g^{\varepsilon }(u_{\varepsilon })\rightarrow \widetilde{g}(u_{0})\text{\ in 
}L^{2}(Q_{T}\times \bar{\Omega})\text{ as }\varepsilon \rightarrow 0
\end{equation*}%
\emph{where }$\widetilde{g}(u_{0})(x,t,\omega )=\int_{\mathcal{K}}\widehat{g}%
(s,s_{0},u_{0}(x,t,\omega ))d\beta $\emph{, so that we can derive the
existence of a subsequence of }$g^{\varepsilon }(u_{\varepsilon })$\emph{\
that converges a.e. in }$Q_{T}\times \bar{\Omega}$\emph{\ to }$\widetilde{g}%
(u_{0})$\emph{.}
\end{remark}

\bigskip We will need the following spaces: 
\begin{equation*}
\mathbb{F}_{0}^{1}=L^{2}(\bar{\Omega}\times \left( 0,T\right)
;H_{0}^{1}(Q))\times L^{2}(Q_{T}\times \bar{\Omega};\mathcal{W})
\end{equation*}%
and 
\begin{equation*}
\mathcal{F}_{0}^{\infty }=[B(\bar{\Omega})\otimes \mathcal{C}_{0}^{\infty
}(Q_{T})]\times \lbrack B(\bar{\Omega})\otimes \mathcal{C}_{0}^{\infty
}(Q_{T})\otimes \mathcal{E}]
\end{equation*}%
where $\mathcal{W}=\left\{ v\in \mathcal{V}:\overline{\partial }v/\partial
\tau \in \mathcal{V}^{\prime }\right\} $ with $\mathcal{V}=\mathcal{B}%
_{AP}^{2}(\mathbb{R}_{\tau };\mathcal{B}_{\#AP}^{1,2}(\mathbb{R}_{y}^{N}))$,
and $\mathcal{E}=\mathcal{D}_{AP}(\mathbb{R}_{\tau })\otimes \lbrack 
\mathcal{D}_{AP}(\mathbb{R}_{y}^{N})/\mathbb{C}]$. $\mathbb{F}_{0}^{1}$ is a
Hilbert space under the norm 
\begin{equation*}
\left\Vert (u_{0},u_{1})\right\Vert _{\mathbb{F}_{0}^{1}}=\left\Vert
u_{0}\right\Vert _{L^{2}(\bar{\Omega}\times \left( 0,T\right)
;H_{0}^{1}(Q))}+\left\Vert u_{1}\right\Vert _{L^{2}(Q_{T}\times \bar{\Omega};%
\mathcal{W})}.
\end{equation*}%
Moreover, since $B(\bar{\Omega})$ is dense in $L^{2}(\bar{\Omega})$, it is
an easy matter to check that $\mathcal{F}_{0}^{\infty }$ is dense in $%
\mathbb{F}_{0}^{1}$.

\subsection{Global homogenized problem}

Let $(u_{\varepsilon _{j}})$ be the sequence determined in Section 4 and
satisfying Eq. (\ref{4.10}). It therefore satisfies the a priori estimates (%
\ref{4.3})-(\ref{4.4}), so that, by the diagonal process, one can find a
subsequence of $(u_{\varepsilon _{j}})_{j}$ not relabeled, which weakly
converges in $L^{2}(\bar{\Omega};L^{2}(0,T;H_{0}^{1}(Q)))$ to $u_{0}$
determined by the Skorokhod's theorem and satisfying (\ref{4.9}). From
Theorem \ref{t3.3}, we infer the existence of a function $u_{1}\in L^{2}(%
\bar{\Omega};L^{2}(Q_{T};\mathcal{B}_{AP}^{2}(\mathbb{R}_{\tau };\mathcal{B}%
_{\#AP}^{1,2}(\mathbb{R}_{y}^{N}))))$ such that the convergence results 
\begin{equation}
u_{\varepsilon _{j}}\rightarrow u_{0}\text{ in }L^{2}(Q_{T})\text{ almost
surely}  \label{5.2}
\end{equation}%
and 
\begin{equation}
\frac{\partial u_{\varepsilon _{j}}}{\partial x_{i}}\rightarrow \frac{%
\partial u_{0}}{\partial x_{i}}+\frac{\overline{\partial }u_{1}}{\partial
y_{i}}\text{ in }L^{2}(Q_{T}\times \bar{\Omega})\text{-weak }\Sigma \text{ }%
(1\leq i\leq N)  \label{5.3}
\end{equation}%
hold when $\varepsilon _{j}\rightarrow 0$. The following result holds.

\begin{proposition}
\label{p5.1}The couple $(u_{0},u_{1})\in \mathbb{F}_{0}^{1}$ determined
above solves the following variational problem 
\begin{equation}
\left\{ 
\begin{array}{l}
-\int_{Q_{T}\times \bar{\Omega}}u_{0}\psi _{0}^{\prime }dxdtd\bar{\mathbb{P}}%
+\int_{Q_{T}\times \bar{\Omega}}\left[ \frac{\overline{\partial }u_{1}}{%
\partial \tau },\psi _{1}\right] dxdtd\bar{\mathbb{P}} \\ 
=-\iint_{Q_{T}\times \bar{\Omega}\times \mathcal{K}}\widehat{a}%
(Du_{0}+\partial \widehat{u}_{1})\cdot (D\psi _{0}+\partial \widehat{\psi }%
_{1})dxdtd\bar{\mathbb{P}}d\beta \\ 
+\iint_{Q_{T}\times \bar{\Omega}\times \mathcal{K}}\widehat{g}(s,s_{0},u_{0})%
\widehat{\psi }_{1}dxdtd\bar{\mathbb{P}}d\beta -\iint_{Q_{T}\times \bar{%
\Omega}\times \mathcal{K}}\widehat{G}(s,s_{0},u_{0})\cdot D\psi _{0}dxdtd%
\bar{\mathbb{P}}d\beta \\ 
-\iint_{Q_{T}\times \bar{\Omega}\times \mathcal{K}}\left( \widehat{\partial
_{u}G}(s,s_{0},u_{0})\cdot (Du_{0}+\partial \widehat{u}_{1})\right) \psi
_{0}dxdtd\bar{\mathbb{P}}d\beta \\ 
+\iint_{Q_{T}\times \bar{\Omega}\times \mathcal{K}}\widehat{M}%
(s,s_{0},u_{0})\psi _{0}d\bar{W}dxd\bar{\mathbb{P}}d\beta \text{\ \ for all }%
(\psi _{0},\psi _{1})\in \mathcal{F}_{0}^{\infty }\text{.}%
\end{array}%
\right.  \label{5.1}
\end{equation}
\end{proposition}

\begin{proof}
In what follows, we omit the index $j$ momentarily from the sequence $%
\varepsilon _{j}$. So we will merely write $\varepsilon $ instead of $%
\varepsilon _{j}$. With this in mind, we set 
\begin{equation*}
\Phi _{\varepsilon }(x,t,\omega )=\psi _{0}(x,t,\omega )+\varepsilon \psi
\left( x,t,\frac{x}{\varepsilon },\frac{t}{\varepsilon ^{2}},\omega \right) 
\text{,\ }(x,t,\omega )\in Q_{T}\times \bar{\Omega}
\end{equation*}%
where $(\psi _{0},\psi _{1})\in \mathcal{F}_{0}^{\infty }$ with $\psi $
being a representative of $\psi _{1}$. Using $\Phi _{\varepsilon }$ as a
test function in the variational formulation of (\ref{4.10}) we get 
\begin{eqnarray}
-\int_{Q_{T}\times \bar{\Omega}}u_{\varepsilon }\frac{\partial \Phi
_{\varepsilon }}{\partial t}dxdtd\bar{\mathbb{P}} &=&-\int_{Q_{T}\times \bar{%
\Omega}}a^{\varepsilon }Du_{\varepsilon }\cdot D\Phi _{\varepsilon }dxdtd%
\bar{\mathbb{P}}  \label{5.4} \\
&&+\frac{1}{\varepsilon }\int_{Q_{T}\times \bar{\Omega}}g^{\varepsilon
}(u_{\varepsilon })\Phi _{\varepsilon }dxdtd\bar{\mathbb{P}}  \notag \\
&&\mathbb{+}\int_{Q_{T}\times \bar{\Omega}}M^{\varepsilon }(u_{\varepsilon
})\Phi _{\varepsilon }dxdW^{\varepsilon }d\bar{\mathbb{P}}  \notag
\end{eqnarray}%
where here and henceforth, we use the notation $a^{\varepsilon
}=a(x/\varepsilon ,t/\varepsilon ^{2})$, $\psi ^{\varepsilon }=\psi
(x,t,x/\varepsilon ,t/\varepsilon ^{2},\omega )$, $M^{\varepsilon
}(u_{\varepsilon })=M(x/\varepsilon ,t/\varepsilon ^{2},u_{\varepsilon })$
and $g^{\varepsilon }(u_{\varepsilon })=g(x/\varepsilon ,t/\varepsilon
^{2},u_{\varepsilon })$. We will consider the terms in (\ref{5.4})
respectively.

We have 
\begin{eqnarray*}
\frac{1}{\varepsilon }\int_{Q_{T}\times \bar{\Omega}}g^{\varepsilon
}(u_{\varepsilon })\Phi _{\varepsilon }dxdtd\bar{\mathbb{P}} &=&\frac{1}{%
\varepsilon }\int_{Q_{T}\times \bar{\Omega}}g^{\varepsilon }(u_{\varepsilon
})\psi _{0}dxdtd\bar{\mathbb{P}} \\
&&+\int_{Q_{T}\times \bar{\Omega}}g^{\varepsilon }(u_{\varepsilon })\psi
^{\varepsilon }dxdtd\bar{\mathbb{P}} \\
&=&I_{\varepsilon }^{1}+I_{\varepsilon }^{2}\text{.}
\end{eqnarray*}%
Lemma \ref{l5.2} and convergence result (\ref{5.2}) imply 
\begin{equation*}
I_{\varepsilon }^{2}\rightarrow \iint_{Q_{T}\times \bar{\Omega}\times 
\mathcal{K}}\widehat{g}(s,s_{0},u_{0})\widehat{\psi }_{1}dxdtd\bar{\mathbb{P}%
}d\beta
\end{equation*}%
since $\widehat{\psi }_{1}=\mathcal{G}_{1}\circ \psi _{1}=\mathcal{G}%
_{1}\circ (\varrho (\psi ))=\mathcal{G}\circ \psi =\widehat{\psi }$. For $%
I_{\varepsilon }^{1}$, we know from assumption \textbf{A4} that 
\begin{equation*}
\frac{1}{\varepsilon }g\left( \frac{x}{\varepsilon },\frac{t}{\varepsilon
^{2}},u_{\varepsilon }\right) =\text{div}_{x}\left[ G\left( \frac{x}{%
\varepsilon },\frac{t}{\varepsilon ^{2}},u_{\varepsilon }\right) \right]
-\partial _{u}G\left( \frac{x}{\varepsilon },\frac{t}{\varepsilon ^{2}}%
,u_{\varepsilon }\right) \cdot Du_{\varepsilon },
\end{equation*}%
in such a way that 
\begin{equation*}
I_{\varepsilon }^{1}=-\int_{Q_{T}\times \bar{\Omega}}G^{\varepsilon
}(u_{\varepsilon })\cdot D\psi _{0}dxdtd\bar{\mathbb{P}}-\int_{Q_{T}\times 
\bar{\Omega}}\left[ \partial _{u}G\left( \frac{x}{\varepsilon },\frac{t}{%
\varepsilon ^{2}},u_{\varepsilon }\right) \cdot Du_{\varepsilon }\right]
\psi _{0}dxdtd\bar{\mathbb{P}}.
\end{equation*}%
Once again, owing to assumption \textbf{A4} (see the inequalities (\ref{4.0}%
) and (\ref{4.0'}) therein) we deduce from Lemma \ref{l5.2}, convergence
results (\ref{5.2}) and (\ref{5.3}) that 
\begin{eqnarray*}
I_{\varepsilon }^{1} &\rightarrow &-\iint_{Q_{T}\times \bar{\Omega}\times 
\mathcal{K}}\widehat{G}(s,s_{0},u_{0})\cdot D\psi _{0}dxdtd\bar{\mathbb{P}}%
d\beta \\
&&-\iint_{Q_{T}\times \bar{\Omega}\times \mathcal{K}}\left[ \widehat{%
\partial _{u}G}(s,s_{0},u_{0})\cdot (Du_{0}+\partial \widehat{u}_{1})\right]
\psi _{0}dxdtd\bar{\mathbb{P}}d\beta .
\end{eqnarray*}%
Next, we have 
\begin{equation*}
\int_{Q_{T}\times \bar{\Omega}}u_{\varepsilon }\frac{\partial \Phi
_{\varepsilon }}{\partial t}dxdtd\bar{\mathbb{P}}=\int_{Q_{T}\times \bar{%
\Omega}}u_{\varepsilon }\frac{\partial \psi _{0}}{\partial t}dxdtd\bar{%
\mathbb{P}}+\int_{Q_{T}\times \bar{\Omega}}\varepsilon u_{\varepsilon }\frac{%
\partial \psi ^{\varepsilon }}{\partial t}dxdtd\bar{\mathbb{P}}
\end{equation*}%
which, from Corollary \ref{c5.1} leads to 
\begin{eqnarray*}
\int_{Q_{T}\times \bar{\Omega}}u_{\varepsilon }\frac{\partial \Phi
_{\varepsilon }}{\partial t}dxdtd\bar{\mathbb{P}} &\rightarrow
&\int_{Q_{T}\times \bar{\Omega}}u_{0}\frac{\partial \psi _{0}}{\partial t}%
dxdtd\bar{\mathbb{P}} \\
&&-\int_{Q_{T}\times \bar{\Omega}}\left[ \frac{\overline{\partial }u_{1}}{%
\partial \tau }(x,t,\omega ),\psi _{1}(x,t,\omega )\right] dxdtd\bar{\mathbb{%
P}}.
\end{eqnarray*}%
It is an easy exercise to see, using Corollary \ref{c3.4} that 
\begin{equation*}
\int_{Q_{T}\times \bar{\Omega}}a^{\varepsilon }Du_{\varepsilon }\cdot D\Phi
_{\varepsilon }dxdtd\bar{\mathbb{P}}\rightarrow \iint_{Q_{T}\times \bar{%
\Omega}\times \mathcal{K}}\widehat{a}(Du_{0}+\partial \widehat{u}_{1})\cdot
(D\psi _{0}+\partial \widehat{\psi }_{1})dxdtd\bar{\mathbb{P}}d\beta .
\end{equation*}%
Next, owing to Remark \ref{r5.1}, assumption \textbf{A5} and the convergence
result (\ref{4.8}) we get 
\begin{equation*}
\int_{Q_{T}\times \bar{\Omega}}M^{\varepsilon }(u_{\varepsilon })\Phi
_{\varepsilon }dxdW^{\varepsilon }d\bar{\mathbb{P}}\rightarrow
\iint_{Q_{T}\times \bar{\Omega}\times \mathcal{K}}\widehat{M}%
(s,s_{0},u_{0})\psi _{0}dxd\bar{W}d\bar{\mathbb{P}}d\beta .
\end{equation*}%
Hence letting $\varepsilon \rightarrow 0$ in (\ref{5.4}) we end up with (\ref%
{5.1}), thereby completing the proof.
\end{proof}

The problem (\ref{5.1}) is called the \textit{global homogenized} problem
for (\ref{4.1}).

\subsection{Homogenized problem}

The problem (\ref{5.1}) is equivalent to the following system: 
\begin{equation}
\left\{ 
\begin{array}{l}
-\int_{Q_{T}\times \bar{\Omega}}\left[ \frac{\overline{\partial }u_{1}}{%
\partial \tau },\psi _{1}\right] dxdtd\bar{\mathbb{P}}-\iint_{Q_{T}\times 
\bar{\Omega}\times \mathcal{K}}\widehat{a}(Du_{0}+\partial \widehat{u}%
_{1})\cdot \partial \widehat{\psi }_{1}dxdtd\bar{\mathbb{P}}d\beta \\ 
+\iint_{Q_{T}\times \bar{\Omega}\times \mathcal{K}}\widehat{g}(s,s_{0},u_{0})%
\widehat{\psi }_{1}dxdtd\bar{\mathbb{P}}d\beta =0 \\ 
\text{for all }\psi _{1}\in B(\bar{\Omega})\otimes \mathcal{C}_{0}^{\infty
}(Q_{T})\otimes \mathcal{E},%
\end{array}%
\right.  \label{5.5}
\end{equation}%
and 
\begin{equation}
\left\{ 
\begin{array}{l}
-\int_{Q_{T}\times \bar{\Omega}}u_{0}\psi _{0}^{\prime }dxdtd\bar{\mathbb{P}}%
=-\iint_{Q_{T}\times \bar{\Omega}\times \mathcal{K}}\widehat{a}%
(Du_{0}+\partial \widehat{u}_{1})\cdot D\psi _{0}dxdtd\bar{\mathbb{P}}d\beta
\\ 
-\iint_{Q_{T}\times \bar{\Omega}\times \mathcal{K}}\widehat{G}%
(s,s_{0},u_{0})\cdot D\psi _{0}dxdtd\bar{\mathbb{P}}d\beta \\ 
-\iint_{Q_{T}\times \bar{\Omega}\times \mathcal{K}}\left( \widehat{\partial
_{u}G}(s,s_{0},u_{0})\cdot (Du_{0}+\partial \widehat{u}_{1})\right) \psi
_{0}dxdtd\bar{\mathbb{P}}d\beta \\ 
+\iint_{Q_{T}\times \bar{\Omega}\times \mathcal{K}}\widehat{M}%
(s,s_{0},u_{0})\psi _{0}d\bar{W}dxd\bar{\mathbb{P}}d\beta \text{\ \ for all }%
\psi _{0}\in B(\bar{\Omega})\otimes \mathcal{C}_{0}^{\infty }(Q_{T})\text{.}%
\end{array}%
\right.  \label{5.6}
\end{equation}%
The following uniqueness result is highlighted.

\begin{proposition}
\label{p5.2}The solution of the variational problem \emph{(\ref{5.5})} is
unique.
\end{proposition}

\begin{proof}
Taking in (\ref{5.5}) $\psi _{1}(x,t,y,\tau ,\omega )=\phi (\omega )\varphi
(x,t)w(y,\tau )$ with $\phi \in B(\bar{\Omega})$, $\varphi \in \mathcal{C}%
_{0}^{\infty }(Q_{T})$ and $w\in \mathcal{E}$, we obtain after mere
computations 
\begin{equation}
\begin{array}{l}
-\left[ \frac{\overline{\partial }u_{1}}{\partial \tau }(x,t,\omega ),w%
\right] -\int_{\mathcal{K}}\widehat{a}(Du_{0}(x,t,\omega )+\partial \widehat{%
u}_{1}(x,t,\omega ))\cdot \partial \widehat{w}d\beta \\ 
\ \ \ +\int_{\mathcal{K}}\widehat{g}(u_{0}(x,t,\omega ))\widehat{w}d\beta =0%
\text{ for all }w\in \mathcal{E}.%
\end{array}
\label{5.7}
\end{equation}%
So, fixing $(x,t,\omega )$, if $u_{1}=u_{1}(x,t,\omega )$ and $%
u_{2}=u_{2}(x,t,\omega )$ are two solutions to (\ref{5.7}), then $%
u=u_{1}-u_{2}$ is solution to 
\begin{equation}
\left[ \frac{\overline{\partial }u}{\partial \tau },w\right] =-\int_{%
\mathcal{K}}\widehat{a}\partial \widehat{u}\cdot \partial \widehat{w}d\beta 
\text{\ for all }w\in \mathcal{E}.  \label{5.8}
\end{equation}%
By the density of $\mathcal{E}$ in $\mathcal{W}$, (\ref{5.8}) still holds
for $w\in \mathcal{W}$. So taking there $w=u$ and using the fact that $%
\overline{\partial }/\partial \tau $ is skew adjoint (which yields $\left[ 
\overline{\partial }u/\partial \tau ,u\right] =0$) we get 
\begin{equation*}
\int_{\mathcal{K}}\widehat{a}\partial \widehat{u}\cdot \partial \widehat{u}%
d\beta =0.
\end{equation*}%
But, since 
\begin{equation*}
\int_{\mathcal{K}}\widehat{a}\partial \widehat{u}\cdot \partial \widehat{u}%
d\beta \geq \Lambda \left\Vert u\right\Vert _{\mathcal{B}_{AP}^{2}(\mathbb{R}%
_{\tau };\mathcal{B}_{\#AP}^{1,2}(\mathbb{R}_{y}^{N}))}^{2},
\end{equation*}%
we are led to $u=0$. Whence the uniqueness of the solution of (\ref{5.5}).
\end{proof}

Let us now deal with some auxiliary equations connected to (\ref{5.5}).

Let $\chi \in (\mathcal{W})^{N}$ and $w_{1}=w_{1}(\cdot ,\cdot ,r)$ (for
fixed $r\in \mathbb{R}$) be determined by the following variational
problems: 
\begin{equation}
\left[ \frac{\overline{\partial }\chi }{\partial \tau },\phi \right] =-\int_{%
\mathcal{K}}\widehat{a}\partial \widehat{\chi }\cdot \partial \widehat{\phi }%
d\beta -\int_{\mathcal{K}}\widehat{a}\cdot \partial \widehat{\phi }d\beta
\;\forall \phi \in \mathcal{W};  \label{5.10}
\end{equation}%
\begin{equation}
\begin{array}{l}
\left[ \frac{\overline{\partial }w_{1}}{\partial \tau },\phi \right] =-\int_{%
\mathcal{K}}\widehat{a}\partial \widehat{w}_{1}\cdot \partial \widehat{\phi }%
d\beta -\int_{\mathcal{K}}\widehat{G}(\cdot ,\cdot ,r)\cdot \partial 
\widehat{\phi }d\beta \\ 
\text{for all }\phi \in \mathcal{W}.%
\end{array}
\label{5.11}
\end{equation}%
Equations (\ref{5.10}) and (\ref{5.11}) are respectively equivalent to the
following equations: 
\begin{equation*}
\frac{\overline{\partial }\chi }{\partial \tau }-\overline{\text{div}}_{y}(a%
\overline{D}_{y}\chi )=\overline{\text{div}}_{y}a\text{ in }\mathcal{W}%
^{\prime },\;\chi \in (\mathcal{W})^{N},
\end{equation*}%
and 
\begin{equation*}
\frac{\overline{\partial }w_{1}}{\partial \tau }-\overline{\text{div}}_{y}(a%
\overline{D}_{y}w_{1})=g(\cdot ,\cdot ,r)\text{ in }\mathcal{W}^{\prime
},\;w_{1}\in \mathcal{W}.
\end{equation*}%
The existence of $\chi $ and $w_{1}(\cdot ,\cdot ,r)$ is ensured by a
classical result \cite{JLLions} since $\overline{\partial }/\partial \tau $
is a maximal monotone operator \cite{EfendievPankov} (see also \cite%
{EfendievPankov1} or \cite{Pankov}) and further the uniqueness of $\chi $
and $w_{1}(\cdot ,\cdot ,r)$ follows the same way of reasoning as in the
proof of Proposition \ref{p5.2}.

Now, taking $r=u_{0}(x,t,\omega )$ in (\ref{5.11}), it is easy to verify
that the function 
\begin{equation*}
(x,t,y,\tau ,\omega )\mapsto \chi (y,\tau )\cdot Du_{0}(x,t,\omega
)+w_{1}(y,\tau ,u_{0}(x,t,\omega ))
\end{equation*}%
solves Eq. (\ref{5.5}), so that, by the uniqueness of its solution, we are
led to 
\begin{equation}
u_{1}(x,t,y,\tau ,\omega )=\chi (y,\tau )\cdot Du_{0}(x,t,\omega
)+w_{1}(y,\tau ,u_{0}(x,t,\omega )).  \label{5.9}
\end{equation}%
For fixed $r\in \mathbb{R}$, and set as in \cite{AllPiat1} 
\begin{eqnarray*}
F_{1}(r) &=&\int_{\mathcal{K}}\widehat{a}\partial \widehat{w}%
_{1}(s,s_{0},r)d\beta \text{;\ }F_{2}(r)=\int_{\mathcal{K}}\widehat{\partial
_{u}g}(s,s_{0},r)\widehat{\chi }d\beta \\
F_{3}(r) &=&\int_{\mathcal{K}}\widehat{\partial _{u}g}(s,s_{0},r)\widehat{w}%
_{1}(s,s_{0},r)d\beta \text{;\ }\widetilde{M}(r)=\int_{\mathcal{K}}\widehat{M%
}(s,s_{0},r)d\beta .
\end{eqnarray*}

With this in mind, we have following

\begin{lemma}
\label{l5.3}The solution $u_{0}$ to the variational problem \emph{(\ref{5.6})%
} solves the following boundary value problem: 
\begin{equation}
\left\{ 
\begin{array}{l}
du_{0}=\left( \Div\left( bDu_{0}\right) +\Div F_{1}(u_{0})-F_{2}(u_{0})\cdot
Du_{0}-F_{3}(u_{0})\right) dt+\widetilde{M}(u_{0})d\bar{W}\text{\ \ in }Q_{T}
\\ 
u_{0}=0\text{\ \ on }\partial Q\times (0,T) \\ 
u_{0}(x,0)=u^{0}(x)\text{\ \ in }Q.%
\end{array}%
\right.  \label{5.14}
\end{equation}
\end{lemma}

\begin{proof}
We replace in Eq. (\ref{5.6}) $u_{1}$ by the expression (\ref{5.9}); we
therefore get 
\begin{equation*}
\left\{ 
\begin{array}{l}
-\int_{Q_{T}\times \bar{\Omega}}u_{0}\psi _{0}^{\prime }dxdtd\bar{\mathbb{P}}%
=-\iint_{Q_{T}\times \bar{\Omega}\times \mathcal{K}}\widehat{G}%
(s,s_{0},u_{0})\cdot D\psi _{0}dxdtd\bar{\mathbb{P}}d\beta \\ 
-\iint_{Q_{T}\times \bar{\Omega}\times \mathcal{K}}\widehat{a}%
(Du_{0}+\partial \widehat{\chi }\cdot Du_{0}+\partial \widehat{w}%
_{1}(s,s_{0},u_{0}))\cdot D\psi _{0}dxdtd\bar{\mathbb{P}}d\beta \\ 
-\iint_{Q_{T}\times \bar{\Omega}\times \mathcal{K}}\left( \widehat{\partial
_{u}G}(s,s_{0},u_{0})\cdot (Du_{0}+\partial \widehat{\chi }\cdot
Du_{0}+\partial \widehat{w}_{1}(s,s_{0},u_{0}))\right) \psi _{0}dxdtd\bar{%
\mathbb{P}}d\beta \\ 
+\iint_{Q_{T}\times \bar{\Omega}\times \mathcal{K}}\widehat{M}%
(s,s_{0},u_{0})\psi _{0}d\bar{W}dxd\bar{\mathbb{P}}d\beta \text{\ \ for all }%
\psi _{0}\in B(\bar{\Omega})\otimes \mathcal{C}_{0}^{\infty }(Q_{T})\text{.}%
\end{array}%
\right.
\end{equation*}%
In particular, for $\psi _{0}=\phi \otimes \varphi $ with $\phi \in B(\bar{%
\Omega})$ and $\varphi \in \mathcal{C}_{0}^{\infty }(Q_{T})$, we obtain 
\begin{equation}
\left\{ 
\begin{array}{l}
-\int_{Q_{T}}u_{0}\varphi ^{\prime }dxdt=-\iint_{Q_{T}\times \mathcal{K}}%
\widehat{a}([I+\partial \widehat{\chi }]\cdot Du_{0})\cdot D\varphi
dxdtd\beta \\ 
-\iint_{Q_{T}\times \mathcal{K}}\widehat{a}(\partial \widehat{w}%
_{1}(s,s_{0},u_{0})\cdot D\varphi dxdtd\beta -\iint_{Q_{T}\times \mathcal{K}}%
\widehat{G}(s,s_{0},u_{0})\cdot D\varphi dxdtd\beta \\ 
-\iint_{Q_{T}\times \mathcal{K}}\left( \widehat{\partial _{u}G}%
(s,s_{0},u_{0})\cdot (Du_{0}+\partial \widehat{\chi }\cdot Du_{0}+\partial 
\widehat{w}_{1}(s,s_{0},u_{0}))\right) \varphi dxdtd\beta \\ 
+\iint_{Q_{T}\times \mathcal{K}}\widehat{M}(s,s_{0},u_{0})\varphi d\bar{W}%
dxd\beta \text{\ \ for all }\varphi \in \mathcal{C}_{0}^{\infty }(Q_{T})%
\text{,}%
\end{array}%
\right.  \label{5.12}
\end{equation}%
where $I$ stands for the unit $N\times N$ matrix, and $\Div_{y}G(y,\tau
,u)=g(y,\tau ,u)$ as in Section 4. Let 
\begin{equation*}
b=\int_{\mathcal{K}}\widehat{a}(I+\partial \widehat{\chi })d\beta
\end{equation*}%
be the homogenized tensor. Since we have 
\begin{equation*}
-\iint_{Q_{T}\times \mathcal{K}}\widehat{G}(s,s_{0},u_{0})\cdot D\varphi
dxdtd\beta =\iint_{Q_{T}\times \mathcal{K}}(\widehat{\partial _{u}G}%
(s,s_{0},u_{0})\cdot Du_{0})\varphi dxdtd\beta ,
\end{equation*}%
\begin{eqnarray*}
&&-\iint_{Q_{T}\times \mathcal{K}}(\widehat{\partial _{u}G}%
(s,s_{0},u_{0})\cdot \partial \widehat{w}_{1}(s,s_{0},u_{0}))\varphi
dxdtd\beta \\
&=&\iint_{Q_{T}\times \mathcal{K}}\widehat{\partial _{u}g}(s,s_{0},u_{0})%
\widehat{w}_{1}(s,s_{0},u_{0})\varphi dxdtd\beta
\end{eqnarray*}%
and 
\begin{eqnarray*}
&&-\iint_{Q_{T}\times \mathcal{K}}(\widehat{\partial _{u}G}%
(s,s_{0},u_{0})\cdot (\partial \widehat{\chi }\cdot Du_{0}))\varphi
dxdtd\beta \\
&=&\iint_{Q_{T}\times \mathcal{K}}\widehat{\partial _{u}g}(s,s_{0},u_{0})(%
\widehat{\chi }\cdot Du_{0})\varphi dxdtd\beta ,
\end{eqnarray*}%
Eq. (\ref{5.12}) becomes 
\begin{equation}
\left\{ 
\begin{array}{l}
-\int_{Q_{T}}u_{0}\varphi ^{\prime }dxdt=-\int_{Q_{T}}(bDu_{0})\cdot
D\varphi dxdt \\ 
-\iint_{Q_{T}\times \mathcal{K}}\widehat{a}\partial \widehat{w}%
_{1}(s,s_{0},u_{0})\cdot D\varphi dxdtd\beta \\ 
-\iint_{Q_{T}\times \mathcal{K}}\widehat{\partial _{u}g}(s,s_{0},u_{0})(%
\widehat{\chi }\cdot Du_{0}+\widehat{w}_{1}(s,s_{0},u_{0}))\varphi dxdtd\beta
\\ 
+\iint_{Q_{T}\times \mathcal{K}}\widehat{M}(s,s_{0},u_{0})\varphi d\bar{W}%
dxd\beta \text{ for all }\varphi \in \mathcal{C}_{0}^{\infty }(Q_{T}),%
\end{array}%
\right.  \label{5.13}
\end{equation}%
which is the variational form of (\ref{5.14}).
\end{proof}

As in \cite{AllPiat1}, it can be checked straightforwardly that the
functions $F_{i}$ ($1\leq i\leq 3$) are Lipschitz continuous functions. As
in \cite{AllPiat1} again, we can show that $F_2(u)$ is uniformly bounded,
that is, there exists $C_{F_2}$ such that $|F_2(u)|\le C_{F_2}$ for any $%
u\in \mathbb{R}$. Likewise, following the same way of reasoning, it can also
be proved that the function $\widetilde{M}$ is Lipschitz continuous.

\begin{proposition}
\label{p5.3} Let $u_{0}$ and $u_{0}^{\#}$ be two solutions of \emph{(\ref%
{5.14})} on the same probabilistic system $(\bar{\Omega},\bar{\mathcal{F}},%
\bar{\mathbb{P}})$, $\bar{W}$, $\bar{\mathcal{F}^{t}}$ with the same initial
condition $u^{0}$. We have that $u_{0}=u_{0}^{\#}$ almost surely.
\end{proposition}

\begin{proof}
Let $w(t)=u_{0}(t)-u_{0}^{\#}(t)$. From It\^{o}'s formula it is easily seen
that $w$ satisfies: 
\begin{equation*}
\begin{split}
d|w(t)|^{2}& =-2(bDw(t),Dw(t))dt+2\Big[(F_{1}(u_{0}(t))-F_{1}(u_{0}^{%
\#}(t)),Dw) \\
& -(F_{2}(u_{0}(t)).Du_{0}(t)-F_{2}(u_{0}^{\#}(t)).Du_{0}^{\#}(t),w(t)) \\
& -(F_{3}(u_{0}(t))-F_{3}(u_{0}^{\#}(t)),w(t))+\frac{1}{2}|\widetilde{M}%
(u_{0}(t))-\widetilde{M}(u_{0}^{\#}(t))|^{2}\Big]dt \\
& +2(\widetilde{M}(u_{0}(t))-\widetilde{M}(u_{0}^{\#}(t)),w(t))d\bar{W}.
\end{split}%
\end{equation*}%
Let $\sigma (t)$ a differentiable function on $[0,T]$. Thanks again to It%
\^{o}'s formula we have that 
\begin{equation*}
d(\sigma (t)|w(t)|^{2})=\sigma ^{\prime }(t)|w(t)|^{2}dt+\sigma
(t)d|w(t)|^{2}.
\end{equation*}%
By using the lipschitzity of $F_{1}$, $F_{3}$, $\widetilde{M}$ and some
elementary inequalities we see that 
\begin{equation*}
\begin{split}
d(\sigma (t)|w(t)|^{2})& \leq \left( \sigma ^{\prime }(t)|w(t)|^{2}+\sigma
(t)\Big[-2(bDw(t),Dw(t))+\delta |Dw(t)|^{2}+C_{\delta }|w(t)|^{2}\Big]%
\right) dt \\
& +\left(
|F_{2}(u_{0}(t)).Du_{0}(t)|+|F_{2}(u_{0}^{\#})(t)).Du_{0}^{\#}(t)|\right)
\sigma (t)|w(t)|dt \\
& +C\sigma (t)|w(t)|^{2}dt+2\sigma (t)(\widetilde{M}(u_{0}(t))-\widetilde{M}%
(u_{0}^{\#}(t)),w(t))d\bar{W},
\end{split}%
\end{equation*}%
where $\delta >0$ is arbitrary. Integrating over $[0,t]$ and taking the
mathematical expectation yields 
\begin{equation*}
\begin{split}
\bar{\mathbb{E}}(\sigma (t)|w(t)|^{2})& \leq -2\bar{\mathbb{E}}%
\int_{0}^{t}\sigma (s)(bDw(s),Dw(s))ds+C\bar{\mathbb{E}}\int_{0}^{t}\sigma
(s)|w(s)|^{2}ds \\
& +\bar{\mathbb{E}}\int_{0}^{t}\left(
|F_{2}(u_{0}).Du_{0}|+|F_{2}(u_{0}^{\#}).Du_{0}^{\#}|\right) \sigma
(s)|w(s)|ds \\
& +\delta \bar{\mathbb{E}}\int_{0}^{t}\sigma (s)|Dw(s)|^{2}ds+\bar{\mathbb{E}%
}\int_{0}^{t}\sigma ^{\prime }(s)|w(s)|^{2}ds.
\end{split}%
\end{equation*}%
Choosing $\delta >0$ so that $\bar{\mathbb{E}}\int_{0}^{t}\sigma
(s)[(bDw,Dw)-\delta |Dw|^{2}]ds>0$, we infer from the last estimate that 
\begin{equation}
\begin{split}
\bar{\mathbb{E}}(\sigma (t)|w(t)|^{2})& \leq C\bar{\mathbb{E}}%
\int_{0}^{t}\sigma (s)|w(s)|^{2}ds+\bar{\mathbb{E}}\int_{0}^{t}\left(
|Du_{0}|+|Du_{0}^{\#}|\right) C_{F_{2}}\sigma (s)|w(s)|ds \\
& +\bar{\mathbb{E}}\int_{0}^{t}\sigma ^{\prime }(s)|w(s)|^{2}ds,
\end{split}
\label{5.15}
\end{equation}%
where we have used the fact that $F_{2}$ is uniformly bounded. By choosing 
\begin{equation*}
\sigma (t)=\exp \left( -\int_{0}^{t}\frac{\left(
|Du_{0}(s)|+|Du_{0}^{\#}(s)|\right) C_{F_{2}}}{|w(s)|}ds\right) ,
\end{equation*}%
we deduce from (\ref{5.15}) that 
\begin{equation*}
\bar{\mathbb{E}}(\sigma (t)|w(t)|^{2})\leq C\bar{\mathbb{E}}%
\int_{0}^{t}\sigma (s)|w(s)|^{2}ds,
\end{equation*}%
from which we derive by using Gronwall's lemma that $|u_{0}(t)-u_{0}^{\prime
}(t)|=0$ almost surely for any $t\in \lbrack 0,T]$. This completes the proof
of the pathwise uniqueness.
\end{proof}

\begin{remark}
\label{r6.1}\emph{The pathwise uniqueness result in Proposition \ref{p5.3}
and Yamada-Watanabe's Theorem (see, for instance, \cite{revuz}) implies the
existence of a unique strong probabilistic solution of (\ref{5.14}) on a
prescribed probabilistic system }$(\Omega ,\mathcal{F},\mathbb{P}),\mathcal{F%
}^{t},W$\emph{.}
\end{remark}

The aim of the rest of this section is to prove the following homogenization
result.

\begin{theorem}
\label{t5.1}Assume \textbf{A1}-\textbf{A5} hold. For each $\varepsilon >0$
let $u_{\varepsilon }$ be the unique solution of \emph{(\ref{1.1})} on a
given stochastic system $(\Omega ,\mathcal{F},\mathbb{P}),\mathcal{F}^{t},W$
defined as in Section \emph{4}. Then the whole sequence $u_{\varepsilon }$
converges in probability to $u_{0}$ as $\varepsilon \rightarrow 0$, in the
topology of $L^{2}(Q_{T})$ (i.e $||u_{\varepsilon }-u_{0}||_{L^{2}(Q_{T})}$
converges to zero in probability) where $u_{0}$ is the unique strong
probabilistic solution of \emph{(\ref{5.14})}.
\end{theorem}

The main ingredients for the proof of this theorem are the pathwise
uniqueness for (\ref{5.14}) and the following criteria for convergence in
probability whose proof can be found in \cite{GYONGY}.

\begin{lemma}
\label{l3.1} Let $X$ be a Polish space. A sequence of a X-valued random
variables $\{x_{n};n\geq 0\}$ converges in probability if and only if for
every subsequence of joint probability laws, $\{\nu _{n_{k},m_{k}};k\geq 0\}$%
, there exists a further subsequence which converges weakly to a probability
measure $\nu $ such that 
\begin{equation*}
\nu \left( \{(x,y)\in X\times X;x=y\}\right) =1.
\end{equation*}
\end{lemma}

Let us set $\mathfrak{S}^{L^{2}}=L^{2}(Q_{T})$, $\mathfrak{S}^{W}=\mathcal{C}%
(0,T:\mathbb{R}^{m})$, $\mathfrak{S}^{L^{2},L^{2}}=L^{2}(Q_{T})\times
L^{2}(Q_{T})$, and finally $\mathfrak{S}=L^{2}(Q_{T})\times
L^{2}(Q_{T})\times \mathfrak{S}^{W}$. For any $S\in \mathcal{B}(\mathfrak{S}%
^{L^{2}})$ we set $\Pi ^{\varepsilon }(S)=\mathbb{P}(u_{\varepsilon }\in S)$
and $\Pi ^{W}=\mathbb{P}(W\in S)$ for any $S\in \mathcal{B}(\mathfrak{S}%
^{W}) $. Next we define the joint probability laws : 
\begin{align*}
\Pi ^{\varepsilon ,\varepsilon ^{\prime }}& =\Pi ^{\varepsilon }\times \Pi
^{\varepsilon ^{\prime }} \\
\nu ^{\varepsilon ,\varepsilon ^{\prime }}& =\Pi ^{\varepsilon }\times \Pi
^{\varepsilon ^{\prime }}\times \Pi ^{W}.
\end{align*}%
The following tightness property holds.

\begin{lemma}
\label{l5.4}The collection $\{\nu ^{\varepsilon ,\varepsilon ^{\prime
}};\varepsilon ,\varepsilon ^{\prime }\in E\}$ (and hence any subsequence $%
\{\nu ^{\varepsilon _{j},\varepsilon _{j}^{\prime }}:\varepsilon
_{j},\varepsilon _{j}^{\prime }\in E^{\prime }\}$) is tight on $\mathfrak{S}$%
.
\end{lemma}

\begin{proof}
The proof is very similar to Theorem \ref{t4.2}. For any $\delta >0$ we
choose the sets $\Sigma _{\delta },Y_{\delta }$ exactly as in the proof of
Theorem \ref{t4.2} with appropriate modification on the constants $M_{\delta
},L_{\delta }$ so that $\Pi ^{\varepsilon }(Y_{\delta })\geq 1-\frac{\delta 
}{4}$ and $\Pi ^{W}(\Sigma _{\delta })\geq 1-\frac{\delta }{2}$ for every $%
\varepsilon \in E$. Now let us take $K_{\delta }=Y_{\delta }\times Y_{\delta
}\times \Sigma _{\delta }$ which is a compact in $\mathfrak{S}$; it is not
difficult to see that $\{\nu ^{\varepsilon ,\varepsilon ^{\prime
}}(K_{\delta })\geq (1-\frac{\delta }{4})^{2}(1-\frac{\delta }{2})\geq
1-\delta $ for all $\varepsilon ,\varepsilon ^{\prime }$. This completes the
proof of the lemma.
\end{proof}

\begin{proof}[Proof of Theorem \protect\ref{t5.1}]
Lemma \ref{l5.4} implies that there exists a subsequence from $\{\nu
^{\varepsilon _{j},\varepsilon _{j}^{\prime }}\}$ still denoted by $\{\nu
^{\varepsilon _{j},\varepsilon _{j}^{\prime }}\}$ which converges to a
probability measure $\nu $. By Skorokhod's theorem there exists a
probability space $(\bar{\Omega},\bar{\mathcal{F}},\bar{\mathbb{P}})$ on
which a sequence $(u_{\varepsilon _{j}},u_{\varepsilon _{j}^{\prime
}},W^{j}) $ is defined and converges almost surely in $\mathfrak{S}%
^{L^{2},L^{2}}\times \mathfrak{S}^{W}$ to a couple of random variables $%
(u_{0},v_{0},\bar{W})$. Furthermore, we have 
\begin{align*}
Law(u_{\varepsilon _{j}},u_{\varepsilon _{j}^{\prime }},W^{j})& =\nu
^{\varepsilon _{j},\varepsilon _{j}^{\prime }}, \\
Law(u_{0},v_{0},\bar{W})& =\nu .
\end{align*}%
Now let $Z_{j}^{u_{\varepsilon }}=(u_{\varepsilon _{j}},W^{j})$, $%
Z_{j}^{u_{\varepsilon ^{\prime }}}=(u_{\varepsilon _{j}^{\prime }},W^{j})$, $%
Z^{u_{0}}=(u_{0},\bar{W})$ and $Z^{v_{0}}=(v_{0},\bar{W})$. We can infer
from the above argument that $\left( \Pi ^{\varepsilon _{j},\varepsilon
_{j}^{\prime }}\right) $ converges to a measure $\Pi $ such that 
\begin{equation*}
\Pi (\cdot )=\bar{\mathbb{P}}((u_{0},v_{0})\in \cdot ).
\end{equation*}%
As above we can show that $Z_{j}^{u_{\varepsilon }}$ and $%
Z_{j}^{u_{\varepsilon ^{\prime }}}$ satisfy (\ref{4.10}) and that $Z^{u}$
and $Z^{v}$ satisfy (\ref{5.14}) on the same stochastic system $(\bar{\Omega}%
,\bar{\mathcal{F}},\bar{\mathbb{P}}),\bar{\mathcal{F}}^{t},\bar{W}$, where $%
\bar{\mathcal{F}^{t}}$ is the filtration generated by the couple $%
(u_{0},v_{0},\bar{W})$. Since we have the uniqueness result above, then we
see that $u^{0}=v^{0}$ almost surely and $u_{0}=v_{0}$ in $L^{2}(Q_{T})$.
Therefore 
\begin{equation*}
\Pi \left( \{(x,y)\in \mathfrak{S}^{L^{2},L^{2}};x=y\}\right) =\bar{\mathbb{P%
}}\left( u_{0}=v_{0}\text{ in }L^{2}(Q_{T})\right) =1.
\end{equation*}%
This fact together with Lemma \ref{l3.1} imply that the original sequence $%
\left( u_{\varepsilon }\right) $ defined on the original probability space $%
(\Omega ,\mathcal{F},\mathbb{P}),\mathcal{F}^{t},W$ converges in probability
to an element $u_{0}$ in the topology of $\mathfrak{S}^{L^{2}}$. By a
passage to the limit's argument as in the previous subsection it is not
difficult to show that $u_{0}$ is the unique solution of (\ref{5.14}) (on
the original probability system $(\Omega ,\mathcal{F},\mathbb{P}),\mathcal{F}%
^{t},W$). This ends the proof of Theorem \ref{t5.1}.
\end{proof}

\section{Some applications}

In this subsection we provide some applications of the results obtained in
the previous sections to some special cases.

\subsection{Example 1}

The first application is related to the periodicity hypothesis stated as
follows:

\begin{itemize}
\item[\textbf{A6}] $g(\cdot ,\cdot ,u)\in \mathcal{C}_{\text{per}}(Y\times
Z) $ for all $u\in \mathbb{R}$ with $\int_{Y}g(y,\tau ,u)dy=0$ for all $\tau
,u\in \mathbb{R}$; $a_{ij},M_{i}(\cdot ,\cdot ,u)\in L_{\text{per}}^{\infty
}(Y\times Z)$ for all $1\leq i,j\leq N$; $M_{i}(\cdot ,\cdot ,u)\in L_{\text{%
per}}^{\infty }(Y\times Z)$ for each $1\leq i\leq m$ and for all $u\in 
\mathbb{R}$,

\noindent where $Y=(0,1)^{N}$ and $Z=(0,1)$ and, $\mathcal{C}_{\text{per}%
}(Y\times Z)$ and $L_{\text{per}}^{\infty }(Y\times Z)$ denote the usual
spaces of $Y\times Z$-periodic functions.
\end{itemize}

As the periodic functions are part of almost periodic functions, all the
results of the previous sections apply to this case. We have the following
result.

\begin{theorem}
\label{t5.2}Assume hypotheses \textbf{A1}-\textbf{A5} are satisfied with the
almost periodicity therein being replaced by the periodicity hypothesis 
\textbf{A6}. For each $\varepsilon >0$ let $u_{\varepsilon }$ be uniquely
determined by \emph{(\ref{1.1})}. Then as $\varepsilon \rightarrow 0$, 
\begin{equation*}
u_{\varepsilon }\rightarrow u_{0}\text{\ in }L^{2}(Q\times (0,T))\text{
almost surely}
\end{equation*}%
and 
\begin{equation*}
\frac{\partial u_{\varepsilon }}{\partial x_{j}}\rightarrow \frac{\partial
u_{0}}{\partial x_{j}}+\frac{\partial u_{1}}{\partial y_{j}}\text{\ in }%
L^{2}(Q\times (0,T)\times \bar{\Omega})\text{-weak }\Sigma \text{ \ }(1\leq
j\leq N)
\end{equation*}%
where $(u_{0},u_{1})\in L^{2}(\bar{\Omega}\times \left( 0,T\right)
;H_{0}^{1}(Q))\times L^{2}(Q_{T}\times \bar{\Omega};\mathcal{W})$ is the
unique solution to the variational problem 
\begin{equation*}
\left\{ 
\begin{array}{l}
-\int_{Q_{T}\times \bar{\Omega}}u_{0}\psi _{0}^{\prime }dxdtd\bar{\mathbb{P}}%
+\int_{Q_{T}\times \bar{\Omega}}\left[ \frac{\partial u_{1}}{\partial \tau }%
,\psi _{1}\right] dxdtd\bar{\mathbb{P}} \\ 
=-\iint_{Q_{T}\times \bar{\Omega}\times Y\times Z}a(Du_{0}+D_{y}u_{1})\cdot
(D\psi _{0}+D_{y}\psi _{1})dxdtd\bar{\mathbb{P}}dyd\tau \\ 
+\iint_{Q_{T}\times \bar{\Omega}\times Y\times Z}g(y,\tau ,u_{0})\psi
_{1}dxdtd\bar{\mathbb{P}}dyd\tau \\ 
-\iint_{Q_{T}\times \bar{\Omega}\times Y\times Z}G(y,\tau ,u_{0})\cdot D\psi
_{0}dxdtd\bar{\mathbb{P}}dyd\tau \\ 
-\iint_{Q_{T}\times \bar{\Omega}\times Y\times Z}\left( \partial
_{u}G(y,\tau ,u_{0})\cdot (Du_{0}+D_{y}u_{1})\right) \psi _{0}dxdtd\bar{%
\mathbb{P}}dyd\tau \\ 
+\iint_{Q_{T}\times \bar{\Omega}\times Y\times Z}M(y,\tau ,u_{0})\psi _{0}d%
\bar{W}dxd\bar{\mathbb{P}}dyd\tau \text{\ \ for all }(\psi _{0},\psi
_{1})\in \mathcal{F}_{0}^{\infty }%
\end{array}%
\right.
\end{equation*}%
where $\mathcal{W}=\{v\in L_{\text{\emph{per}}}^{2}(Z;W_{\#}^{1,2}(Y)):%
\partial v/\partial \tau \in L_{\text{\emph{per}}}^{2}(Z;[W_{\#}^{1,2}(Y)]^{%
\prime })\}$ with $W_{\#}^{1,2}(Y)=\{u\in W_{\text{\emph{per}}%
}^{1,2}(Y):\int_{Y}u(y)dy=0\}$, and $\mathcal{F}_{0}^{\infty }=[B(\bar{\Omega%
})\otimes \mathcal{C}_{0}^{\infty }(Q_{T})]\times \lbrack B(\bar{\Omega}%
)\otimes \mathcal{C}_{0}^{\infty }(Q_{T})\otimes \mathcal{E}$ with $\mathcal{%
E}=\mathcal{C}_{\text{\emph{per}}}^{\infty }(Z)\otimes \mathcal{C}%
_{\#}^{\infty }(Y)$ and $\mathcal{C}_{\#}^{\infty }(Y)=\{u\in \mathcal{C}_{%
\text{\emph{per}}}^{\infty }(Y):\int_{Y}u(y)dy=0\}$.
\end{theorem}

\begin{proof}
Theorem \ref{t5.2} is a consequence of the following facts: (1) in the
periodic setting, the mean value of a function $u\in L_{\text{per}%
}^{p}(Y)=\{u\in L_{\text{loc}}^{p}(\mathbb{R}_{y}^{N}):u$ is $Y$-periodic$\}$
is merely expressed as $M(u)=\int_{Y}u(y)dy$ (the same definition for the
other spaces); (2) the Besicovitch space corresponding to the periodic
functions is exactly the space $L_{\text{per}}^{p}(Y)$; (3) the derivative $%
\overline{\partial }/\partial y_{i}$ (resp. $\overline{\partial }/\partial
\tau )$ is therefore exactly the usual one in the distribution sense $%
\partial /\partial y_{i}$ (resp. $\partial /\partial \tau )$.
\end{proof}

\begin{remark}
\label{r5.2}\emph{The above result extends to the case of stochastic partial
differential equations the result obtained by Allaire and Piatnitski \cite%
{AllPiat1} in the periodic deterministic setting.}
\end{remark}

\subsection{Example 2}

Our purpose in the present example is to study the homogenization problem
for (\ref{1.1}) under the following assumptions, where the indices $1\leq
i,j\leq N$ and $1\leq l\leq m$ are arbitrarily fixed:

\begin{itemize}
\item[(HYP)$_{1}$] $a_{ij}(\cdot ,\tau )\in B_{AP}^{2}(\mathbb{R}_{y}^{N})$
a.e. in $\tau \in \mathbb{R}$.

\item[(HYP)$_{2}$] The function $\tau \mapsto a_{ij}(\cdot ,\tau )$ from $%
\mathbb{R}$ to $B_{AP}^{2}(\mathbb{R}_{y}^{N})$ is piecewise constant in the
sense that there exists a mapping $q_{ij}:\mathbb{Z}\rightarrow B_{AP}^{2}(%
\mathbb{R}_{y}^{N})$ such that 
\begin{equation*}
a_{ij}(\cdot ,\tau )=q_{ij}(k)\text{\ a.e. in }k\leq \tau <k+1\text{ (}k\in 
\mathbb{Z}\text{).}
\end{equation*}%
We assume further that $q_{ij}\in \mathcal{C}_{\text{per}}(\mathbb{Z}%
;B_{AP}^{2}(\mathbb{R}_{y}^{N}))$.

\item[(HYP)$_{3}$] The functions $g(\cdot ,\cdot ,u)\in AP(\mathbb{R}%
_{y,\tau }^{N+1})$ with $M_{y}(g(\cdot ,\cdot ,u))=0$, and $M_{l}(\cdot
,\cdot ,u)\in \mathcal{C}_{\text{per}}(Y\times Z)$ for all $u\in \mathbb{R}$.
\end{itemize}

Then arguing as in \cite{NgWouEJDE} we are led to the homogenization of (\ref%
{1.1}) with in \textbf{A3}-\textbf{A5} the almost periodicity replaced by
(HYP)$_{1}$-(HYP)$_{3}$ above. Indeed the above assumptions lead to the
almost periodicity of the involved functions with respect to $y$ and $\tau $.

\subsection{Example 3}

Our concern here is the study of the homogenization of (\ref{1.1}) under the
following assumptions, the indices $1\leq i,j\leq N$ and $1\leq l\leq m$
being arbitrarily fixed:

\begin{itemize}
\item[(1)] The function $\tau \mapsto a_{ij}(\cdot ,\tau )$ maps
continuously $\mathbb{R}$ into $L_{\text{loc}}^{2}(\mathbb{R}_{y}^{N})$ and
is $Z$-periodic ($Z=(0,1)$).

\item[(2)] For each fixed $\tau \in \mathbb{R}$, the function $a_{ij}(\cdot
,\tau )$ is $Y_{\tau }$-periodic, where $Y_{\tau }=(0,c_{\tau })^{N}$ with $%
c_{\tau }>0$.

\item[(3)] $g(\cdot ,\cdot ,u)\in \mathcal{C}_{\text{per}}(Y\times Z)$ with $%
\int_{Y}g(y,\tau ,u)dy=0$ for all $\tau ,u\in \mathbb{R}$, and $M_{l}(\cdot
,\cdot ,u)\in B_{AP}^{2}(\mathbb{R}_{y,\tau }^{N+1})$ for all $u\in \mathbb{R%
}$.
\end{itemize}

Hypothesis (1) and (2) imply that $a_{ij}\in \mathcal{C}_{\text{per}%
}(Z;B_{AP}^{2}(\mathbb{R}_{y}^{N}))\subset B_{AP}^{2}(\mathbb{R}_{y,\tau
}^{N+1})$, such that the homogenization of (\ref{1.1}) under the above
hypotheses is solvable.

\end{document}